\documentclass[12pt,oneside,english]{amsart}
\usepackage[T1]{fontenc}
\usepackage[latin9]{inputenc}
\usepackage{geometry}
\usepackage{amstext}
\usepackage{amsthm}
\usepackage{amssymb}
\usepackage{wasysym}

\makeatletter
\numberwithin{equation}{section}
\numberwithin{figure}{section}
\theoremstyle{plain}
\newtheorem{thm}{\protect\theoremname}[section]
\theoremstyle{remark}
\newtheorem{rem}[thm]{\protect\remarkname}
\theoremstyle{plain}
\newtheorem{cor}[thm]{\protect\corollaryname}
\theoremstyle{plain}
\newtheorem{lem}[thm]{\protect\lemmaname}
\theoremstyle{definition}
\newtheorem{defn}[thm]{\protect\definitionname}
\theoremstyle{plain}
\newtheorem{prop}[thm]{\protect\propositionname}

\makeatother

\usepackage{babel}
\usepackage[bibstyle=alphabetic,citestyle=alphabetic,doi=false,isbn=false,url=false]{biblatex}
\providecommand{\corollaryname}{Corollary}
\providecommand{\definitionname}{Definition}
\providecommand{\lemmaname}{Lemma}
\providecommand{\propositionname}{Proposition}
\providecommand{\remarkname}{Remark}
\providecommand{\theoremname}{Theorem}

\begin{document}
\title{Ricci Flow with Ricci Curvature and Volume Bounded Below}
\author{Max Hallgren}
\begin{abstract}
We show that a simply-connected closed four-dimensional Ricci flow
whose Ricci curvature is uniformly bounded below and whose volume
does not approach zero must converge to a $C^{0}$ orbifold at any
finite-time singularity, so has an extension through the singularity
via orbifold Ricci flow. Moreover, a Type-I blowup of the flow based
at any orbifold point converges to a flat cone in the Gromov-Hausdorff
sense, without passing to a subsequence. In addition, we prove $L^{p}$
bounds for the curvature tensor on time-slices for any $p<2$. In
higher dimensions, we show that every singular point of the flow is
a Type-II point, and that any tangent flow at a singular point is
a static flow corresponding to a Ricci flat cone. 
\end{abstract}

\maketitle

\section{Introduction}

In this paper, we study closed solutions $(M^{n},(g_{t})_{t\in[0,T)})$
of Ricci flow satisfying the following assumptions for all $t\in[0,T)$:

\begin{equation}
Rc(g_{t})\geq-Ag_{t},\label{eq:ric}
\end{equation}
\begin{equation}
|M|_{g_{t}}\geq A^{-1},\label{eq:vol}
\end{equation}
where $A<\infty$ is constant. One motiviation for studying such flows
is to get a clearer picture of how the curvature of a Ricci flow fails
to be controlled near finite-time singularities. N. {\v S}e{\v s}um showed in
\cite{sesumricci} that any Ricci flow satisfying a two-sided curvature
bound
\[
-Ag_{t}\leq Rc(g_{t})\leq Ag_{t}
\]
cannot develop a finite-time singularity. B. Wang showed in \cite{wang1}
that a Ricci flow satisfying (\ref{eq:ric}) as well as the spacetime
integral estimate for scalar curvature
\[
\int_{0}^{T}\int_{M}R^{\frac{n+2}{2}}dg_{t}dt<\infty
\]
cannot develop a singularity either. Even when $n=4$, it is still
an open problem whether a finite-time singularity can occur for a
Ricci flow with bounded scalar curvature:
\begin{equation}
\sup_{M\times[0,T)}R<\infty,\label{eq:scal}
\end{equation}
(in which case assumption (\ref{eq:vol}) holds automatically) though
considerable progress has been made \cite{BZ1,BZ2,bam2,chenwang1,chenwang2a,chenwang2b,simonscalar}.
In particular, Bamler-Zhang \cite{BZ1} proved, and it was shown independently
by M. Simon \cite{simonscalar}, that in four dimensions, a Ricci
flow satisfying (\ref{eq:scal}) must converge in the Gromov-Hausdorff
sense as $t\nearrow T$ to a $C^{0}$ Riemannian orbifold, with smooth
convergence away from the orbifold points. Simon also showed that
such a flow can be continued through the singularity via orbifold
Ricci flow (Theorem 9.1 of \cite{simonscalar}). 

Ricci flow solutions satisfying (\ref{eq:ric}),(\ref{eq:vol}) were
studied by X. Chen and F. Yuan in \cite{chenlowerbound}, where they
asked whether such a flow can develop a singularity in finite time.
They found that this cannot occur in dimension three. This had previously
been shown by Z. Zhang (Theorem 1.1 of \cite{zhanglowerricci}) in
all dimensions if we assume $g_{0}$ is K{\"a}hler. Assumption (\ref{eq:ric})
by itself is clearly insufficient to rule out singularities (the round
shrinking sphere is a counterexample), and in general the behavior
of Ricci flows satisfying (\ref{eq:ric}) and $\lim_{t\to T}|M|_{g_{t}}=0$
is more complicated than those satisfying (\ref{eq:ric}),(\ref{eq:vol}). 

Another motivation for considering conditions (\ref{eq:ric}),(\ref{eq:vol})
is the extensive compactness and partial regularity theory developed
by Cheeger-Colding-Naber-Tian-Jiang \cite{coldingvolume,cheegercoldingwarped,cheegercolding1,cheegercolding2,cheegercoldingtian,cheegernaberquant,cheegernaberdim4,rectifiability}
for sequences of Riemannian manifolds $(M_{i},g_{i})$ satisfying
assumptions

\begin{equation}
Rc(g_{i})\geq-Ag_{i},\label{eq:ccnric}
\end{equation}
\begin{equation}
|M_{i}|_{g_{i}}\geq A^{-1},\label{eq:ccnvol}
\end{equation}
\begin{equation}
\text{diam}_{g_{i}}(M_{i})\leq A.\label{eq:ccndiam}
\end{equation}
We refer to Gromov-Hausdorff limits of such $(M_{i},g_{i})$ (or pointed
Gromov-Hausdorff limits if we drop assumption (\ref{eq:ccndiam}))
as noncollapsed Ricci limit spaces. The aforementioned works showed
that any sequence satisfying (\ref{eq:ccnric}),(\ref{eq:ccnvol}),(\ref{eq:ccndiam})
must subsequentially converge to a compact metric length space $(X,d)$
which is bi-H\"older homeomorphic to a smooth Riemannian manifold on
a subset $\mathcal{R}$ whose complement has Hausdorff codimension
two \cite{rectifiability}. Also, if $dg_{i}$ is the Riemannian volume
measure of $(M_{i},g_{i})$, then 
\[
(M_{i},g_{i},dg_{i})\to(X,d,\mathcal{H}^{n})
\]
in the measured Gromov-Hausdorff sense, where $\mathcal{H}^{n}$ is
the n-dimensional Hausdorff measure of $X$ (this is known as Colding's
Volume Convergence Theorem \cite{coldingvolume}). Theorem 1 of \cite{chenlowerbound}
states that if $(M_{i},g_{i})=(M,g_{t_{i}})$ for some sequence of
times $t_{i}\nearrow T$, where $(M,(g_{t})_{t\in[0,T)})$ is a closed
Ricci flow satisfying (\ref{eq:ric}),(\ref{eq:vol}), then $\mathcal{R}$
is open and actually has the structure of a smooth Riemannian manifold.
The intuition espoused in \cite{chenlowerbound} is that the Ricci
flow equation should imply that in some respects, $(M,g_{t_{i}})_{i\in\mathbb{N}}$
behaves like a sequence of Riemannian manifolds satisfying (\ref{eq:ccnvol}),(\ref{eq:ccndiam}),
and the two-sided Ricci bound
\begin{equation}
-Ag_{i}\leq Rc(g_{i})\leq Ag_{i},\label{eq:2sidedric}
\end{equation}
whose limit spaces have singularities of codimension four \cite{cheegernaberdim4}.
Some evidence for this idea is that limits of $(M,g_{t_{i}})_{i\in\mathbb{N}}$
must satisfy Anderson's $\epsilon$-regularity theorem for the volume
ratio, which usually only holds for Ricci limit spaces corresponding
to sequences satisfying (\ref{eq:2sidedric}) (Theorem 3.2 and Remark
3.3 of \cite{anderson}).

In this paper, we address the question posed in \cite{chenlowerbound},
with our strongest results holding only in the special case of dimension
four. The structure of the limit space in four dimensions is summarized
in the following theorem.
\begin{thm}
\label{thm:theorem1} Suppose $(M^{4},(g_{t})_{t\in[0,T)})$ is a
simply connected four-dimensional closed Ricci flow satisfying $Rc(g_{t})\geq-Ag_{t}$
and $\inf_{t\in[0,T)}|M|_{g_{t}}\geq A^{-1}>0$ for some constant
$A<\infty$. Then $(M,d_{g_{t}})$ converge in the Gromov-Hausdorff
sense as $t\nearrow T$ to a Ricci limit space $(X,d)$ with the structure
of a $C^{0}$ Riemannian orbifold with finite singular set. Moreover,
the convergence is smooth away from the singular points of $X$, and
each singular point $\overline{x}\in X$ has unique tangent cone equal
to $C(\mathbb{S}^{3}/\Gamma_{\overline{x}})$ for some finite subgroup
$\Gamma_{\overline{x}}\leq O(4,\mathbb{R})$. 
\end{thm}

\begin{rem}
The assumption of simple connectedness is only used to imply orientability
and the generalized Jordan-Brouwer Separation Theorem, so may be replaced
with the weaker condition $H_{1}(M,\mathbb{Z}/2\mathbb{Z})=\{0\}$.
Even this condition may not be necessary, though it is convenient
for our proof of Theorem \ref{thm:theorem1}.
\end{rem}

In addition, $X=M/\sim$ is a topological quotient of the original
space $M$ (see Section 3, and also Corollary 1.3 of \cite{BZ2}).
The orbifold structure of $X$ mirrors the corresponding result for
four-dimensional Gromov-Hausdorff limits of sequences of pointed Riemannian
manifolds satisfying (\ref{eq:ccnvol}),(\ref{eq:ccndiam}),(\ref{eq:2sidedric}).
In fact, \cite{nakajimaale,cheegernaberdim4} imply that such limits
must be $C^{0}$ Riemannian orbifolds. We also note that the same
result was proved for time-slices of Ricci flows by \cite{BZ1,simonscalar}
with the hypotheses (\ref{eq:ric}),(\ref{eq:vol}) replaced by (\ref{eq:scal}).
Using Theorem \ref{thm:theorem1}, we can apply Theorem 9.1 of \cite{simonscalar}
to conclude that there is a flow through the singularity at time $T$.
\begin{cor}
With the same hypotheses of Theorem \ref{thm:theorem1}, there exists
$\delta>0$ and a solution $(\widetilde{M}^{4},(\widetilde{g}_{t})_{t\in[T,T+\delta)})$
of the orbifold Ricci flow such that 
\[
\lim_{t\searrow T}d_{GH}\left((\widetilde{M},d_{\widetilde{g}_{t}}),(X,d)\right)=0.
\]
\end{cor}

One major difference between our setting and that of Ricci flows satisfying
(\ref{eq:scal}) is that in \cite{BZ1}, it is shown that the 'deepest
bubbles' of $(M,(g_{t})_{t\in[0,T)})$ satisfying (\ref{eq:scal})
(that is, the dilation limits arising from rescaling so that the maximum
value of $|Rm|_{g_{t}}$ is 1) are Ricci-flat ALE spaces (Corollary
1.9 of \cite{BZ1}), so that orbifold singularities are precluded
by a simple topological condition (Corollary 1.10 of \cite{BZ1}).
We are unable to prove such results in our setting at present, but
we are still able to give a fairly complete description of the behavior
of the flow at the Type-I curvature scale. Both the statement and
the proof of this behavior rely heavily on the recent convergence
and partial regularity theory for Ricci flows produced by R. Bamler
\cite{bamlergen1,bamlergen2,bamlergen3}, which can be seen as the
parabolic analogue of Cheeger-Colding-Naber-Tian-Jiang's theory. In
\cite{bamlergen2}, Bamler defines the parabolic notion of a metric
space, called a metric flow (Definition 3.2 of \cite{bamlergen2}),
and the parabolic analogue of pointed Gromov-Hausdorff convergence
of pointed metric spaces, termed $\mathbb{F}$-convergence of metric
flow pairs (Definition 5.8 \cite{bamlergen2}). Roughly speaking,
he shows that any sequence of Ricci flows $(M_{i},(g_{t}^{i})_{t\in I^{i}},(\nu_{x_{i},t_{i};t}^{i})_{t\in I^{i}\cap(-\infty,t_{i})})$,
where $\nu_{x_{i},t_{i};t}^{i}=K^{i}(x_{i},t_{i};\cdot,t)dg_{t}$,
and $K^{i}(x_{i},t_{i};\cdot,\cdot)$ is the conjugate heat kernel
of $(M_{i},(g_{t}^{i})_{t\in I^{i}})$ based at $(x_{i},t_{i})$,
$\mathbb{F}$-converge (after passing to a subsequence) to some metric
flow, which is a smooth Ricci flow spacetime (see Definition 9.1 of
\cite{bamlergen2}) outside of a set of parabolic codimension four.
If $(M_{i},(g_{t}^{i})_{t\in I^{i}})$ are Type-I rescalings of any
fixed Ricci flow with fixed basepoint, Bamler proves that any $\mathbb{F}$-limit
(called a tangent flow) must be a singular shrinking gradient Ricci
soliton with singularities of Minkowski codimension four (Theorem
1.19 of \cite{bamlergen3}). We show that under assumptions (\ref{eq:ric}),(\ref{eq:vol}),
such a tangent flow must be a static flow modeled on $C(\mathbb{S}^{3}/\Gamma)$
for some finite subgroup $\Gamma\leq O(4,\mathbb{R})$, and show that
convergence also occurs in the Gromov-Hausdorff sense on each time
slice. 
\begin{thm}
\label{thm:theorem2} Given the notation and hypotheses of Theorem
\ref{thm:theorem1}, let $x\in M$ correspond to a singular point
$\overline{x}\in X=M/\sim$, and let $C(\mathbb{S}^{3}/\Gamma_{\overline{x}})$
be the corresponding tangent cone as in Theorem \ref{thm:theorem1}. 

$(i)$ Let $(\nu_{x,T;t})_{t\in(0,T)}$ be a conjugate heat kernel
at the singular time based at $x$ (see the discussion at the beginning
of Section 3). Then every corresponding tangent flow (c.f. Theorem
1.38 of \cite{bamlergen3}) is a static metric flow corresponding
to the singular space $C(\mathbb{S}^{3}/\Gamma_{\overline{x}})$.

$(ii)$ If $o_{\ast}$ is the vertex of $C(\mathbb{S}^{3}/\Gamma_{\overline{x}})$,
then 
\[
(M,(T-t)^{-\frac{1}{2}}d_{g_{t}},x)\to(C(\mathbb{S}^{3}/\Gamma_{\overline{x}}),d_{C(\mathbb{S}^{3}/\Gamma_{\overline{x}})},o_{\ast})
\]
in the pointed Gromov-Hausdorff sense as $t\nearrow T$.
\end{thm}

It was shown in \cite{cheegernaberquant} that any Riemannian manifold
$(M,g)$ satisfying (\ref{eq:ccnvol}),(\ref{eq:ccndiam}),(\ref{eq:2sidedric})
also satisfies an estimate of the form 
\[
\int_{M}r_{h}^{-p}(x)dg\leq C(A,p)
\]
for any $p\in(0,4)$, where $r_{h}$ is the $C^{1,\alpha}$ harmonic
radius. If in addition $(M,g)$ is Einstein, then $r_{h}$ can be
replaced by the curvature scale $\widetilde{r}_{Rm}$ (see Section
2 for the definition). Our next theorem states that such an estimate
also holds for the time slices of a Ricci flow satisfying (\ref{eq:ric}),(\ref{eq:vol})
when $n=4$.
\begin{thm}
\label{thm:theorem3} Given the hypotheses of Theorem \ref{thm:theorem1},
we have 
\[
\sup_{t\in[0,T)}\int_{M}|Rm|^{\frac{p}{2}}(x,t)dg_{t}(x)\leq\sup_{t\in[0,T)}\int_{M}\widetilde{r}_{Rm}^{-p}(x,t)dg_{t}(x)<\infty
\]
for any $p\in(0,4)$. Moreover, there exists $E<\infty$ such that
for all $s\in(0,1]$ and $t\in[0,T)$, we have 
\[
|\{\widetilde{r}_{Rm}(\cdot,t)<s\}|_{g_{t}}\leq Es^{4}.
\]
\end{thm}

For $p\in[2,4)$, we are unable to get a bound which is uniform in
the parameter $A$, even if we only take the supremum over $t\in[\frac{T}{2},T)$
and include an upper bound on diameter or only integrate over geodesic
balls of fixed radius. It is unclear whether or not the estimate of
Theorem \ref{thm:theorem3} can be made uniform in the parameters
$A,T,p$. Another remaining question is whether this estimate holds
when $p=4$.

We now consider the higher-dimensional case.
\begin{thm}
\label{thm:theoremhigherdim} Let $(M^{n},(g_{t})_{t\in[0,T)})$ be
a closed Ricci flow satisfying (\ref{eq:ric}),(\ref{eq:vol}), and
let $(\nu_{x,T;t})_{t\in(0,T)}$ be a conjugate heat kernel based
at the singular time at $x\in M$, corresponding to a singular point
$\overline{x}$ in the Gromov-Hausdorff limit $\lim_{t\nearrow T}(M,d_{g_{t}})$.
Then every tangent flow at $\nu_{x,T}$ is a static metric flow modeled
on a metric cone $(C(Z),o_{\ast})$ which is Ricci-flat on its regular
part, but $C(Z)\neq\mathbb{R}^{n}$. Moreover, if $(g_{t}^{i})_{t\in[-\tau_{i}^{-1}T,0)}$
is a sequence of Type-I rescalings $\mathbb{F}$-converging to a given
tangent flow, then we can pass to a subsequence so that for every
$t\in(-\infty,0)$, there exist $x_{t}^{i}\in M$ such that
\[
(M,d_{g_{t}^{i}},x_{t}^{i})\to(C(Z),d_{C(Z)},o_{\ast})
\]
in the pointed Gromov-Hausdorff sense as $i\to\infty$.
\end{thm}

We also prove a local version of the fact that a Ricci flow satisfying
(\ref{eq:ric}),(\ref{eq:vol}) cannot develop a Type-I singularity.
\begin{thm}
\label{thm:theorem4} Let $(M^{n},(g_{t})_{t\in[0,T)})$ be a closed
Ricci flow satisfying $(\ref{eq:ric}),(\ref{eq:vol})$. Again let
$(X,d)=\lim_{t\nearrow T}(M,d_{g_{t}})$ be the corresponding Ricci
limit space. Then, for any $x\in M$ corresponding to a singular point
$\overline{x}\in X$, we have
\[
\limsup_{t\nearrow T}(T-t)r_{Rm}^{-2}(x,t)=\infty.
\]
\end{thm}

Observe that this implies there is a sequence $(x_{i},t_{i})\in M\times[0,T)$
with $t_{i}\nearrow T$, $d_{g_{t_{i}}}(x_{i},x)=o(\sqrt{T-t_{i}})$
and 
\[
\lim_{i\to\infty}|Rm|(x_{i},t_{i})(T-t_{i})=\infty.
\]
In the notation and terminology of Definition 1.1 of \cite{buzanodimatteo},
this means that every singular point $x\in M$ is a Type-II point,
so $\Sigma=\Sigma_{II}$.

Finally, though we are so far unable to prove $L^{p}$ curvature estimates
for $p\in[1,2)$ in higher dimensions, we do get $L^{p}$ bounds on
the (time-dependent) set where the curvature scale is smaller than
$\sqrt{T-t}$. This is made precise in the following Theorem. 
\begin{thm}
\label{thm:highdimlp} For any $A<\infty$, $\underline{T}>0$, and
$p\in(0,4)$, there exist $r_{0}=r_{0}(p,A,\underline{T})>0$, $E=E(A,\underline{T})<\infty$,
such that the following hold. Let $(M^{n},(g_{t})_{t\in[0,T)})$ be
a closed Ricci flow satisfying $Rc(g_{t})\geq-Ag_{t}$ and $|B(x,t,r)|_{g_{t}}\geq A^{-1}r^{n}$
for all $(x,t)\in M\times[0,T)$ and $r\in(0,1]$, where $T\geq\underline{T}$. 

$(i)$ For any $(x,t)\in M\times[\frac{T}{2},T)$, $r\in(0,r_{0}\sqrt{T-t}]$,
and $s\in(0,1]$, we have 
\[
|\{\widetilde{r}_{Rm}(\cdot,t)<sr\}\cap B(x,t,r)|_{g_{t}}\leq Es^{4}r^{n}.
\]

$(ii)$ For any $\alpha\in[1,\infty)$, there exists $C_{\alpha}:=C_{\alpha}(p,A,\underline{T})<\infty$
such that
\[
\int_{B(x,t,\alpha\sqrt{T-t})}|Rm|^{\frac{p}{2}}dg_{t}\leq\int_{B(x,t,\alpha\sqrt{T-t})}\widetilde{r}_{Rm}^{-p}(\cdot,t)dg_{t}\leq C_{\alpha}(T-t)^{\frac{n-p}{2}}.
\]
for all $(x,t)\in M\times[\frac{T}{2},T)$. 
\end{thm}

We now give an outline of the paper and the proofs of the main theorems.

In Section 2, we recall and develop some necessary facts relevant
to Ricci flow and Ricci limit spaces.

In Section 3, we prove Theorems \ref{thm:theoremhigherdim},\ref{thm:theorem4},
and part $(i)$ of Theorem \ref{thm:theorem2}. The rough idea of
the proof comes from observing that any singularity model which is
a shrinking gradient Ricci soliton (GRS) with nonnegative Ricci curvature
and maximal volume growth must be flat (Corollary 1.1 in \cite{carilloni}).
If the flow $(M,(g_{t})_{t\in[0,T)})$ were Type-I, then any Type-I
singularity model at any point would therefore be flat, so that no
singularity occurs (for example, using Theorem 1.2 of \cite{enderstoppingmuller}).
To make this work for Type-II flows, we need to apply the compactness
and partial regularity theory of \cite{bamlergen1,bamlergen2,bamlergen3}.
In particular, we fix $x\in M$ corresponding to a singular point
$\overline{x}\in X$, a sequence $\tau_{i}\searrow0$, and consider
the Type-I rescaled solutions $g_{t}^{i}:=\tau_{i}^{-1}g_{T+\tau_{i}t}$
and their corresponding conjugate heat kernels $K^{i}$ based at $x$
at the singular time. Seen as metric flow pairs, we can pass to a
subsequence to obtain $\mathbb{F}$-convergence (see Section 5 of
\cite{bamlergen2}) to a metric soliton modeled on a singular shrinking
GRS $(Y,d,\mathcal{R}_{Y},g_{Y},f)$ with nonnegative Ricci curvature
on $\mathcal{R}_{Y}$ (see Section 2 for relevent definitions). This
implies in particular that there are diffeomorphisms $\psi_{i}:U_{i}\to M$,
where $(U_{i})$ is an exhaustion of $\mathcal{R}_{Y}$, such that
$\psi_{i}^{\ast}g^{i}\to g^{\infty}$ in $C_{loc}^{\infty}(\mathcal{R}_{Y})$.

One technical hurdle for proceeding as in the Type-I case is that
in general volume convergence does not follow from $\mathbb{F}$-convergence,
which makes it difficult to show that the limiting GRS has maximal
volume growth. To address this, we prove an estimate (Lemma \ref{lem:firsttechlemma})
for the flows $g^{i}$ showing that, for some $\sigma>0$, the set
of points near an $H_{n}$-center of $x$ which satisfy $K^{i}(\cdot,-1)\geq\sigma$
has almost-full measure. This relies on several results from Cheeger-Colding
theory, including Cheeger-Colding's segment inequality and an estimate
on the size of the quantitative singular strata on each time slice.
From this inequality, we can extract pointed Gromov-Hausdorff convergence
of time slices from $W_{1}$-Gromov-Wasserstein convergence, and then
apply Colding's volume convergence theorem. We use this to establish
Theorem \ref{thm:theoremhigherdim}. A modification of a proof of
Ni's (Proposition 1.1 in \cite{leinisoliton}) shows that Corollary
1.1 of \cite{carilloni} holds for singular shrinking solitons in
our setting, which we use to prove Theorem \ref{thm:theorem4} and
part $(i)$ of Theorem \ref{thm:theorem2}.

In Section 4, we specialize to dimension four, and prove part $(ii)$
of Theorem \ref{thm:theorem2} by showing that for each time slice
of the Type-I rescaled solutions, $x$ is near the part of $M$ corresponding
to the vertex $o_{\ast}\in C(\mathbb{S}^{3}/\Gamma)$. The idea is
that the (rescaled) conjugate heat kernels based at $(x,T)$ converge
in the Cheeger-Gromov sense to the heat kernel of $C(\mathbb{S}^{3}/\Gamma)$,
which gives a lower bound on $K^{i}$ at bounded distance from points
corresponding to the vertex. On the other hand, Bamler's Gaussian
upper bound on the conjugate heat kernel in terms of distance from
an $H_{4}$-center contradicts this lower bound if an $H_{4}$ center
is too far away from points corresponding to the vertex. We combine
this with a Gaussian heat kernel upper bound (see the Appendix) to
show that $x$ is in the ``inner'' component of $M$ separated by
some hypersurfaces $\psi_{i}(\partial B(o_{\ast},\alpha_{i}))$, where
$\alpha_{i}\searrow0$, and then show that the diameter of this component
goes to zero. We prove Theorem \ref{thm:theorem1} by applying Perelman's
pseudolocality theorem on regions of $(M,g_{t}^{i})$ corresponding
to annuli centered at the vertex of $C(\mathbb{S}^{3}/\Gamma)$, which
(as in \cite{enderstoppingmuller}) implies that the curvature bound
extends forwards in time all the way to the singular time, and is
uniform on annuli centered at $x$. We can extract from this that
$\overline{x}$ is an isolated singular point of $X$, and that the
curvature bound blows up no faster than the rate $\frac{1}{d^{2}(\cdot,\overline{x})}$
on $X$. Then a maximum principle argument for the Ricci tensor rules
out nonflat tangent cones, and implies that the curvature blows up
along $X$ at a rate strictly slower than the rate $\frac{1}{d^{2}(\cdot,\overline{x})}$.
Standard arguments then show that $X$ must be a $C^{0}$ orbifold,
and Simon's construction allows one to flow through the singularity. 

In Section 5, we prove a codimension two $\epsilon$-regularity result,
which roughly says that, for a closed Ricci flow satisfying (\ref{eq:ric}),(\ref{eq:vol}),
Gromov-Hausdorff closeness to a cone which splits $\mathbb{R}^{n-2}$
implies a curvature bound, at scales $<<\sqrt{T-t}$. This is done
by a contradiction-compactness argument with careful point-picking,
where we first rescale and obtain points which converge in the Gromov-Hausdorff
sense to $\mathbb{R}^{n-2}\times Z$ for some metric space $Z$, but
whose curvature blows up. We use Cheeger-Naber's slicing theorem as
in the proof of Theorem 5.2 in \cite{cheegernaberdim4} to change
the blowup points (without changing the time slices) and further rescale
to get smooth Cheeger-Gromov convergence to $\mathbb{R}^{n-2}\times S$,
where $S$ is an immortal 2-dimensional Ricci flow with nonnegative
curvature. Using the classification of 2-dimensional steady and expanding
solitons, we can change basepoints and scales once again to get smooth
Cheeger-Gromov convergence to $\mathbb{R}^{n-2}\times E$, where $E$
is an expanding Ricci soliton asymptotic at infinity to $C(\mathbb{S}_{\beta}^{1})$,
where $\mathbb{S}_{\beta}^{1}$ denotes the circle with circumference
$\beta\in(0,2\pi)$. Because this is still a blowup of the original
flow, we can use Bamler's compactness theory for $\mathbb{F}$-convergence
to obtain an ancient metric flow coinciding with $\mathbb{R}^{n-2}\times E$
for time $t\in(0,1)$. However, at its initial time $t=0$, the soliton
$E$ converges smoothly to the flat cone $C(\mathbb{S}_{\beta}^{1})$
away from its vertex, so we can show that the metric flow is static
for all times $t\in(-\infty,0]$, and in fact coincides with $\mathbb{R}^{n-2}\times C(\mathbb{S}_{\beta}^{1})$
on each time slice, contradicting the Minkowski dimension estimates
for the singular set of static metric flows. 

In Section 6, we prove a codimension three $\epsilon$-regularity
result via another contradiction argument, where now the Gromov-Hausdorff
limit is $\mathbb{R}^{n-3}\times C(Z)$ for some two dimensional metric
space $Z$. To show $Z$ is smooth, we use the codimension two $\epsilon$-regularity
theorem to estimate size of the quantitative singular set at scales
$<<\sqrt{T-t}$. This rules out codimension 2 singularities for $\mathbb{R}^{n-3}\times C(Z)$,
showing that $Z$ is smooth. Then a maximum principle argument for
the Ricci curvature on the regular set of $\mathbb{R}^{n-3}\times C(Z)$
guarantees that the cone is flat, and because $M$ is orientable,
it follows that $\mathbb{R}^{n-3}\times C(Z)$ is the Euclidean cone.
From this, we obtain Theorem \ref{thm:highdimlp}.

In Section 7, we prove Theorem \ref{thm:theorem3} using a time-dependent
decomposition of time slices to estimate the size of the set $\{\widetilde{r}_{Rm}(\cdot,t)<s\}$.
When $s<<\sqrt{T-t},$ this is achieved by combining codimension three
$\epsilon$-regularity with Cheeger-Jiang-Naber's volume estimates
for quantitative singular strata. When $s\approx\sqrt{T-t}$, our
estimate uses our knowledge of the flow at the Type-I scale. When
$s>>\sqrt{T-t}$, we use pseudolocality to estimate this region using
the corresponding region of the limit space $X$.

Finally, in the Appendix, we indicate a modification of the proof
of Theorem 3.1 in \cite{cao}, which leads to a Gaussian upper bound
for the heat kernel in the presence of a negative Ricci curvature
lower bound. This result was used in Section 4.

\medskip{}

The author thanks his advisor Xiaodong Cao for his encouragement and
helpful discussions, as well as Richard Bamler and Bennett Chow for
their useful comments and suggestions.

\section{Preliminaries}

Throughout this paper, we usually omit the dependence of constants
on the dimension $n$. Given a closed Ricci flow $(M^{n},(g_{t})_{t\in[0,T)})$,
along with $(x,t)\in M\times[0,T)$ and $r>0$, we write 
\[
B_{g}(x,t,r):=\{y\in M;d_{g_{t}}(x,y)<r\}.
\]
Let $dg_{t}$ denote the Riemannian volume measure on $M$ corresponding
to $g_{t}$, and for a subset $A\subseteq M$, we let $|A|_{g_{t}}$
denote the volume of $A$ with respect to $g_{t}$. Given $(x,t)\in M\times[0,T)$,
we let $K(x,t;\cdot,\cdot):M\times[0,t)\to(0,\infty)$ denote the
conjugate heat kernel based at $(x,t)$. We also define the probability
measures $\nu_{x,t;s}:=K(x,t;\cdot,s)dg_{s}=(4\pi(t-s))^{-\frac{n}{2}}e^{-f_{x,t}}dg_{s}$
for $s<t$. Also, let 
\[
\mathcal{N}_{x,t}^{g}(\tau):=\int_{M}fd\nu_{x,t;t-\tau}-\frac{n}{2}
\]
denote the pointed Nash entropy. Denote by $\ell_{(x,t)}^{g}$ Perelman's
reduced distance, based at $(x,t)\in M\times[0,T)$. 

Given a metric space $(X,d)$, along with $x\in X$ and $r>0$, we
denote by $B^{X}(x,r)$ the open metric ball, though we may write
$B(x,r)$ when $(X,d)$ is unambiguous. If $\text{diam}(X)\leq\pi$,
we let $(C(X),d_{C(X)},o_{\ast})$ denote the metric cone over $(X,d)$,
the pointed metric space whose underlying set is $([0,\infty)\times X)/(\{0\}\times X)$,
with vertex $o_{\ast}$, and with metric defined by
\[
d_{C(X)}^{2}([r_{1},x_{1}],[r_{2},x_{2}])=r_{1}^{2}+r_{2}^{2}-2r_{1}r_{2}\cos(d_{X}(x_{1},x_{2}))
\]
for $[r_{1},x_{1}],[r_{2},x_{2}]$, where we let $[r,x]\in C(X)$
denote the equivalence class corresponding to $(r,x)$. 

For bounded subsets $B_{1},B_{2}$ of a metric space $(X,d)$, let
$d_{H}^{X}(B_{1},B_{2})$ be the corresponding Hausdorff distance.
For Borel probability measures $\mu_{1},\mu_{2}$ on $(X,d)$, let
$d_{W_{1}}^{X}(\mu_{1},\mu_{2})$ denote their $W_{1}$-Wasserstein
distance. For pointed metric spaces $(X_{1},d_{1},x_{1})$, $(X_{2},d_{2},x_{2})$,
we denote by 
\[
d_{GH}((X_{1},d_{1}),(X_{2},d_{2}))
\]
the Gromov-Hausdorff distance between the underlying metric spaces,
and
\[
d_{PGH}((X_{1},d_{1},x_{1}),(X_{2},d_{2},x_{2}))
\]
the pointed Gromov-Hausdorff distance between the pointed metric spaces
(for definitions and basic properties, see Chapters 7,8 of \cite{metricgeometry}).
For convenience, when $(X_{i},d_{i},x_{i})$ are not bounded, we define
\[
d_{PGH}((X_{1},d_{1},x_{1}),(X_{2},d_{2},x_{2})):=\sum_{j=1}^{\infty}2^{-j}\frac{d_{PGH}((B^{X_{1}}(x_{1},j),d_{1},x_{1}),(B^{X_{2}}(x_{2},j),d_{2},x_{2}))}{1+d_{PGH}((B^{X_{1}}(x_{1},j),d_{1},x_{1}),(B^{X_{2}}(x_{2},j),d_{2},x_{2}))},
\]
which metrizes the pointed Gromov-Hausdorff topology on the class
of isometry classes of complete metric length spaces. 

We now establish several basic geometric estimates that follow from
assumptions (\ref{eq:ric}),(\ref{eq:vol}).
\begin{lem}
\label{lem:basiclemma} There exists $\overline{A}=\overline{A}(n,A,T,\text{diam}_{g_{0}}(M))\in(1,\infty)$
such that the following hold for any Ricci flow satisfying (\ref{eq:ric}),(\ref{eq:vol}):

$(i)$$\overline{A}^{-1}\leq\text{diam}_{g_{t}}(M)\leq\overline{A}$
for all $t\in[0,T)$,

$(ii)$ For every $r\in(0,\max\{1,\sqrt{T}\}]$ and $(x,t)\in M\times[0,T)$,
$\overline{A}^{-1}r^{n}\leq|B(x,t,r)|_{g_{t}}\leq\overline{A}r^{n}$,

$(iii)$ $\mathcal{N}_{x,t}(\tau)\geq-\overline{A}$ for all $(x,t)\in M\times[0,T)$
and $\tau\in(0,t]$. 
\end{lem}

\begin{proof}
$(i)$ Since $[0,T)\to(0,\infty),t\mapsto\text{diam}_{g(t)}(M)$ is
locally Lipschitz, it suffices to note that, for any $x\in M$ and
$V\in T_{x}M$, we have
\[
\log\left(\frac{g_{t}(V,V)}{g_{0}(V,V)}\right)=\int_{0}^{t}\partial_{s}\log g_{s}(V,V)ds=-2\int_{0}^{t}\dfrac{Rc_{g_{s}}(V,V)}{g_{s}(V,V)}ds\leq2At,
\]
so that $\text{diam}_{g_{t}}(M)\leq e^{AT}\text{diam}_{g_{0}}(M)$.
The lower bound follows by combining $\inf_{t\in[0,T)}|M|_{g_{t}}\geq A^{-1}$
with $|M|_{g_{t}}\leq v_{-A}(\text{diam}_{g_{t}}(M))$.

\noindent $(ii)$ The upper bound follows from Bishop volume comparison.
For the lower bound, fix $(x,t)\in[0,T)$, $r\in(0,\max\{1,\sqrt{T}\}]$,
and let $\overline{A}=\overline{A}(n,\text{diam}_{g_{0}}(M),T,A)$
be the constant from $(i)$. Then
\begin{align*}
|B(x,t,r)|_{g_{t}}= & \dfrac{|B(x,t,r)|_{g_{t}}}{|B(x,t,\overline{A}+\max\{1,\sqrt{T}\})|_{g_{t}}}|M|_{g_{t}}\\
\geq & \dfrac{v_{-A}(r)}{v_{-A}(\overline{A}+\max\{1,\sqrt{T}\})}A^{-1}\\
\geq & C(n,\text{diam}_{g_{0}}(M),T,A)r^{n},
\end{align*}
so $(ii)$ follows after possibly modifying $\overline{A}$. 

$(iii)$ Because $R\geq-nA$, we can apply Theorem 8.1 of \cite{bamlergen1}
with $r=\sqrt{\tau}$ and use $(ii)$ to estimate
\[
\overline{A}^{-1}\tau^{\frac{n}{2}}\leq\left|B\left(x,t,\sqrt{\tau}\right)\right|_{g_{t}}\leq C(n,A,T)\exp(\mathcal{N}_{x,t}(\tau))\tau^{\frac{n}{2}}.
\]
\end{proof}
We recall the following notions of curvature scale:
\begin{defn}
For $(x,t)\in M\times[0,T)$, we define 
\[
r_{Rm}(x,t):=\sup\{r>0;|Rm|\leq r^{-2}\text{ on }B(x,t,r)\times([t-r^{2},t+r^{2}]\cap[0,T))\},
\]
\[
\widetilde{r}_{Rm}(x,t):=\sup\{r>0;|Rm|\leq r^{-2}\text{ on }B(x,t,r)\}.
\]

It is important that these two notions are comparable, which follows
from combining Perelman's pseudolocality theorem \cite{perl1} (we
use the version stated and proved in \cite{lupenglocal}) with the
backwards pseudolocality theorem, proved for Ricci flows with a Ricci
lower bound by Chen-Yuan (Theorem 3 of \cite{chenlowerbound}), and
later in the general setting by Bamler (Theorem 1.48 of \cite{bamlergen3}).
\end{defn}

\begin{thm}
\label{thm:pseudolocality} (Forwards and Backwards Pseudolocality)
For any $A<\infty$, there exists $\epsilon_{P}=\epsilon_{P}(A)>0$
such that the following hold. Suppose $(M^{n},(g_{t})_{t\in[-2,T]})$
is a closed, pointed Ricci flow (where $T\geq0$) satisfying $Rc(g_{t})\geq-Ag_{t}$
and $|B(x,t,r)|_{g_{t}}\geq A^{-1}r^{n}$ for all $(x,t)\in M\times[-1,T]$
and $r\in(0,1]$. If $x\in M$ and $r\in(0,1]$ are such that $|Rm|(\cdot,0)\leq r^{-2}$
on $B(x,0,r)$, then $|Rm|\leq(\epsilon_{P}r)^{-2}$ on $B(x,0,\epsilon_{P}r)\times[-(\epsilon_{P}r)^{2},\min((\epsilon_{P}r)^{2},T)]$.
In particular, we have 
\[
r_{Rm}(y,t)\leq\widetilde{r}_{Rm}(y,t)\leq\epsilon_{P}^{-1}(A)r_{Rm}(y,t)
\]
for all $(y,t)\in M\times[-1,T]$. 
\end{thm}

\begin{proof}
\noindent The bound $|Rm|\leq(\epsilon_{0}r)^{-2}$ on $B(x,0,\epsilon_{0}r)\times[0,\min((\epsilon_{0}r)^{2},T)]$
for some $\epsilon_{0}=\epsilon_{0}(A)>0$ is Proposition 2.1 of \cite{lupenglocal},
while the analogous bound on $B(x,0,\epsilon_{1}r)\times[-(\epsilon_{1}r)^{2},0]$
for some $\epsilon_{1}=\epsilon_{1}(A)>0$ is Theorem 2.47 of \cite{bamlergen3}.
Take $\epsilon_{P}(A):=\min(\epsilon_{0}(A),\epsilon_{1}(A))$. For
the remaining claim, we suppose $(y,t)\in M\times[-1,T]$ satisfies
$\widetilde{r}_{Rm}(y,t)=r$. Then $|Rm|(\cdot,t)\leq r^{-2}$, so
applying the first claim to the time-translated flow gives $r_{Rm}(y,t)\geq\epsilon_{P}(A)r$. 
\end{proof}
If $(M,g)$ is a (possibly incomplete) Riemannian manifold, we let
$\widetilde{r}_{Rm}^{(M,g)}(x)$ be the supremum of all $r>0$ such
that $B_{g}(x,r)$ has compact closure in $M$ and $|Rm|\leq r^{-2}$
on $B_{g}(x,r)$. If $(X,d)$ is a noncollapsed Ricci limit space
whose regular part $\mathcal{R}$ is open and has the structure of
a smooth Riemannian manifold $(\mathcal{R},g)$, we let $\widetilde{r}_{Rm}^{X}(x)$
be the supremum of $r>0$ such that $B(x,r)\subseteq\mathcal{R}$
and $|Rm|\leq r^{-2}$ on $B(x,r)$. Equivalently, $\widetilde{r}_{Rm}^{X}=\widetilde{r}_{Rm}^{(\mathcal{R},g)}$
on $\mathcal{R}$, while $\widetilde{r}_{Rm}^{X}=0$ on $X\setminus\mathcal{R}$. 

We also make use of the following theorem from \cite{chenlowerbound},
which is essentially a combination of backwards pseudolocality and
the point-picking procedure from Anderson's $\epsilon$-regularity
theorem. We let $v_{-\rho}(r)$ denote the volume of a geodesic ball
of radius $r$ in the simply connected $n$-dimensional space form
with curvature $-\frac{\rho}{n-1}$, and let $\omega_{n}$ be the
volume of the unit ball in $\mathbb{R}^{n}$.
\begin{thm}
\label{thm:chenyuan} (Chen-Yuan $\epsilon$-Regularity) For any $A\in[0,\infty)$
and $\lambda\in(0,1]$, there exists $\delta=\delta(A)>0$ such that
the following hold. If $(M^{n},(g_{t})_{t\in[-\lambda^{-1},0]})$
is a closed, pointed Ricci flow satisfying $Rc(g_{t})\geq-\lambda Ag_{t}$
and $|B(x,t,r)|_{g_{t}}\geq A^{-1}r^{n}$ for all $(x,t)\in M\times[-\lambda^{-1},0]$
and $r\in(0,\lambda^{-\frac{1}{2}}]$, then for any $x\in M$ and
$r\in(0,\lambda^{-\frac{1}{2}}]$ with $|B(x,0,r)|_{g_{0}}>(1-\delta)v_{-\lambda A}(r)$,
we have $\widetilde{r}_{Rm}(x,0)\geq\delta r$. 
\end{thm}

\begin{rem}
\label{rem:volremark} We will often consider a blowup sequence $(g_{t}^{i})_{t\in[-\lambda_{i},0]}$
of Ricci flows satisfying (\ref{eq:ric}),(\ref{eq:vol}), in which
case we will apply this theorem with parameters $\lambda_{i}\searrow0$.
Then for any $r\in(0,D]$, we can replace the almost-maximal volume
ratio with the assumption
\[
|B_{g^{i}}(x,0,r)|_{g_{0}^{i}}>(1-\delta)\omega_{n}r^{n}
\]
when $i=i(D)\in\mathbb{N}$ is sufficiently large, after possibly
adjusting $\delta>0$, since 
\[
\inf_{r\in(0,D]}\frac{\omega_{n}r^{n}}{v_{-\lambda_{i}A}(r)}=\frac{\int_{0}^{D}s^{n-1}ds}{\int_{0}^{D}\left(\sqrt{\frac{n-1}{\lambda_{i}A}}\sinh\left(\sqrt{\frac{\lambda_{i}A}{n-1}s}\right)\right)^{n-1}ds}\to1
\]
as $i\to\infty$.

The proof of Theorem \ref{thm:chenyuan} given in \cite{chenlowerbound}
contains a small gap (their point-picking does not lead to a limit
with Euclidean volume growth), so we indicate the changes needed to
give a correct proof of (a modification of) their statement.
\end{rem}

\begin{proof}
We first prove the theorem when $\lambda=1$. As in \cite{chenlowerbound},
proceed by contradiction and choose (after parabolic rescaling) a
sequence $(M_{i},(g_{t}^{i})_{t\in[-1,0]})$ along with $x_{i}\in M_{i}$,
$\delta_{i}\searrow0$, $A_{i}\in[0,A]$, such that $Rc(g_{t}^{i})\geq-A_{i}g_{t}^{i}$,
$|B_{g^{i}}(x,t,r)|_{g_{t}^{i}}\geq A^{-1}r^{n}$ for all $(x,t)\in M\times[-1,0]$
and $r\in(0,1]$, $|B_{g^{i}}(x_{i},0,1)|_{g_{0}^{i}}>(\omega_{n}-\delta_{i})v_{-A_{i}}(1)$,
but also $\widetilde{r}_{Rm}^{g^{i}}(x_{i},0)<\delta_{i}$. Choose
$s_{i}\searrow0$ such that $\lim_{i\to\infty}\frac{\widetilde{r}_{Rm}^{g^{i}}(x_{i},0)}{s_{i}}=0$,
and then choose $y_{i}\in B_{g^{i}}(x_{i},0,s_{i})$ minimizing
\[
\epsilon_{i}(y):=\dfrac{\widetilde{r}_{Rm}^{g^{i}}(y,0)}{d_{g_{0}^{i}}(y,\partial B_{g^{i}}(x_{i},0,s_{i}))},
\]
so that 
\[
\epsilon_{i}(y_{i})\leq\epsilon_{i}(x_{i})=\frac{\widetilde{r}_{Rm}^{g^{i}}(x_{i},0)}{s_{i}}\to0
\]
as $i\to\infty$. Set $r_{i}:=\widetilde{r}_{Rm}^{g^{i}}(y_{i},0)$,
$\widetilde{g}_{t}^{i}:=r_{i}^{-2}g_{r_{i}^{2}t}^{i}$ for $t\in[-1,0]$,
so that $\widetilde{r}_{Rm}^{\widetilde{g}^{i}}(y_{i},0)=1$. For
$y\in M_{i}$ with $d_{\widetilde{g}_{0}^{i}}(y_{i},y)\leq\frac{1}{2}d_{\widetilde{g}_{0}^{i}}(y_{i},\partial B_{g_{0}^{i}}(x_{i},0,s_{i}))$,
we have 
\[
\widetilde{r}_{Rm}^{\widetilde{g}^{i}}(y,0)=\dfrac{\widetilde{r}_{Rm}^{\widetilde{g}^{i}}(y,0)}{\widetilde{r}_{Rm}^{\widetilde{g}^{i}}(y_{i},0)}\geq\dfrac{d_{\widetilde{g}_{0}^{i}}(y,\partial B_{g_{0}^{i}}(x_{i},0,s_{i}))}{d_{\widetilde{g}_{0}^{i}}(y_{i},\partial B_{g_{0}^{i}}(x_{i},0,s_{i}))}\geq\frac{1}{2}.
\]
However, we also know that 
\[
d_{\widetilde{g}_{0}^{i}}(y_{i},\partial B_{g_{0}^{i}}(x_{i},0,s_{i}))=\frac{d_{g_{0}^{i}}(y_{i},\partial B_{g_{0}^{i}}(x_{i},0,s_{i}))}{\widetilde{r}_{Rm}^{g_{0}^{i}}(y_{i},0)}=\epsilon_{i}(y_{i})^{-1}\to\infty,
\]
so, for any $D<\infty$, we have $\widetilde{r}_{Rm}^{\widetilde{g}^{i}}(\cdot,0)\geq\frac{1}{2}$
on $B_{\widetilde{g}^{i}}(y_{i},0,D)$ for $i=i(D)\in\mathbb{N}$
sufficiently large. By Theorem \ref{thm:pseudolocality}, Shi's estimates,
and the assumed lower bound on volume of geodesic balls, we can pass
to a subsequence to get $C^{\infty}$ Cheeger-Gromov convergence $(M_{i},\widetilde{g}_{0}^{i},y_{i})\to(M_{\infty},\widetilde{g}^{\infty},y_{\infty})$
for some complete Riemannian manifold with bounded curvature, such
that $\widetilde{r}_{Rm}^{\widetilde{g}^{\infty}}(y_{\infty})=1$
and $Rc(\widetilde{g}^{\infty})\geq0$ everywhere. 

\noindent \textbf{Claim: }$|B_{\widetilde{g}^{\infty}}(y_{\infty},r)|_{\widetilde{g}^{\infty}}=\omega_{n}r^{n}$
for all $r\in(0,\infty)$.

Since $Rc(g_{0}^{i})\geq-A_{i}g_{0}^{i}$, $|B_{g^{i}}(y_{i},0,1)|_{g_{0}^{i}}>(1-\delta_{i})v_{-A_{i}}(1)$,
and $d_{i}:=d_{g_{0}^{i}}(x_{i},y_{i})\to0$, we can use Bishop volume
comparison to conclude
\begin{align*}
|B_{g^{i}}(y_{i},0,1)|_{g_{0}^{i}}\geq & \dfrac{v_{-A_{i}}(1)}{v_{-A_{i}}(1+d_{i})}|B_{g^{i}}(y_{i},0,1+d_{i})|_{g_{0}^{i}}\\
> & (1-\delta_{i})\frac{v_{-A_{i}}(1)}{v_{-A_{i}}(1+d_{i})}v_{-A_{i}}(1)=:\tau_{i}v_{-A_{i}}(1),
\end{align*}
where $\tau_{i}\to1$ as $i\to\infty$. Now rescale and pass to the
limit to get, for any fixed $r\in(0,\infty)$, 
\begin{align*}
|B_{\widetilde{g}^{\infty}}(y_{\infty},r)|_{\widetilde{g}^{\infty}}= & \lim_{i\to\infty}|B_{\widetilde{g}^{i}}(y_{i},0,r)|_{\widetilde{g}_{0}^{i}}=\lim_{i\to\infty}r_{i}^{-n}|B_{g^{i}}(y_{i},0,r_{i}r)|_{g_{0}^{i}}\\
\geq & \limsup_{i\to\infty}r_{i}^{-n}\tau_{i}v_{-A_{i}}(r_{i}r)=\limsup_{i\to\infty}v_{-r_{i}^{2}A_{i}}(r)=\omega_{n}r^{n}.\quad\text{\ensuremath{\square}}
\end{align*}

The only complete Riemannian manifold with nonnegative Ricci curvature
and Eulidean volume growth is flat $\mathbb{R}^{n}$, so $(M_{\infty},\widetilde{g}^{\infty})$
is flat $\mathbb{R}^{n}$, contradicting $\widetilde{r}_{Rm}^{g^{\infty}}(y_{\infty})=1$. 

Finally, consider the case where $\lambda\in(0,1)$. Define $\overline{g}_{t}:=\lambda g_{\lambda^{-1}t}$
for $t\in[-1,0]$, so that $Rc(\overline{g}_{t})\geq-A\overline{g}_{t}$,
$|B_{\overline{g}}(y,t,s)|_{\overline{g}_{t}}\geq A^{-1}s^{n}$ for
all $(y,t)\in M\times[-1,0]$, $s\in(0,1]$, and also 
\[
|B_{\overline{g}}(x,0,\lambda^{\frac{1}{2}}r)|_{\overline{g}}\geq(1-\delta)\lambda^{\frac{n}{2}}v_{-\lambda A}(r)=(1-\delta)v_{-A}(\lambda^{\frac{1}{2}}r),
\]
where $\lambda^{\frac{1}{2}}r\in(0,1]$. Then we can apply the $\lambda=1$
case, replacing $r$ with $\lambda^{\frac{1}{2}}r$, to obtain $\widetilde{r}_{Rm}^{\overline{g}}(x,0)\geq\delta\lambda^{\frac{1}{2}}r$,
and the claim follows.
\end{proof}
Using Theorem \ref{thm:chenyuan} and volume comparison, Chen-Yuan
prove that the regular set of a noncollapsed Ricci limit space is
open and equipped with a smooth Riemannian metric if it is the Gromov-Hausdorff
limit of time-slices of Ricci flows satisfying (\ref{eq:ric}),(\ref{eq:vol}).
This is made precise in the following.
\begin{thm}
\label{thm:openandsmooth} (Theorem 1 of \cite{chenlowerbound}) Given
$A<\infty$ and $\underline{T}>0$, suppose $(M_{i},(g_{t}^{i})_{t\in[-T_{i},0]})$
is a sequence of Ricci flows satisfying $Rc(g_{t}^{i})\geq-Ag_{t}^{i}$
and $|B_{g^{i}}(x,t,r)|_{g_{t}^{i}}\geq A^{-1}r^{n}$ for all $(x,t)\in M\times[-T_{i},0]$
and $r\in(0,1]$, where $T_{i}\geq\underline{T}$. If $(X,d,x)$ is
a pointed, complete metric length space such that 
\[
(M,d_{g_{0}^{i}},x_{i})\to(X,d,x)
\]
in the pointed Gromov-Hausdorff sense for some $x_{i}\in M_{i}$,
then $(X,d)$ is a noncollapsed Ricci limit space with open regular
part $\mathcal{R}$, and $\mathcal{R}$ admits the structure of a
smooth Riemannian manifold $(\mathcal{R},g)$ such that $d|(\mathcal{R}\times\mathcal{R})$
is the length metric $d_{g}$ of $(\mathcal{R},g)$. Moreover, there
is an increasing exhaustion $(U_{i})$ of $\mathcal{R}$ by precompact
open sets with diffeomorphisms $\psi_{i}:U_{i}\to M$ such that $\psi_{i}(y)\to y$
locally uniformly in $\mathcal{R}$ (with respect to the Gromov-Hausdorff
convergence) and $\psi_{i}^{\ast}g_{0}^{i}\to g$ in $C_{loc}^{\infty}(\mathcal{R})$. 
\end{thm}

The following proposition gives a sufficient criterion for Gromov-Wasserstein
convergence to imply Gromov-Hausdorff convergence. 
\begin{prop}
\label{prop:easyconvergence} Suppose $(X_{i},d_{i},\mu_{i})$ is
a sequence of complete metric measure spaces converging in the $W_{1}$-Gromov-Wasserstein
sense to a complete, locally compact metric measure space $(X_{\infty},d_{\infty},\mu_{\infty})$
with $\text{supp}(\mu_{\infty})=X_{\infty}$, such that $(X_{\infty},d_{\infty}),(X_{i},d_{i})$
are all length spaces. Suppose $x_{i}\in X$ are such that, for any
$D<\infty$ and $r>0$, there exists $c=c(r,D)>0$ with $\mu_{i}(B^{X_{i}}(x,r))\geq c$
for all $x\in B^{X_{i}}(x_{i},D)$. Finally, assume $\text{Var}(\mu_{\infty})\leq H<\infty$,
and choose $x_{\infty}^{\ast}\in X_{\infty}$ such that $\text{Var}(\mu_{\infty},\delta_{x_{\infty}^{\ast}})\leq H$.
Then there exist $x_{i}^{\ast}\in X_{i}$ with $d_{i}(x_{i},x_{i}^{\ast})\leq C(H,c(\frac{1}{2},1))$
such that 
\[
(X_{i},d_{i},x_{i}^{\ast})\to(X_{\infty},d_{\infty},x_{\infty}^{\ast})
\]
in the pointed Gromov-Haudorff sense. In fact, if $\phi_{i}:(X_{i},d_{i})\to(Z,d^{Z})$
and $\phi_{\infty}:(X_{\infty},d_{\infty})\to(Z,d^{Z})$ are isometric
embeddings into a common metric space such that 
\begin{equation}
\lim_{i\to\infty}d_{W_{1}}^{Z}\left((\phi_{i})_{\ast}\mu_{i},(\phi_{\infty})_{\ast}\mu_{\infty}\right)=0,\label{eq:W1convergence}
\end{equation}
then we also have
\[
\lim_{i\to\infty}d_{H}^{Z}\left(B^{Z}(\phi_{i}(x_{i}^{\ast}),r),B^{Z}(\phi_{\infty}(x_{\infty}^{\ast}),r)\right)=0
\]
for each fixed $r\in(0,\infty)$. 

By passing to a subsequence, we can therefore find $x_{\infty}\in X_{\infty}$
such that $(X_{i},d_{i},x_{i})\to(X_{\infty},d_{\infty},x_{\infty})$
in the pointed Gromov-Hausdorff sense. 
\end{prop}

\begin{proof}
Given that $(X_{i},d_{i},\mu_{i})\to(X_{\infty},d_{\infty},\mu_{\infty})$
in the $W_{1}$-Gromov-Wasserstein sense, a direct limit construction
gives $(Z,d^{Z}),\phi_{i},\phi_{\infty}$ such that (\ref{eq:W1convergence})
is satisfied. Let $q_{i}$ be couplings of $\mu_{i},\mu_{\infty}$
realizing this convergence, so that 
\[
\lim_{i\to\infty}\int_{X_{i}\times X_{\infty}}d^{Z}(\phi_{i}(x),\phi_{\infty}(y))dq_{i}(x,y)=0.
\]

\noindent \textbf{Claim 1:} For each $D<\infty$, we have
\[
\lim_{i\to\infty}\sup_{x\in B^{X_{i}}(x_{i},D)}d^{Z}\left(\phi_{i}(x),\phi_{\infty}(X_{\infty})\right)=0.
\]
Otherwise, we can pass to a subsequence to obtain $\epsilon>0$ and
$y_{i}\in B^{X_{i}}(x_{i},D)$ such that $d^{Z}(\phi_{i}(y_{i}),\phi_{\infty}(X_{\infty}))\geq2\epsilon$
for all $i\in\mathbb{N}$. Then $d^{Z}(\phi_{i}(x),\phi_{\infty}(X_{\infty}))\geq\epsilon$
for all $x\in B^{X_{i}}(y_{i},\epsilon)$, hence 
\[
\epsilon\mu_{i}(B^{X_{i}}(y_{i},\epsilon))\leq\int_{B^{X_{i}}(y_{i},\epsilon)\times X_{\infty}}d^{Z}(\phi_{i}(x),\phi_{\infty}(y))dq_{i}(x,y)\to0,
\]
contradicting $\liminf_{i\to\infty}\mu_{i}(B^{X_{i}}(y_{i},\epsilon))\geq c(\epsilon,D)>0$.
$\square$

Now choose $x_{i}'\in X_{\infty}$ such that $d^{Z}(\phi_{\infty}(x_{i}'),\phi_{i}(x_{i}))\to0$. 

\noindent \textbf{Claim 2: }$\liminf_{i\to\infty}\mu_{\infty}(B^{X_{\infty}}(x_{i}',1))\geq c(1/2,1)$.

When $i\in\mathbb{N}$ is sufficiently large, we have
\begin{align*}
d^{Z}(\phi_{i}(x),\phi_{\infty}(y))\geq & d^{Z}(\phi_{\infty}(y),\phi_{\infty}(x_{i}'))-d^{Z}(\phi_{\infty}(x_{i}'),\phi_{i}(x_{i}))-d^{Z}(\phi_{i}(x_{i}),\phi_{i}(x))\geq\frac{1}{4}
\end{align*}
for all $x\in B^{X_{i}}(x_{i},\frac{1}{2})$ and $y\in X_{\infty}\setminus B^{X_{\infty}}(x_{i}',1)$
so
\[
\frac{1}{4}q_{i}\left(B^{X_{i}}(x_{i},\frac{1}{2})\times\left(X_{\infty}\setminus B^{X_{\infty}}(x_{i}',1)\right)\right)\leq\int_{B^{X_{i}}(x_{i},\frac{1}{2})\times\left(X_{\infty}\setminus B^{X_{\infty}}(x_{i}',1)\right)}d^{Z}(\phi_{i}(x),\phi_{\infty}(y))dq_{i}(x,y),
\]
but
\[
q_{i}\left(B^{X_{i}}(x_{i},\frac{1}{2})\times B^{X_{\infty}}(x_{i}',1)\right)\leq\mu_{\infty}(B^{X_{\infty}}(x_{i}',1)),
\]
so combining estimates gives
\begin{align*}
c(1/2,1)\leq & \liminf_{i\to\infty}\mu_{i}\left(B^{X_{i}}(x_{i},\frac{1}{2})\right)=\liminf_{i\to\infty}q_{i}\left(B^{X_{i}}(x_{i},\frac{1}{2})\times X_{\infty}\right)\\
\leq & \liminf_{i\to\infty}\mu_{\infty}(B^{X_{\infty}}(x_{i}',1)),\hfill\hfill\hfill\hfill\hfill\square
\end{align*}

If $d(x_{i}',x_{\infty}^{\ast})>\sqrt{\frac{H}{c(1/2,1)}}+1$, then
\[
H\geq\int_{B(x_{i}',1)}d^{2}(x_{\infty}^{\ast},x)d\mu_{\infty}(x)>\frac{H}{c(1/2,1)}\mu_{\infty}(B(x_{i}',1))\geq H,
\]
a contradiction. We therefore have $\Lambda=\Lambda(H,c(1/2,1))<\infty$
such that $x_{i}'\in B(x_{\infty}^{\ast},\frac{1}{2}\Lambda)$ for
all $i\in\mathbb{N}$ sufficiently large. 

\noindent \textbf{Claim 3: }For each $D<\infty$, we have 
\[
\lim_{i\to\infty}\sup_{x\in B^{X_{\infty}}(x_{\infty}^{\ast},D)}d^{Z}(\phi_{i}(X_{i}),\phi_{\infty}(x))=0.
\]
Otherwise, we can pass to a subsequence to obtain $\epsilon>0$ and
$y_{i}\in B^{X_{\infty}}(x_{\infty}^{\ast},D)$ such that $d^{Z}(\phi_{\infty}(y_{i}),\phi_{i}(X_{i}))\geq3\epsilon$
for all $i\in\mathbb{N}$. We can pass to a further subsequence so
that $y_{i}\to y_{\infty}\in\overline{B}^{X_{\infty}}(x_{\infty}^{\ast},D)$.
Then
\[
d^{Z}(\phi_{\infty}(y_{\infty}),\phi_{i}(X_{i}))\geq3\epsilon-d^{Z}(\phi_{\infty}(y_{i}),\phi_{\infty}(y_{\infty}))\geq2\epsilon
\]
for sufficiently large $i\in\mathbb{N}$, hence
\[
\epsilon\mu_{\infty}\left(B^{X_{\infty}}(y_{\infty},\epsilon)\right)\leq\int_{X_{i}\times B^{X_{\infty}}(y_{\infty},\epsilon)}d^{Z}(\phi_{i}(x),\phi_{\infty}(y))dq_{i}(x,y)\to0
\]
as $i\to\infty$, contradicting the fact that $\mu_{\infty}$ has
full support. $\square$

In particular, we can find $x_{i}^{\ast}\in X_{i}$ such that $d^{Z}(\phi_{i}(x_{i}^{\ast}),\phi_{\infty}(x_{\infty}^{\ast}))\to0$,
which implies 
\[
d_{i}(x_{i},x_{i}^{\ast})\leq d^{Z}(\phi_{i}(x_{i}),\phi_{\infty}(x_{i}'))+d_{\infty}(x_{i}',x_{\infty}^{\ast})+d^{Z}(\phi_{i}(x_{i}^{\ast}),\phi_{\infty}(x_{\infty}^{\ast}))\leq\Lambda
\]
for sufficiently large $i\in\mathbb{N}$. This implies the following
for $i\in\mathbb{N}$ sufficiently large: for any $D<\infty$, $r>0$,
and $y\in B^{X_{i}}(x_{i}^{\ast},r)$, we have 
\[
\mu_{i}(B^{X_{i}}(y,r))\geq c(r,D+\Lambda).
\]
Arguing as in Claim 1 thus gives
\[
\lim_{i\to\infty}\sup_{x\in B^{X_{i}}(x_{i}^{\ast},D)}d^{Z}(\phi_{i}(x),\phi_{\infty}(X_{\infty}))=0
\]
for any $D<\infty$. 

Now fix $D\in(0,\infty)$ and $\epsilon\in(0,D)$. For any $x\in B^{X_{i}}(x_{i}^{\ast},D)$,
because $(X_{i},d_{i})$ is a length space, we can find $x'\in B^{X_{i}}(x_{i}^{\ast},D-\epsilon/2)$
such that $d_{i}(x,x')<\epsilon/2$. We have shown that there exists
$x_{i}'\in X_{\infty}$ such that 
\[
d^{Z}(\phi_{i}(x'),\phi_{\infty}(y'))<\frac{\epsilon}{4}
\]
when $i=i(\epsilon,D)$ is sufficiently large. Thus, for $i=i(\epsilon,D)\in\mathbb{N}$
large, we have
\[
d_{\infty}(y',x_{\infty}^{\ast})\leq d^{Z}(\phi_{\infty}(y'),\phi_{i}(x'))+d_{i}(x',x_{i}^{\ast})+d^{Z}(\phi_{i}(x_{i}^{\ast}),\phi_{\infty}(x_{\infty}^{\ast}))<D,
\]
hence 
\[
\lim_{i\to\infty}\sup_{x\in B^{X_{i}}(x_{i}^{\ast},D)}d^{Z}\left(\phi_{i}(x),\phi_{\infty}(B^{X_{\infty}}(x_{\infty}^{\ast},D))\right)=0.
\]
Similarly, because $(X_{\infty},d_{\infty})$ is a length space, we
have 
\[
\lim_{i\to\infty}\sup_{x\in B^{X_{\infty}}(x_{\infty}^{\ast},D)}d^{Z}\left(\phi_{\infty}(x),\phi_{\infty}(B^{X_{i}}(x_{i}^{\ast},D))\right)=0.
\]
Together, these facts give 
\[
\lim_{i\to\infty}d_{H}\left(\phi_{i}(B^{X_{i}}(x_{i}^{\ast},D)),\phi_{\infty}(B^{X_{\infty}}(x_{\infty}^{\ast},D))\right)=0.
\]
The remaining claim follows from the distance bound $d_{i}(x_{i},x_{i}^{\ast})\leq\Lambda$. 
\end{proof}
\begin{defn}
We recall the definition of the quantitative singular strata of a
Riemannian manifold $(M^{n},g)$: $\mathcal{S}_{\eta,r}^{k}$ is the
set of $x\in M$ such that there does not exist $s\in[r,2]$ and a
metric cone $C(Z)$ such that
\[
d_{PGH}\left(\left(B(x,s),d_{g},x\right),\left(B\left((0^{k+1},z_{\ast}),s\right),d_{\mathbb{R}^{k+1}\times C(Z)},(0^{k+1},z_{\ast})\right)\right)<\eta s,
\]
where $z_{\ast}$ is the vertex of the cone $C(Z)$. Given a noncollapsed
Ricci limit space $(X,d)$, we let $\mathcal{S}^{k}(X)$ be the set
of $x\in X$ such that no tangent cone of $(X,d)$ at $x$ isometrically
splits a factor of $\mathbb{R}^{k+1}$, and let $\mathcal{S}(X)$
be the singular set, which consists of $x\in X$ such that no tangent
cone of $(X,d)$ at $x$ is $\mathbb{R}^{n}$. The regular set is
$\mathcal{R}(X):=X\setminus\mathcal{S}(X)$. 
\end{defn}

The following obtained from Theorem 1.7 in \cite{rectifiability}
by rescaling.
\begin{thm}
(Estimating the Size of Quantitative Singular Strata)\label{thm:singsetsize}
For any $\eta>0$ and $A<\infty$, there exists $C=C(A,\eta)>0$ such
that for any Riemannian manifold $(M^{n},g)$ satisfying $Rc(g)\geq-Ag$
and $|B_{g}(y,1)|_{g}\geq A^{-1}$ for all $y\in M$, we have 
\[
|\mathcal{S}_{\eta,sr}^{k}\cap B_{g}(y,r)|_{g}\leq Cs^{n-k}r^{n}
\]
for any $y\in M$ and $r,s\in(0,1]$. 
\end{thm}

\begin{rem}
In fact, if $\widetilde{g}:=r^{-2}g$ and $\widetilde{\mathcal{S}}_{\eta,s}^{k}$
denotes the quantitative singular set of the rescaled manifold $(M,\widetilde{g})$,
we just note that $\mathcal{S}_{\eta,sr}^{k}\subseteq\widetilde{\mathcal{S}}_{\eta,s}^{k}$.
\end{rem}

We will combine this with various Gromov-Hausdorff $\epsilon$-regularity
theorems as in \cite{cheegernaberquant} to get $L^{p}$ estimates
for $\widetilde{r}_{Rm}^{-2}$. 
\begin{prop}
\label{prop:codim1} (Codimension one $\epsilon$-Regularity) For
any $\underline{T}>0$, and $A<\infty$, there exists $\epsilon_{0}=\epsilon_{0}(\underline{T},A)>0$
such that the following holds for any closed Ricci flow $(M^{n},(g_{t})_{t\in[0,T)})$
with $T\geq\underline{T}$ satisfying $Rc(g_{t})\geq-Ag_{t}$ and
$|B(x,t,r)|_{g_{t}}\geq A^{-1}r^{n}$, for all $(x,t)\in M\times[0,T)$
and $r\in(0,1]$. For any $(x,t)\in M\times[\frac{T}{2},T)$ and $r\in(0,1]$,
if 
\[
d_{PGH}\left(\left(B(x,t,\epsilon_{0}^{-1}r),d_{g_{t}},x\right),\left(B\left((0^{n-1},z_{\ast}),\epsilon_{0}^{-1}r\right),d_{\mathbb{R}^{n-1}\times C(Z)},(0^{n-1},z_{\ast})\right)\right)<\epsilon_{0}r
\]
for some metric cone $C(Z)$, then $\widetilde{r}_{Rm}(x,t)\geq\epsilon_{0}r$. 
\end{prop}

\begin{proof}
Suppose the claim is false. Then, after rescaling and time translating,
we can find closed Ricci flows $(M_{i}^{n},(g_{t}^{i})_{t\in[-\frac{T_{i}}{2},0]})$
with $T_{i}\geq\underline{T}$ satisfying $Rc(g_{t}^{i})\geq-Ag_{t}^{i}$
and $|B_{g^{i}}(x,t,r)|_{g_{t}^{i}}\geq A^{-1}r^{n}$ for all $(x,t)\in M\times[-\frac{T_{i}}{2},0]$
and $r\in(0,1]$, and also $\epsilon_{i}\searrow0$, $x_{i}\in M_{i}$
such that 
\[
d_{PGH}\left(\left(B_{g^{i}}(x_{i},0,\epsilon_{i}^{-1}),d_{g_{t}^{i}},x_{i}\right),\left(B\left((0^{n-1},z_{\ast}^{i}),\epsilon_{i}^{-1}\right),d_{\mathbb{R}^{n-1}\times C(Z_{i})},(0^{n-1},z_{\ast}^{i})\right)\right)<\epsilon_{i}
\]
for some metric spaces $(Z_{i},d_{i})$, yet $\widetilde{r}_{Rm}^{g^{i}}(x_{i},0)<\epsilon_{i}$.
We can pass to a subsequence to assume that $(M_{i},d_{g_{0}^{i}},x_{i})$
converges in the pointed Gromov-Hausdorff sense to some metric cone
$(\mathbb{R}^{n-1}\times C(Z_{\infty}),d_{\mathbb{R}^{n-1}\times C(Z_{\infty})},(0^{n-1},z_{\ast}^{\infty}))$.
By Theorem 6.1 of \cite{cheegercolding1}, $\mathcal{S}(\mathbb{R}^{n-1}\times C(Z_{\infty}))$
has Hausdorff dimension $\leq n-2$, so $\mathcal{S}(C(Z_{\infty}))=\emptyset$,
hence $C(Z_{\infty})=\mathbb{R}$. By Theorem \ref{thm:openandsmooth},
the convergence is actually smooth everywhere, so $\lim_{i\to\infty}\widetilde{r}_{Rm}(x_{i},0)=\infty$,
a contradiction.
\end{proof}
\begin{rem}
\label{rem:remarque} Suppose that we replace the Gromov-Hausdorff
closeness assumption of Proposition \ref{prop:codim1} by
\[
d_{PGH}\left(\left(B_{g}(x,t_{0},r),d_{g_{t}},x\right),\left(B\left((0^{n-1},z_{\ast}),r\right),d_{\mathbb{R}^{n-1}\times C(Z)},(0^{n-1},z_{\ast})\right)\right)<\epsilon_{0}^{2}r
\]
for some $r\in(0,1]$. Then we can apply Proposition \ref{prop:codim1}
with $r$ replaced by $\epsilon_{0}r$, to get
\[
\widetilde{r}_{Rm}^{g}(x,t_{0})\geq\epsilon_{0}^{2}r.
\]
Thus the conclusion of \ref{prop:codim1} holds (after replacing $\epsilon_{0}$
with $\epsilon_{0}^{2}$) if we only ask for a Gromov-Hausdorff closeness
condition for $B_{g}(x,t,1)$ rather than $B_{g}(x,t,\epsilon_{0}^{-1})$. 
\end{rem}

\begin{cor}
\label{cor:highcurvaturesize} (Estimating the Size of High-Curvature
Regions) For any $p\in(0,2)$, $\underline{T}>0$, $D<\infty$, and
$A<\infty$, there exist $E=E(D,\underline{T},A)<\infty$ and $C=C(\underline{T},A,p,D)<\infty$
such that the following hold for any closed Ricci flow $(M^{n},(g_{t})_{t\in[0,T)})$
with $T\geq\underline{T}$ satisfying $Rc(g_{t})\geq-Ag_{t}$ and
$|B(x,t,r)|_{g_{t}}\geq A^{-1}r^{n}$, for all $(x,t)\in M\times[0,T)$
and $r\in(0,1]$. 

$(i)$ For all $(x,t)\in M\times[\frac{T}{2},T)$ and $s\in(0,1]$,
\[
|\{\widetilde{r}_{Rm}(\cdot,t)<s\}\cap B_{g}(x,t,D)|_{g_{t}}\leq Es^{2}.
\]

$(ii)$ For all $(x,t)\in M\times[\frac{T}{2},T)$,
\[
\int_{B(x,t,D)}|Rm|^{\frac{p}{2}}(\cdot,t)dg_{t}\leq\int_{B(x,t,D)}\widetilde{r}_{Rm}^{-p}(\cdot,t)dg_{t}\leq C.
\]
\end{cor}

\begin{proof}
$(i)$ We first assume $D=1$. Then, for any $y\in B_{g}(x,t,1)$
with $\widetilde{r}_{Rm}(y,t)<\epsilon_{0}s$, Proposition \ref{prop:codim1}
and Remark \ref{rem:remarque} give
\[
d_{PGH}\left(\left(B(x,t,r'),d_{g},x\right),\left(B\left((0^{n-1},z_{\ast}),r'\right),d_{\mathbb{R}^{n-1}\times C(Z)},(0^{n-1},z_{\ast})\right)\right)>\epsilon_{0}r'
\]
for all metric cones $C(Z)$ and all $r'\in[s,1]$. In other words,
$y\in\mathcal{S}_{\epsilon_{0},s}^{n-2}$, so Theorem \ref{cor:highcurvaturesize}
with $\eta:=\epsilon_{0}$ gives
\begin{align*}
|\{\widetilde{r}_{Rm}(\cdot,t)<s\}\cap B_{g}(x,t,1)|_{g_{t}}\leq & |\mathcal{S}_{\epsilon_{0},\epsilon_{0}^{-1}s}^{n-2}\cap B(x,t,1)|_{g_{t}}\leq C(A,\underline{T})s^{2}
\end{align*}
for $s\in(0,\epsilon_{0}]$. For $s\in(\epsilon_{0},1]$, we estimate
\[
|\{\widetilde{r}_{Rm}(\cdot,t)<s\}\cap B_{g}(x,t,1)|_{g}\leq C(A)\leq C(A,\underline{T})s^{2}.
\]
For arbitrary $D<\infty$, we can combine the $D=1$ case with a standard
covering argument to get
\[
|\{\widetilde{r}_{Rm}(\cdot,t)<s\}\cap B_{g}(x,t,D)|_{g_{t}}\leq C(A,\underline{T},D)s^{2}
\]
for all $s\in(0,1]$.

$(ii)$ For the remaining claim, we apply the first estimate to get
\begin{align*}
\int_{B(x,t,D)}|Rm|^{\frac{p}{2}}(\cdot,t)dg_{t}\leq & \int_{B(x,t,D)}\widetilde{r}_{Rm}^{-p}(\cdot,t)dg_{t}\\
\leq & |B(x,t,D)|_{g_{t}}+\frac{1}{p}\int_{1}^{\infty}s^{p-1}|B(x,t,D)\cap\{\widetilde{r}_{Rm}^{-1}(\cdot,t)>s\}|_{g_{t}}ds\\
\leq & C(\underline{T},A,p,D)\left(1+\int_{1}^{\infty}s^{p-1}s^{-2}ds\right)\\
\leq & C(\underline{T},A,p,D).
\end{align*}
\end{proof}
\begin{rem}
After possibly modifying the constants $E_{p},C$, Corollary \ref{cor:highcurvaturesize}
also holds with $\widetilde{r}_{Rm}$ replaced with $r_{Rm}$, using
the proof of part $(iii)$ of Lemma \ref{lem:basiclemma}. 
\end{rem}

\begin{rem}
\label{rem:epsregremark}The proof of Corollary \ref{cor:highcurvaturesize}
can be trivially modified to produce stronger estimates given stronger
$\epsilon$-regularity theorems.
\end{rem}

By \cite{bamlergen2,bamlergen3}, a metric flow arising as an $\mathbb{F}$-limit
of noncollapsed Ricci flow is characterized by its regular set, which
is endowed with the structure of a Ricci flow spacetime. We now introduce
some definitions and notation relevant to this notion.
\begin{defn}
(Definition 1.2 in \cite{klott1}) A Ricci flow spacetime is a tuple
$(\mathcal{R},g,\mathfrak{t},\partial_{\mathfrak{t}})$, where $\mathcal{R}$
is a smooth $(n+1)$-dimensional manifold, $\mathfrak{t}:\mathcal{R}\to\mathbb{R}$
is a smooth submersion, $\partial_{\mathfrak{t}}\in\mathfrak{X}(\mathcal{R})$
satisfies $\partial_{\mathfrak{t}}\mathfrak{t}=1$, and $g$ is a
bundle metric on $\ker(d\mathfrak{t})$ satisfying $\mathcal{L}_{\partial_{\mathfrak{t}}}g=-2Rc(g_{t})$,
where $g_{t}$ is Riemannian metric on $\mathcal{R}_{t}:=\mathfrak{t}^{-1}(t)$
obtained by restricting $g$. Given a point $x\in\mathcal{R}$, let
$\gamma:I_{x}\to\mathcal{R}$ denote the maximally defined integral
curve of $\partial_{\mathfrak{t}}$ satisfying $\gamma(\mathfrak{t}(x))=x$.
If $t\in(\alpha,\beta)$, then we say $x$ survives until time $t$,
and we write $x(t):=\gamma(t)$. For any subset $S\subseteq\mathcal{R}$,
write $S_{t}:=S\cap\mathcal{R}_{t}$, and if $S\subseteq\mathcal{R}_{t}$,
define 
\[
S(t'):=\{y(t');y\in S,t'\in I_{y}\}.
\]
Also define the parabolic neighborhood 
\[
P(x;A,-T^{-},T^{+}):=\cup_{t\in[\mathfrak{t}(x)-T^{-},\mathfrak{t}(x)+T^{+}]}\left(B_{g_{\mathfrak{t}(x)}}(x,A)\right)(t).
\]
We call $P(x;A,-T^{-},T^{+})$ unscathed if $B_{g_{\mathfrak{t}(x)}}(x,A)$
has compact closure in $\mathcal{R}_{t}$ and $[\mathfrak{t}(x)-T^{-},\mathfrak{t}(x)+T^{+}]\subseteq I_{y}$
for all $y\in B_{g_{\mathfrak{t}(x)}}(x,A)$. 
\end{defn}

According to \cite{bamlergen3}, special limits of noncollapsed Ricci
flows (in particular, static flows and metric solitons) are actually
determined by the restriction of the metric flow to a single time
slice, where the corresponding metric space has the structure of singular
space. We will thus often restrict our attention to studying the properties
of these singular spaces, so it is important to review the following
definitions.
\begin{defn}
(Definition 1.7 in \cite{bam1}) A singular space is a tuple $(X,d,\mathcal{R}_{X},g_{X})$,
where $(X,d)$ is a complete metric length space and $\mathcal{R}\subseteq X$
is a dense open subset equipped with the structure of an $n$-dimensional
Riemannian manifold $(\mathcal{R}_{X},g_{X})$ such that $d|(\mathcal{R}_{X}\times\mathcal{R}_{X})$
is the length metric $d_{g_{X}}$, and for any compact subset $K\subseteq X$
and $D<\infty$, there exist $0<\kappa_{1}(K,D)<\kappa_{2}(K,D)<\infty$
such that 
\[
\kappa_{1}r^{n}<|B^{X}(x,r)\cap\mathcal{R}_{X}|_{g_{X}}<\kappa_{2}r^{n}
\]
for all $x\in K$ and $r\in(0,D)$. 

A singular shrinking soliton is a tuple $(X,d,\mathcal{R}_{X},g_{X},f_{X})$,
where $(X,d,\mathcal{R}_{X},g_{X})$ is a singular space and $f_{X}\in C^{\infty}(\mathcal{R}_{X})$
satisfies the Ricci soliton equation
\[
Rc(g_{X})+\nabla^{2}f_{X}=\frac{1}{2}g_{X}
\]
on $\mathcal{R}_{X}$. 
\end{defn}

\section{Tangent Flows are Ricci-Flat Cones}

We first recall the notion of a conjugate heat kernel based at the
singular time. Suppose $(M^{n},(g_{t})_{t\in[0,T)},p)$ is any closed,
pointed Ricci flow. Let $K(x,t;\cdot,\cdot)$ denote the conjugate
heat kernel based at $(x,t)\in M\times(0,T)$. By Lemma 2.2 of \cite{mant},
for any sequence of times $t_{i}\nearrow T$, we can pass to a subsequence
so that $K(x,t_{i};\cdot,\cdot)$ converge in $C_{loc}^{\infty}(M\times[0,T))$
to a solution $K(x,T;\cdot,\cdot)$ of the conjugate heat equation
on $M\times[0,T)$ satisfying $\int_{M}K(x,T;y,t)dg_{t}(y)=1$ for
all $t\in[0,T)$. In particular, $d\nu_{x,T;t}:=K(x,T;\cdot,t)dg_{t}$
is a probability measure. By a slight abuse of language, we refer
to both $K(x,T;\cdot,\cdot)$ and $(\nu_{x,T;t})_{t\in[0,T)}$ as
a conjugate heat kernel at the singular time based at $x$. Note that
$K(x,T;\cdot,\cdot)$ is not unique, and may depend on the sequence
$(t_{i})$. 
\begin{lem}
\label{lem:heatkersingtime} $(i)$ $\lim_{i\to\infty}d_{W_{1}}^{g_{s}}(\nu_{x,t_{i};s},\nu_{x,T;s})$
for each $s\in[0,T)$.

$(ii)$ $\text{Var}(\nu_{x,T;s})\leq H_{n}(T-s)$ for all $s\in[0,T)$.
\end{lem}

\begin{proof}
$(i)$ In fact, for any $1$-Lipschitz function $f:(M,d_{g_{s}})\to\mathbb{R}$,
and any fixed point $y_{0}\in M$,
\begin{align*}
\int_{M}fd\nu_{x,t_{i};s}-\int_{M}fd\nu_{x,T;s}= & \int_{M}f(y)\left(K(x,t_{i};y,s)-K(x,T;y,s)\right)dg_{s}(y)\\
= & \int_{M}(f(y)-f(y_{0}))\left(K(x,t_{i};y,s)-K(x,T;y,s)\right)dg_{s}(y)\\
\leq & \left(\sup_{t\in[0,T)}\text{diam}_{g_{t}}(M)\right)\int_{M}|K(x,t_{i};y,s)-K(x,T;y,s)|dg_{s}(y),
\end{align*}
but the right hand side approaches 0 as $i\to\infty$ by the dominated
convergence theorem.

$(ii)$ Corollary 3.8 of \cite{bamlergen1} gives $\text{Var}(\nu_{x,t_{i};s})\leq H_{n}(t_{i}-s)$,
so the claim follows from the lower semicontinuity of $\text{Var}$
under $W_{1}$-convergence. 
\end{proof}
\begin{rem}
In Section 1.7 of \cite{bamlergen3}, Bamler defines $\nu_{x,T;s}$
as the limit in $W_{1}$ of $\nu_{x,t_{i};s}$ as $t_{i}\nearrow T$.
We choose to define $\nu_{x,T;s}$ via the smooth convergence of $K(x,t_{i};\cdot,\cdot)$
because we will need to pass the Gaussian heat kernel estimate stated
in Proposition \ref{prop:heatkernel} to the limit as $t_{i}\nearrow T$
in Section 4, and it is not immediate to us how to justify this using
only $W_{1}$-convergence.
\end{rem}

By assertion $(ii)$ of Lemma \ref{lem:heatkersingtime}, $(\nu_{x,T;t})_{t\in[0,T)}$
admits $H_{n}$-centers: for any $t\in[0,T)$, there exists $y\in M$
such that $\text{Var}(\nu_{x,T;t},\delta_{y})\leq H_{n}(T-t)$. The
following lemma shows that near such an $H_{n}$-center $y$, there
is a pointwise lower bound for $K(x,T;\cdot,t)$ on a set of almost
full measure in $B(y,t,1)$. 
\begin{lem}
\label{lem:firsttechlemma} Let $(M,(g_{t})_{t\in[-4,0)},(\nu_{t})_{t\in[-4,0)})$
be a closed Ricci flow satisfying $Rc(g_{t})\geq-Ag_{t}$ and $A^{-1}r^{n}<|B(x,t,r)|_{g_{t}}$
for all $(x,t)\in M\times[-4,0)$ and $r\in(0,1]$, where $\nu_{t}$
is a conjugate heat kernel based at the singular time. For any $D<\infty$
and $\delta>0$, there exists $\sigma=\sigma(A,D,\delta)>0$ such
that the following holds. For any $H_{n}$-center $(y,t)\in M\times[-2,-1]$
of $\nu$, there is a subset $S\subseteq B(y,t,D)$ such that $|B(y,t,D)\setminus S|_{g_{t}}<\delta$
and $K(\cdot,t)>\sigma$ on $S$, where $d\nu_{t}=K(\cdot,t)dg_{t}$. 
\end{lem}

\begin{proof}
\textbf{Step 1: }Find a small set in a future time-slice with controlled
curvature and conjugate heat kernel. 

Fix an $H_{n}$-center $(y,t)\in M\times[-2,-1]$ of $\nu$. Fix $\tau_{0}\in(0,\frac{1}{2})$
to be determined, and let $(z,t+\tau_{0})$ be an $H_{n}$-center
of $\nu$, so that 
\[
\nu_{t+\tau_{0}}(B(z,t+\tau_{0},\sqrt{4H_{n}}))\geq\frac{1}{2}.
\]
By assertion $(iii)$ of Lemma 2.6, for any $(x,t)\in M\times[-2,0]$
and $\tau\in(0,2]$, we can estimate $\mathcal{N}_{x,t}(\tau)\geq-Y$
for some $Y=Y(A)<\infty$. Thus Bamler's on-diagonal heat-kernel upper
bound (Theorem 7.1 of \cite{bamlergen1}) gives $K\leq C_{1}(A)$
on $M\times[-2,-\frac{1}{2}]$. Because (by volume comparison)
\[
|B(z,t+\tau_{0},\sqrt{4H_{n}})|_{g_{t+\tau_{0}}}\leq C_{2}(n,A),
\]
the subset $\widetilde{S}$ consisting of points $z'\in B(z,t+\tau_{0},\sqrt{4H_{n}})$
with $K(z',t+\tau_{0})>\frac{1}{4C_{2}}$ satisfies
\begin{align*}
\frac{1}{2}\leq & \int_{B(z,t+\tau_{0},\sqrt{4H_{n}})}K(\cdot,t+\tau_{0})dg_{t+\tau_{0}}\\
\leq & \int_{\widetilde{S}}K(\cdot,t+\tau_{0})dg_{t+\tau_{0}}+\frac{1}{4C_{2}}|B(z,t+\tau_{0},\sqrt{4H_{n}})\setminus\widetilde{S}|_{g_{t+\tau_{0}}}\\
\leq & C_{1}|\widetilde{S}|_{g_{t+\tau_{0}}}+\frac{1}{4},
\end{align*}
and in particular, $|\widetilde{S}|_{g_{t+\tau_{0}}^{i}}\geq c_{3}(A)>0$.
By Corollary \ref{cor:highcurvaturesize}, there exists $E_{0}:=E(\sqrt{4H_{n}},1,A)<\infty$
such that 
\[
|\{\widetilde{r}_{Rm}(\cdot,t+\tau_{0})<s\}\cap B(z,t+\tau_{0},\sqrt{4H_{n}})|_{g_{t+\tau_{0}}}\leq E_{0}s^{2}
\]
for all $s\in(0,1]$. In particular, we can choose $s_{0}=s_{0}(A)>0$
such that 
\[
\widehat{S}:=\{z'\in\widetilde{S};\widetilde{r}_{Rm}(z',t+\tau_{0})\geq s_{0}\}
\]
satisfies $|\widehat{S}|_{g_{t+\tau_{0}}}\geq\frac{1}{2}c_{3}$. By
Theorem \ref{thm:pseudolocality}, we can then modify $\tau_{0}=\tau_{0}(A)>0$
so that $\widetilde{r}_{Rm}(z',t')\geq\tau_{0}$ for all $(z',t')\in\widehat{S}\times[t,t+2\tau_{0}]$.
Now fix $z'\in\widehat{S}$, and apply Lemma 9.15 of \cite{bamlergen2}
to get 
\[
|\partial_{t}K|+|\nabla K|\leq C_{4}(A)
\]
on $B(z',t+\tau_{0},\frac{\tau_{0}}{2})\times[t,t+\tau_{0}]$. By
again modifying $\tau_{0}(A)$, and by standard distortion estimates,
we can integrate along geodesics in $B(z',t+\tau_{0},\frac{\tau_{0}}{2})$
emanating from $z'$, and then integrate backwards in time to conclude
that $K(z'',s)>(8C_{2})^{-1}$ for all $s\in[t,t+\tau_{0}]$ and $z''\in B(z',t+\tau_{0},\alpha)$,
where $\alpha=\alpha(A)>0$. 

Applying the volume lower bound once again, we get 
\[
\nu_{t}(B(z',t,\alpha))\geq\frac{A^{-1}\alpha^{\frac{n}{2}}}{4C_{2}}=:c_{5},
\]
where $c_{5}=c_{5}(A)>0$. On the other hand, we know
\[
\nu_{t}\left(B\left(y,t,\sqrt{4H_{n}c_{5}^{-1}}\right)\right)\geq1-\frac{c_{5}}{2},
\]
so that
\[
B\left(y,t,\sqrt{4H_{n}c_{5}^{-1}}\right)\cap B(z',t,\alpha)\neq\emptyset,
\]
hence $z'\in B(y,t,C_{6})$ for some $C_{6}=C_{6}(A)<\infty$. 

\noindent \textbf{Step 2: }Combine weak $L^{p}$ curvature scale estimates
with Colding's segment inequality to construct curves with controlled
$\mathcal{L}$-length.

Now apply Corollary \ref{cor:highcurvaturesize} to obtain $E:=E(D',1,A)<\infty$
such that
\[
|\{\widetilde{r}_{Rm}(\cdot,s)<\theta\}\cap B(y,s,D')|_{g_{s}}\leq E\theta^{2}
\]
for any $s\in[t,t+\tau_{0}]$, $D'<\infty$, and $\theta\in(0,1]$.
Fix $\tau_{1}\in(0,\frac{1}{2}\tau_{0})$ to be determined, and set
$D':=D'(A,D):=2C_{6}+8(D+\alpha)$. 

Let $\gamma_{y_{1},y_{2}}:[0,l_{y_{1},y_{2}}]\to M$ denote a unit-speed
minimizing geodesic from $y_{1}\in B(z',t,\alpha)$ to $y_{2}\in B(y,t,D)$
with respect to $g_{t+\tau_{1}}$. To deal with nonuniqueness, we
will only integrate over a set of $(y_{1},y_{2})$ with full $dg_{t+\tau_{1}}\otimes dg_{t+\tau_{1}}$-measure
such that there is a unique such $\gamma_{y_{1},y_{2}}$, and so that
$(y_{1},y_{2})\mapsto l_{y_{1},y_{2}}$ and $(y_{1},y_{2},u)\mapsto\gamma_{y_{1},y_{2}}(u)$
are smooth. For any such geodesic $\gamma_{y_{1},y_{2}}$, and any
$u\in[0,l_{y_{1},y_{2}}]$, we can estimate
\begin{align*}
d_{g_{t+\tau_{1}}}(\gamma_{y_{1},y_{2}}(u),y)\leq & d_{g_{t+\tau_{1}}}(y_{1},y_{2})+d_{g_{t+\tau_{1}}}(y_{2},y)\\
\leq & 4(D+\alpha)+d_{g_{t+\tau_{1}}}(z',y)+2D\leq D',
\end{align*}
We apply the Cheeger-Colding segment inequality (Theorem 2.11 of \cite{cheegercoldingwarped})
with respect to the time slice $g_{t+\tau_{1}}$, the sets 
\[
B(z',t,\alpha)\subseteq B(z',t+\tau_{1},2\alpha),
\]
\[
B(y,t,D)\subseteq B(y,t+\tau_{1},2D),
\]
and with the function $\chi_{W}$, where

\begin{align*}
W:= & \left\{ w\in B(y,t+\tau_{1},D');\widetilde{r}_{Rm}(w,t+\tau_{1})<\theta\right\} ,
\end{align*}
and $\theta\in(0,1]$ is to be determined, to get (since $d_{g_{t+\tau_{1}}}(y,z')\leq2C_{6}$)
\begin{align*}
\int_{B(z',t,\alpha)\times B(y,t,D)}|\{u & \in[0,l_{y_{1},y_{2}}];\widetilde{r}_{Rm}(\gamma_{y_{1},y_{2}}(u),t+\tau_{1})<\theta\}|dg_{t+\tau_{1}}(y_{1})dg_{t+\tau_{1}}(y_{2})\\
\leq & C_{7}(A,D')\left(|B(z',t,\alpha)|_{g_{t+\tau_{1}}}+|B(y,t,D)|_{g_{t+\tau_{1}}}\right)\int_{B(y,t+\tau_{1},D')}\chi_{W}dg_{t+\tau_{1}}\\
\leq & C_{7}(A,D')\left(|B(z',t+\tau_{1},2\alpha)|_{g_{t+\tau_{1}}}+|B(y,t+\tau_{1},D)|_{g_{t+\tau_{1}}}\right)|W|_{g_{t+\tau_{1}}}\\
\leq & C_{8}(A,D)\theta^{2}.
\end{align*}
We now define 
\[
S:=\left\{ y_{2}\in B(y,t,D);\int_{B(z',t,\alpha)}|\{u\in[0,l_{y_{1},y_{2}}];\widetilde{r}_{Rm}(\gamma_{y_{1},y_{2}}(u),t+\tau_{1})<\theta\}|dg_{t+\tau_{1}}(y_{1})\leq\theta^{\frac{3}{2}}\right\} .
\]
Then 
\begin{align*}
\theta^{\frac{3}{2}}|B(y,t,D)\setminus S|_{g_{t+\tau_{1}}\hfill}\\
\leq\int_{B(z',t,\alpha)\times(B(y,t,D)\setminus S)} & |\{u\in[0,l_{y_{1},y_{2}}];\widetilde{r}_{Rm}(\gamma_{y_{1},y_{2}}(u),t+\tau_{1})<\theta\}|dg_{t+\tau_{1}}(y_{1})dg_{t+\tau_{1}}(y_{2})\\
\leq C_{8}\theta^{2}\hfill\qquad\,\qquad\qquad
\end{align*}
implies $|B(y,t,D)\setminus S|_{g_{t+\tau_{1}}}\leq C_{8}\theta^{\frac{1}{2}}$.
For any $y_{2}\in S$, the subset $S_{y_{2}}$ consisting of $y_{1}\in B_{g^{i}}(z',t,\alpha)$
such that 
\[
|\{u\in[0,l_{y_{1},y_{2}}];\widetilde{r}_{Rm}(\gamma_{y_{1},y_{2}}(u),t+\tau_{1})<\theta\}|\leq\frac{1}{4}\theta
\]
must satisfy 
\[
\frac{1}{4}\theta|B(z',t,\alpha)\setminus S_{y_{2}}|_{g_{t+\tau_{1}}}\leq\int_{B(z',t,\alpha)}|\{u\in[0,l_{y_{1},y_{2}}];\widetilde{r}_{Rm}(\gamma_{y_{1},y_{2}}(u),t+\tau_{1})<\theta\}|dg_{t+\tau_{1}}(y)\leq\theta^{\frac{3}{2}},
\]
or equivalently, $|B(z',t,\alpha)\setminus S_{y_{2}}|_{g_{t+\tau_{1}}}\leq4\theta^{\frac{1}{2}}$.
For any $y_{2}\in S$ and $y_{1}\in S_{y_{2}}\setminus B(y_{2},t+\tau_{1},\theta)$,
if $\widetilde{r}_{Rm}(\gamma_{y_{1},y_{2}}(u),t+\tau_{1})<\frac{1}{2}\theta$
for some $u\in[0,l_{y_{1},y_{2}}]$, then because $\widetilde{r}_{Rm}(\cdot,t+\tau_{1})$
is 1-Lipschitz with respect to $g_{t+\tau_{1}}$, we get 
\[
|\{u\in[0,l_{y_{1},y_{2}}];\widetilde{r}_{Rm}(\gamma_{y_{1},y_{2}}(u),t+\tau_{1})<\theta\}|>\frac{1}{2}\theta,
\]
a contradiction. Assume we have chosen $\theta<\min\{\alpha,\tau_{0}\}$.
Then, because $\widetilde{r}_{Rm}(\cdot,t+\tau_{1})\geq\tau_{0}$
on $B(z',t+\tau_{1},\alpha)$, we know that if $y_{1}\in S_{y_{2}}\cap B(y_{2},t+\tau_{1},\theta_{\ast})$,
we must have 
\[
\gamma_{y_{1},y_{2}}([0,l_{y_{1},y_{2}}])\subseteq B(z',t+\tau_{1},\alpha)\subseteq\{\widetilde{r}_{Rm}(\cdot,t+\tau_{1})\geq\theta\}.
\]
In either case, we conclude that $\widetilde{r}_{Rm}(\cdot,t+\tau_{1})>\frac{1}{2}\theta$
along $\gamma_{y_{1},y_{2}}$ for any $y_{2}\in S_{1}$, $y_{1}\in S_{y_{2}}$.
Let $\epsilon_{P}=\epsilon_{P}(A)>0$ be as in Theorem \ref{thm:pseudolocality},
and assume we have chosen $\tau_{1}=\tau_{1}(A,D,\theta)\in(0,\epsilon_{P}^{2}\theta^{4})$.
Then, for any $y_{2}\in S$, $y_{1}\in S_{y_{2}}$, $u\in[0,l_{y_{1},y_{2}}]$,
and $s\in[t,t+\tau_{1}]$, we have
\[
|Rm|(\gamma_{y_{1},y_{2}}(u),s)\leq\epsilon_{P}^{-2}\theta^{-2}.
\]
In particular, we can estimate
\begin{align*}
\left|\partial_{s}\log|\dot{\gamma}_{y_{1},y_{2}}(u)|_{g_{s}}^{2}\right|\leq & 2n\epsilon_{P}^{-2}\theta^{-2},
\end{align*}
so that integration in time and $|\dot{\gamma}_{y_{1},y_{2}}(u)|_{g_{t+\tau_{1}}}^{2}=1$
give $|\dot{\gamma}_{y_{1},y_{2}}(u)|_{g_{s}}\leq e^{n\epsilon_{P}^{-2}}$
for all $s\in[t,t+\tau_{1}]$. 

For the moment, fix $y_{2}\in S$ and $y_{1}\in S_{y_{2}}$, and set
$\eta(r):=\gamma_{y_{1},y_{2}}(\tau_{1}^{-1}l_{y_{1},y_{2}}r)$ for
$t\in[0,\tau_{1}]$. Because $l_{y_{1},y_{2}}\leq D'$, we have the
reduced length estimate
\begin{align*}
\ell_{(y_{1},t+\tau_{1})}(y_{2},t)\leq & \frac{1}{2\sqrt{\tau_{1}}}\int_{0}^{\tau_{1}}\sqrt{r}\left(R(\eta(r),t+\tau_{1}-r)+|\dot{\eta}(r)|_{g_{t+\tau_{1}-r}}^{2}\right)dr\\
\leq & C(n)\epsilon_{P}^{-2}\theta^{-2}+e^{2n\epsilon_{P}^{-2}}\int_{0}^{\tau_{1}}\frac{l_{y_{1},y_{2}}^{2}}{\tau_{1}^{2}}dr\leq C_{9}(A,D,\theta,\tau_{1}).
\end{align*}
Because $\widetilde{r}_{Rm}(z'',s)\geq\tau_{0}$ for all $z''\in B(z',t,\alpha)$,
we can integrate $\partial_{s}dg_{s}|_{z''}\geq-c(n)\tau_{0}^{-2}dg_{s}|_{z''}$
from $t$ to $t+\tau_{1}$ to get
\[
|B(z',t,\alpha)|_{g_{t+\tau_{1}}}\geq c_{10}(A).
\]
Then, for $y_{2}\in S$ fixed, we integrate over $y_{1}\in S_{y_{2}}$
and use the integrated form of Perelman's differential Harnack inequality
(see Proposition 16.54 of \cite{chowbook2}) to get a lower bound
for $K$ at $(y_{2},t)$:
\begin{align*}
K(y_{2},t)\geq & (4\pi\tau_{1})^{-\frac{n}{2}}e^{-C_{9}}\int_{S_{y_{2}}}K(y',t+\tau_{1})dg_{t+\tau_{1}}(y')\\
\geq & c_{11}(A,D,\theta,\tau_{1})\left(|B(z',t,\alpha)|_{g_{t+\tau_{1}}}-|B(z',t,\alpha)\setminus S_{y_{2}}|_{g_{t+\tau_{1}}}\right)\\
\geq & c_{11}(A,D,\theta,\tau_{1})\left(c_{10}(A)-4\theta^{\frac{1}{2}}\right)=:\sigma>0.
\end{align*}
assuming we have chosen $\theta\leq(\frac{1}{8}c_{10}(n,A))^{2}$. 

Finally, we have $\partial_{s}dg_{s}|_{z''}\ge-\frac{c(n)}{\tau_{1}}dg_{s}|_{z''}$
for all $z''\in M$, $s\in[t,t+\tau_{1}]$ with $\widetilde{r}_{Rm}(z'',t)\geq\epsilon_{P}^{-1}\sqrt{\tau_{1}}$,
so we can integrate from $t$ to $t+\tau_{1}$ to obtain $dg_{t+\tau_{1}}\geq e^{-c(n)}dg_{t}$
for such $z''$, hence
\begin{align*}
|B(y,t,D)\setminus S|_{g_{t}}\leq & \left|\left(B(y,t,D)\setminus S\right)\cap\{\widetilde{r}_{Rm}(\cdot,t)\geq\epsilon_{P}^{-1}\sqrt{\tau_{1}}\}\right|_{g_{t}}\\
 & +|B(y,t,D)\cap\{\widetilde{r}_{Rm}(\cdot,t)<\epsilon_{P}^{-1}\sqrt{\tau_{1}}\}|_{g_{t}}\\
\leq & e^{c(n)}|B(y,t,D)\setminus S|_{g_{t+\tau_{1}}}+E(\epsilon_{P}^{-2}\sqrt{\tau_{1}})^{2}\\
\leq & C_{8}e^{c(n)}\theta^{\frac{1}{4}}+C(A,D)\tau_{1}.
\end{align*}
The claim follows by taking $\theta=\theta(n,A,D,\delta)>0$, then
(since $\tau_{1}$ depends on $\theta$) $\tau_{1}=\tau_{1}(A,D,\delta)>0$
sufficiently small. 
\end{proof}
Now let $(M_{i},(g_{t}^{i})_{t\in[-4,0)},(\nu_{t}^{i})_{t\in[-4,0)})$
be a sequence of Ricci flow solutions satisfying $Rc(g_{t}^{i})\geq-Ag_{t}^{i}$
and $|B_{g^{i}}(x,t,r)|_{g_{t}^{i}}>A^{-1}r^{n}$ for all $(x,t)\in M_{i}\times[-4,0)$
and $r\in(0,1]$, where $\nu_{t}^{i}=K^{i}(\cdot,t)dg_{t}^{i}$ are
conjugate heat kernels based at the singular time (if $(g_{t}^{i})_{t\in[-4,0)}$
does not develop a singularity at time $t=0$, these are just the
usual conjugate heat kernel based at some points $x_{i}\in M_{i}$).
By Theorem 1.38 of \cite{bamlergen3}, we can pass to a subsequence
to obtain $\mathbb{F}$-convergence within some correspondence $\mathfrak{C}$:

\noindent 
\[
(M,(g_{t}^{i})_{t\in(-4,0)},(\nu_{t}^{i})_{t\in(-4,0)})\xrightarrow[i\to\infty]{\mathbb{F},\mathfrak{C}}(\mathcal{Y},(\mu_{t}^{\infty})_{t\in(-4,0)}),
\]
where $\mathcal{Y}$ is an $H_{n}$-concentrated, future-continuous
metric flow of full support.
\begin{lem}
\label{lem:gromovhaus} Let $(y_{i},t)$ be $H_{n}$-centers of $\nu^{i}$
for some fixed time $t\in[-2,-1]$ where the $\mathbb{F}$-convergence
is timewise. Then there exist $\Lambda<\infty$, $y_{i}'\in B_{g^{i}}(y_{i},t,\Lambda)$,
and $y_{\infty}'\in\mathcal{Y}_{t}$ such that, after passing to a
subsequence,
\[
(M_{i},d_{g_{t}^{i}},y_{i}')\to(\mathcal{Y}_{t},d_{g_{t}^{\infty}},y_{\infty}')
\]
in the pointed Gromov-Hausdorff sense. By passing to a subsequence,
we may also find $y_{\infty}\in\mathcal{Y}_{t}$ such that 
\[
(M_{i},d_{g_{t}^{i}},y_{i})\to(\mathcal{Y}_{t},d_{g_{t}^{\infty}},y_{\infty})
\]
in the pointed Gromov-Hausdorff sense. 
\end{lem}

\begin{proof}
Fix $D<\infty$ and $r\in(0,1]$. Fix $\delta>0$ to be determined,
and apply Lemma \ref{lem:firsttechlemma} to obtain subsets $S_{i}\subseteq B_{g^{i}}(y_{i},t,D+1)$
such that 
\[
|B_{g^{i}}(y_{i},t,D+1)\setminus S_{i}|_{g_{t}^{i}}<\delta
\]
and $\sigma=\sigma(A,D,\delta)>0$ such that $K^{i}(\cdot,t)>\sigma$
on $S_{i}$. For any $z\in B_{g^{i}}(y_{i},t,D+1)$ and $r\in(0,1]$,
we estimate
\begin{align*}
|B_{g^{i}}(z,t,r)\cap S_{i}|_{g_{t}^{i}}\geq & |B_{g^{i}}(z,t,r)|_{g_{t}^{i}}-|B_{g^{i}}(y_{i},t,D+1)\setminus S_{i}|_{g_{t}^{i}}\\
\geq & A^{-1}r^{n}-\delta,
\end{align*}
so if we choose $\delta:=\frac{1}{2}A^{-1}r^{n}$, then 
\[
\nu^{i}(B_{g^{i}}(z,t,r))=\int_{B_{g^{i}}(z,t,r)}K^{i}(\cdot,t)dg_{t}^{i}\geq\sigma|B_{g^{i}}(z,t,r)\cap S_{i}|_{g_{t}^{i}}\geq\sigma',
\]
where $\sigma'=\sigma'(A,D,r)>0$. We may therefore apply Proposition
\ref{prop:easyconvergence} along with the fact that time-wise convergence
at time $t$ implies 
\[
(M_{i},d_{g_{t}^{i}},\nu_{t}^{i})\to(\mathcal{Y},d_{g_{t}^{\infty}},\nu_{t}^{\infty})
\]
in the $W_{1}$-Gromov-Wasserstein sense.
\end{proof}
\begin{lem}
\label{lem:niargument} Suppose that $(X,d,\mathcal{R},g,f)$ is a
singular shrinking GRS with $Rc(g)\geq0$ on $\mathcal{R}$, corresponding
to a tangent flow of a smooth, closed Ricci flow at the singular time.
Also assume that there are closed Ricci flows $(M_{i},(g_{t}^{i})_{t\in[-4,0]},p_{i})$
satisfying $Rc(g_{t}^{i})\geq-Ag_{t}^{i}$ and $|B_{g^{i}}(x,t,r)|_{g_{t}^{i}}\geq A^{-1}r^{n}$
for all $(x,t)\in M_{i}\times[-4,0]$ and $r\in(0,1]$, such that
$(M_{i},d_{g_{0}^{i}},p_{i})\to(X,d,p)$ in the pointed Gromov-Hausdorff
sense for some $p\in\mathcal{R}$ and $t'\in[-2,-1]$. 

\noindent If $(\mathcal{R},g)$ is not Ricci flat, then $\inf_{\mathcal{R}}R(g)>0$. 
\end{lem}

\begin{rem}
Both the statement and the proof of this are modifications of Proposition
1.1 in \cite{leinisoliton}. The main technical difficulty is showing
that the integral curve of $\nabla f$ starting at $x\in\mathcal{R}$
is complete for almost-every $x\in\mathcal{R}$. To establish this,
we argue similarly to Claim 2.32 of \cite{chenwang2a}. The argument
of Claim 2.32 used $\nabla^{2}f=0$ to establish estimates for the
distortion of the volume form along the gradient flow of $\nabla f$.
We no longer have this estimate, but $Rc\geq0$ along with $R\leq f$
and a locally uniform upper bound for $f$ tell us that $|Rc|$ is
locally bounded on $\mathcal{R}$, and it turns out this is enough
to make the argument work. 

We observe that the proof is a trivial modification of Ni's when $n=4$
since the regular set $\mathcal{R}$ is convex and all orbifold points
are critical points for $f$, hence the gradient flow is complete.
\end{rem}

\begin{proof}
Suppose $(\mathcal{R},g)$ is not Ricci flat, so that $R>0$ on $\mathcal{R}$
by Theorem 1.19 of \cite{bamlergen3}. Because $\mathcal{R}$ is connected,
we can add a constant to $f$ to assume that $R+|\nabla f|^{2}=f$.
Write $r:=d(\cdot,p)$. Integrating $|\nabla\sqrt{f}|\leq\frac{1}{2}$
along almost-minimizing curves in $\mathcal{R}$ from $p$ to $x\in\mathcal{R}$
gives $\sqrt{f(x)}\leq\sqrt{f(p)}+\frac{1}{2}r(x)$, so that
\[
f(x)\leq\frac{1}{2}r^{2}(x)+2f(p)
\]
for all $x\in\mathcal{R}$. By Theorem 3.7 of \cite{cheegercolding2}
(here we are also using that $\mathcal{R}$ is open by Theorem \ref{thm:openandsmooth})
and the proof of Proposition 2.3$(c)$ in \cite{bam1}, there is an
open subset $\mathcal{G}^{\ast}\subseteq\mathcal{R}$ of full measure
such that, for any $x\in\mathcal{G}^{\ast}$, there is a unique minimizing
geodesic of $(X,d)$ from $p$ to $x$ that lies entirely in $\mathcal{R}$.
Given $x\in\mathcal{G}$, let $\gamma:[0,l]\to\mathcal{R}$ be such
a unit-speed arclength minimizing geodesic, where $l:=r(x)$. Let
$(E_{i})_{i=1}^{n}$ be an orthonormal frame at $p$ with $E_{n}:=\dot{\gamma}(0)$,
and let $E_{i}\in\mathfrak{X}(\gamma)$ be the corresponding parallel
translations along $\gamma$. If $l\geq2$, then for any $r_{0}\in[0,1]$,
we define 
\[
Y_{i}(s):=\begin{cases}
sE_{i}(s), & s\in[0,1]\\
E_{i}(s), & s\in[1,l-r_{0}]\\
\frac{l-s}{r_{0}}E_{i}(s) & s\in[l-r_{0},l]
\end{cases}
\]
for $i=1,...,n-1$. Because $\gamma$ is minimizing, hence stable,
we have
\begin{align*}
0\leq & \int_{0}^{l}\left(|\nabla_{\dot{\gamma}}Y_{i}|^{2}-R(Y_{i},\dot{\gamma},\dot{\gamma},Y_{i})\right)ds\\
= & \int_{0}^{1}\left(1-s^{2}R(E_{i},\dot{\gamma},\dot{\gamma},E_{i})\right)ds-\int_{1}^{l-r_{0}}R(E_{i},\dot{\gamma},\dot{\gamma},E_{i})ds\\
 & +\int_{l-r_{0}}^{l}\left(\frac{1}{r_{0}^{2}}-\left(\frac{l-s}{r_{0}}\right)^{2}R(E_{i},\dot{\gamma},\dot{\gamma},E_{i})\right)ds,
\end{align*}
so we can sum to obtain (using $Rc\geq0$)
\begin{align*}
0\leq & n-1+\frac{n-1}{r_{0}}+\int_{0}^{1}(1-s^{2})Rc(\dot{\gamma},\dot{\gamma})ds-\int_{0}^{l-r_{0}}Rc(\dot{\gamma},\dot{\gamma})ds-\int_{l-r_{0}}^{l}\left(\frac{l-s}{r_{0}}\right)^{2}Rc(\dot{\gamma},\dot{\gamma})ds\\
\leq & C\left(n,\sup_{B(p,1)\cap\mathcal{R}}|Rc|\right)-\int_{0}^{l-r_{0}}Rc(\dot{\gamma},\dot{\gamma})ds+\frac{n-1}{r_{0}}.
\end{align*}
Moreover, we know $Rc\geq0$ and $R\leq f\leq\frac{1}{2}r^{2}(x)+2f(p)$
on $\mathcal{R}$, so combining these gives
\[
\sup_{B(p,1)\cap\mathcal{R}}|Rc|\leq\sup_{B(p,1)\cap\mathcal{R}}R\leq\frac{1}{2}+2f(p)<\infty
\]

\noindent \textbf{Claim 1: }There exist $C<\infty$ and $\Lambda<\infty$
such that if $R(x)\leq1$ and $l\geq\Lambda$, then 
\[
\int_{0}^{l}Rc(\dot{\gamma},\dot{\gamma})ds\leq\frac{l}{4}+C.
\]

Choose $\Lambda>1$ such that $f(x)\leq r^{2}(x)$ for $x\in\mathcal{G}^{\ast}\setminus B(p,\Lambda)$.
Since $R+|\nabla f|^{2}=f$, this implies $|\nabla f(x)|\leq r(x)$
for $x\in\mathcal{G}^{\ast}\setminus B(p,\Lambda)$. Thus
\[
|\nabla R|^{2}=4|Rc(\nabla f)|^{2}\leq4R^{2}|\nabla f|^{2}.
\]
Combining estimates gives $|\nabla\log R|\leq2|\nabla f|\leq2r$ on
$\mathcal{G}^{\ast}\setminus B(p,\Lambda)$. Set $r_{0}:=\min\{\frac{4(n-1)}{l},1\}$,
so that if $l\geq2\Lambda$, then for $s\in[l-r_{0},l]$,
\[
\log\left(\dfrac{R(\gamma(s))}{R(x)}\right)\leq2lr_{0}\leq2(n-1),
\]
hence $R(\gamma(s))\leq e^{2(n-1)}R(x)$ for $s\in[l-r_{0},l]$. Combining
estimates gives 
\begin{align*}
\int_{0}^{l}Rc(\dot{\gamma},\dot{\gamma})ds\leq & \int_{0}^{l-r_{0}}Rc(\dot{\gamma},\dot{\gamma})ds+\int_{l-r_{0}}^{l}R(\gamma(s))ds\\
\leq & C(g,p)+\frac{n-1}{r_{0}}+e^{n-1}r_{0}R(x)\leq C(g,p)+\frac{l}{4}.\text{\ensuremath{\square}}
\end{align*}
Whenever $r(x)=l\geq\Lambda$ and $R(x)\leq1$, we can thus estimate

\begin{align*}
\langle\nabla f(x),\dot{\gamma}(l)\rangle-\langle\nabla f(p),\dot{\gamma}(0)\rangle= & \int_{0}^{l}\nabla^{2}f(\dot{\gamma},\dot{\gamma})ds=\int_{0}^{l}\left(\frac{1}{2}-Rc(\dot{\gamma},\dot{\gamma})\right)ds\\
\geq & \frac{l}{2}-C-\frac{l}{4}\geq\frac{l}{8},
\end{align*}
so that,
\[
\langle\nabla f(x),\dot{\gamma}(l)\rangle\geq\left(\frac{1}{8}r(x)-|\nabla f(p)|\right),
\]
which implies that $\nabla f(x)\neq0$ if in addition $r(x)>8|\nabla f(p)|$.

By Theorem 1.19 of \cite{bamlergen3}, $\mathcal{S}$ has Minkowski
dimension 4, which implies that for any $D<\infty$, there exists
$E=E(A,D)<\infty$ such that 
\begin{equation}
|\{d(\cdot,\mathcal{S})<s\}\cap B(p,D)\cap\mathcal{R}|_{g}\leq Es^{\frac{7}{2}}\label{eq:mink4}
\end{equation}
for all $s\in(0,1]$. 

\noindent \textbf{Claim 2: }There is a Borel subset $\mathcal{G}'\subseteq\mathcal{G}^{\ast}$
of full measure such that for any $x\in\mathcal{G}'$, the integral
curve of $\nabla f$ through $x$ exists for all time.

Let $(\varphi_{t})$ be the (partially defined) flow corresponding
to $\nabla f$. We first observe that because $|\nabla f|$ is locally
bounded, the escape lemma for ODEs guarantees that $\varphi_{t}(x)$
exists for all time unless $t\mapsto\varphi_{t}(x)$ has a limit point
in $\mathcal{S}$. For each $D<\infty$ and $s\in(0,1]$, define
\[
S_{D,s}:=\left\{ x\in\mathcal{R}\cap B(p,D);d(\varphi_{t}(x),\mathcal{S})<\frac{1}{2}s\text{ for some }t\in[-D,D]\right\} .
\]
Because $d(\cdot,\mathcal{S})$ is 1-Lipschitz, we can find $h\in C^{\infty}(\mathcal{R})$
such that $|\nabla h|\leq2$ and $\frac{1}{2}d(\cdot,\mathcal{S})<h<2d(\cdot,\mathcal{S})$
on $\mathcal{R}$. For any $D'<\infty$, (\ref{eq:mink4}) gives $E=E(A,D')<\infty$
such that 
\[
|\{h<2s\}\cap B(p,D')\cap\mathcal{R}|_{g}\leq Es^{\frac{7}{2}}
\]
for all $s\in(0,1]$. Thus, by the coarea formula, 
\begin{align*}
\int_{s}^{2s}\mathcal{H}^{n-1}(h^{-1}(t)\cap B(p,D'))dt= & \int_{\{s\leq h\leq2s\}\cap B(p,D')\cap\mathcal{R}}|\nabla h|dg\\
\leq & 2|\{h<2s\}\cap B(p,D')\cap\mathcal{R}|_{g}\leq2Es^{\frac{7}{2}}.
\end{align*}
By Sard's theorem, for any $s\in(0,1]$, we can find $t=t(s)\in(s,2s)$
such that $\Sigma_{s}:=h^{-1}(t)\cap B(p,D')$ is smooth and $\mathcal{H}^{n-1}(\Sigma_{s})\leq4Es^{\frac{5}{2}}$.
Write $S_{D,s}:=I_{s}\cup II_{s}$, where 
\[
I_{s}:=\{x\in S_{D,s};d(x,\mathcal{S})\leq4s\},
\]
\[
II_{s}:=\{x\in S_{D,s};d(x,\mathcal{S})>4s\},
\]
so that $|I_{s}|_{g}\leq Es^{\frac{7}{2}}$. For any $x\in II_{s}$,
there exists $t\in(-D,D)$ such that $\varphi_{t}(x)\in\Sigma_{s}$.
Because $|\nabla f|\leq\sqrt{f}\leq C(r+1)$, we have
\[
\left|\frac{d}{dt}r(\varphi_{t}(x))\right|=|\langle\nabla r,\nabla f\rangle|(\varphi_{t}(x))\leq C\left(r(\varphi_{t}(x))+1\right)
\]
for almost every $t$, so we can find $D'=D'(D)<\infty$ such that
$\varphi_{t}(x)\in B(p,D')$ for all $x\in II_{s}$ and $t\in(-D,D)$
such that $\varphi_{t}(x)$ exists. Set
\[
\Omega_{s}:=\{(t,x)\in(-D',D')\times\Sigma_{s};\varphi_{t}(x)\text{ is well defined\}},
\]
which is open in $(-D,D)\times\Sigma_{s}$, and define $\eta:\Omega_{s}\to\mathcal{R},(t,x)\mapsto\varphi_{t}(x).$
For any $x\in\Sigma_{s}$ and $v\in T_{x}\Sigma_{s}$, we have (since
$|Rc|\leq C(D)$ on $\eta(\Omega_{s})$)
\begin{align*}
\left|\frac{d}{dt}(\eta^{\ast}g)_{(t,x)}(v,v)\right|= & \left|\frac{d}{dt}(\varphi_{t}^{\ast}g)_{x}(v,v)\right|\\
= & \left|(\mathcal{L}_{\nabla f}g)(d(\varphi_{t})_{x}v,d(\varphi_{t})_{x}v)\right|\\
= & \left|2\nabla^{2}f_{\varphi_{t}(x)}(d(\varphi_{t})_{x}v,d(\varphi_{t})_{x}v)\right|\\
\leq & 2|\nabla^{2}f|(\varphi_{t}(x))|g_{\varphi_{t}(x)}(d(\varphi_{t})_{x}v,d(\varphi_{t})_{x}v)|\\
\leq & C(D)\left|(\eta^{\ast}g)_{(t,x)}(v,v)\right|
\end{align*}
which we can integrate in $t$ to obtain $(\eta^{\ast}g)_{(t,x)}(v,v)\leq C(D)$
for $|v|_{g_{\Sigma_{s}}}=1$, where $g_{\Sigma_{s}}$ is the restriction
of $g$ to $\Sigma_{s}$. Also,
\[
(\eta^{\ast}g)_{(t,x)}(\partial_{t},\partial_{t})=|\nabla f|_{\varphi_{t}(x)}^{2}\leq C(D)
\]
so the Jacobian $\mathcal{J}=\frac{\eta^{\ast}dg}{dt\wedge dg_{\Sigma_{s}}}$
satisfies $|\mathcal{J}|\leq C(D)$. Because $II_{s}\subseteq\eta(\Omega_{s})$,
we thus obtain 
\[
|II_{s}|_{g}\leq\int_{\Sigma_{s}\times[-D,D]}d(\eta^{\ast}g)\leq\int_{-D}^{D}\int_{\Sigma_{s}}C(D)dg_{\Sigma_{s}}dt\leq C(D)\mathcal{H}^{n-1}(\Sigma_{s})\leq C(A,D)s^{\frac{5}{2}}.
\]
Combine estimates to get $|S_{D,s}|\leq C(A,D)s^{\frac{5}{2}}$. By
taking $s\searrow0$, we get that the set $S_{D}$ of $x\in B(p,D)$
such that $\varphi_{t}(x)$ is undefined for some $t\in(-D,D)$ has
measure zero. Now taking $D\nearrow\infty$, we see that the set of
$x\in\mathcal{R}$ such that $\varphi_{t}(x)$ is undefined for some
$t\in\mathbb{R}$ has measure zero. $\square$

Next, we observe that $\mathcal{G}'\subseteq\mathcal{R}$ is a set
of full measure which is preserved by the flow $(\varphi_{t})$. If
$\mathcal{D}\subseteq\mathbb{R}\times\mathcal{R}$ is the (open) maximal
flow domain, then 
\[
\xi:\mathcal{D}\to\mathbb{R}\times\mathcal{R},\quad(t,x)\mapsto(t,\varphi_{t}(x))
\]
is a diffeomorphism onto its (open) image. Note that $\mathbb{R}\times(\mathcal{R}\setminus\mathcal{G}^{\ast})$
has measure zero in $\mathbb{R}\times\mathcal{R}$, hence 
\[
\xi^{-1}(\mathbb{R}\times(\mathcal{R}\setminus\mathcal{G}^{\ast}))=\{(t,x)\in\mathcal{D};\varphi_{t}(x)\in\mathcal{R}\setminus\mathcal{G}^{\ast}\}
\]
has measure zero in $\mathcal{D}$. In particular, 
\[
\{(t,x)\in\mathbb{R}\times\mathcal{G}';\varphi_{t}(x)\in\mathcal{R}\setminus\mathcal{G}^{\ast}\}\subseteq\xi^{-1}(\mathbb{R}\times(\mathcal{R}\setminus\mathcal{G}'))
\]
has measure zero in $\mathbb{R}\times\mathcal{G}'$. By Fubini's theorem,
we may conclude that the set 
\[
\mathcal{G}'':=\left\{ x\in\mathcal{G}';|\{t\in\mathbb{R};\varphi_{t}(x)\in\mathcal{R}\setminus\mathcal{G}^{\ast}\}|=0\right\} 
\]
has full measure in $\mathcal{G}'$, hence in $\mathcal{R}$. 

Suppose $x\in\mathcal{G}''\setminus B(p,\Lambda)$ satisfies $R(x)\leq1$,
and let $\sigma:\mathbb{R}\to\mathcal{R}$ be the integral curve of
$\nabla f$ with $\sigma(0)=x$. Then 
\[
\frac{d}{dt}R(\sigma(t))=\langle\nabla R(\sigma(t)),\nabla f(\sigma(t))\rangle=2Rc(\nabla f,\nabla f)(\sigma(t))\geq0,
\]
which implies $R(\sigma(t))\geq R(x)$ for $t\leq0$. Recall that
$r$ in smooth on $\mathcal{G}^{\ast}$ (see Proposition 2.3 of \cite{bam1}),
so for any $t\in\mathbb{R}$ such that $r(\sigma(t))\geq\Lambda$
and $\sigma(t)\in\mathcal{G}$, we have
\[
\frac{d}{dt}r(\sigma(t))=\langle\nabla r(\sigma(t)),\nabla f(\sigma(t))\rangle\geq\frac{1}{8}r(\sigma(t))-|\nabla f(p)|,
\]
so that 
\[
\frac{d}{dt}\log\left(\frac{1}{8}r(\sigma(t))-|\nabla f(p)|\right)\geq\frac{1}{8}.
\]
Because $t\mapsto\log(\frac{1}{8}r(\sigma(t))-|\nabla f(p)|)$ is
locally Lipschitz, with derivative almost-everywhere $\geq\frac{1}{8}$,
we can thus integrate to obtain that $\sigma$ moves $x$ into $B(p,\Lambda)\cap\mathcal{R}$
for sufficiently negative $t<0$, but $R$ decreases as $t$ decreases,
so 
\[
R(x)\geq\min\left\{ \inf_{B(p,\Lambda)\cap\mathcal{R}}R,1\right\} 
\]
for all $x\in\mathcal{G}''$. Next, recall that we have the estimate
$|\nabla R|=2|Rc(\nabla f)|\leq CR\sqrt{f}\leq C(\Lambda)R$ on $B(p,2\Lambda)\cap\mathcal{R}$,
or equivalently $|\nabla\log R|\leq C(\Lambda)$. Fix $\eta>0$, and
suppose $R(y)<\eta$ for some $y\in B(p,\Lambda)\cap\mathcal{R}$.
Integrating $|\nabla\log R|\leq C(\Lambda)$ along a curve from $p$
to $y$ within $B(p,2\Lambda)\cap\mathcal{R}$ gives $R(p)\leq C(\Lambda)\eta$.
By choosing $\eta<c(\Lambda,R(p))$ sufficiently small, we get a contradiction,
so we conclude that $\inf_{B(p,\Lambda)\cap\mathcal{R}}R>0$. Because
$R|\mathcal{R}$ is continuous and $\mathcal{G}''$ is dense in $\mathcal{R}$,
the claim follows. 
\end{proof}
\noindent We now restrict our attention to tangent flows of a fixed
Ricci flow. 

For the remainder of this section, we assume that $(M^{n},(g_{t})_{t\in[0,T)})$
is a closed Ricci flow satisfying $Rc(g_{t})\geq-Ag_{t}$ for all
$t\in[0,T)$, as well as $A^{-1}r^{n}\leq|B(x,t,r)|_{g_{t}}\leq Ar^{n}$
for all $(x,t)\in M\times[0,T)$ and $r\in(0,1]$. Also assume $\text{diam}_{g_{t}}(M)\leq A$
for all $t\in[0,T)$, and that $T\leq A$. By Lemma \ref{lem:basiclemma},
all of these assumptions hold with $A$ replaced by some $\overline{A}(n,A,T,\text{diam}_{g_{0}}(M))<\infty$
for a closed Ricci flow satisfying (\ref{eq:ric}),(\ref{eq:vol}). 

The Ricci lower bound implies that $t\mapsto e^{-At}d_{g_{t}}$ are
pointwise nonincreasing as functions on $M\times M$, so in particular
\[
d_{g_{T}}(x,y):=\lim_{t\to\infty}d_{g_{t}}(x,y)=e^{AT}\lim_{t\to\infty}e^{-At}d_{g_{t}}(x,y)\in[0,\infty)
\]
is a well-defined pseudometric on $M$. We can form the metric space
whose underlying set is $X=M/\sim$, where $x\sim y$ if and only
if $\lim_{t\nearrow T}d_{t}(x,y)=0$, and equip $X$ with the induced
metric $d_{X}$ from passing $d_{g_{T}}$ to the quotient.
\begin{lem}
$\lim_{t\nearrow T}d_{GH}\left((M,d_{g_{t}}),(X,d_{X})\right)=0$.
\end{lem}

\begin{proof}
It suffices to show that for any $\epsilon>0$, there exists $\delta>0$
such that the quotient map $\pi:(M,d_{g_{t}})\to(X,d_{X})$ is an
$\epsilon$-Gromov-Hausdorff approximation for all $t\in(T-\delta,T)$.
Let $\{x_{1},...,x_{N}\}$ be an $\frac{1}{5}e^{-AT}\epsilon$-dense
subset of $(M,d_{g_{0}})$, which is thus $\frac{\epsilon}{5}$-dense
in each $(M,d_{g_{t}})$. Next, choose $\delta>0$ such that $|d_{g_{t}}(x_{i},x_{j})-d_{g_{T}}(x_{i},x_{j})|<\frac{\epsilon}{5}$
for all $i,j\in\{1,...,N\}$ whenever $t\in(T-\delta,T)$. Let $x,y\in M$
and $t\in(T-\delta,T)$ be arbitrary, and choose $i,j\in\{1,...,N\}$
such that $d_{g_{t}}(x,x_{i}),d_{g_{t}}(y,x_{j})<\frac{1}{5}e^{-AT}\epsilon$,
hence $d_{g_{T}}(x,x_{i}),d_{g_{T}}(y,x_{j})<\frac{\epsilon}{5}$.
Then we can estimate
\begin{align*}
|d_{g_{t}}(x,y)-d_{g_{T}}(x,y)|\leq & d_{g_{t}}(x,x_{i})+d_{g_{t}}(y,x_{j})+d_{g_{T}}(x,x_{i})+d_{g_{T}}(y,x_{j})\\
 & +|d_{g_{t}}(x_{i},x_{j})-d_{g_{T}}(x_{i},x_{j})|\\
< & \epsilon.
\end{align*}
Because $d_{g_{T}}(x,y)=d_{X}(\pi(x),\pi(y))$ and $\pi$ is surjective
(hence $\epsilon$-dense), the claim follows.
\end{proof}
In particular, $(X,d_{X})$ is a noncollapsed Ricci limit space, so
by Theorem \ref{thm:openandsmooth}, there is an open dense subset
$\mathcal{R}_{X}$ defined as the set of points $x\in X$ which have
a tangent cone isometric to the standard Euclidean space $\mathbb{R}^{n}$.
Moreover, $(X,d_{X})$ is a compact metric length space, and it follows
from Theorem 3.7 of \cite{cheegercolding2} that $(X,d_{X})$ is the
completion of $\mathcal{R}_{X}$ equipped with the length metric corresponding
to $d_{X}|(\mathcal{R}_{X}\times\mathcal{R}_{X})$. We also set $\mathcal{S}(X):=X\setminus\mathcal{R}_{X}$,
and for any $x\in M$, we denote by $\overline{x}\in X$ the corresponding
equivalence class. 
\begin{lem}
$\pi$ restricts to a homeomorphism from the open dense subset 
\[
M\setminus\Sigma:=\left\{ x\in M;\sup_{U\times[0,T)}|Rm|<\infty\text{ for some neighborhood }U\text{ of }x\text{ in }M\right\} 
\]
of $M$ to its image $(\mathcal{R}_{X},g_{X})$ .
\end{lem}

\begin{proof}
If $x\in M\setminus\Sigma$, then standard distortion estimates imply
$d_{g_{T}}(x,y)\neq0$ for all $y\in M\setminus\{x\}$, so $\pi|(M\setminus\Sigma)$
is injective. Also, $\pi:(M\setminus\Sigma,d_{g_{T}})\hookrightarrow(\mathcal{R},d)$
is an isometry, hence a homeomorphism onto its image. Now suppose
we are given $\overline{x}\in\mathcal{R}_{X}$. Then for any $\sigma>0$,
there exists $r_{0}=r_{0}(\sigma)>0$ such that 
\[
d_{PGH}\left((B^{X}(\overline{x},r),d_{X},\overline{x}),(B^{\mathbb{R}^{n}}(0^{n},r)),d_{\mathbb{R}^{n}},0^{n})\right)<\sigma r
\]
for all $r\in(0,r_{0}]$. 

\noindent \textbf{Claim: }$\lim_{t\nearrow T}d_{PGH}\left((B^{X}(\overline{x},r),d_{X},\overline{x}),(B_{g}(x,t,r),d_{g_{t}},x)\right)=0$
for any fixed $r\in(0,r_{0}]$.

Fix $\epsilon>0$. By the proof of the previous lemma, there exists
$\delta=\delta(\epsilon)>0$ such that $\pi$ satisfies 
\[
|d_{X}(\overline{y}_{1},\overline{y}_{2})-d_{g_{t}}(y_{1},y_{2})|<\epsilon
\]
for all $t\in(T-\delta,T)$ and $y_{1},y_{2}\in B_{g}(x,t,r)$. It
thus remains to show that $\pi(B_{g}(x,t,r))$ is $\epsilon$-dense
in $B^{X}(\overline{x},r)$. To see this, choose an $\frac{\epsilon}{2}$-dense
subset $\{\overline{y}_{1},...,\overline{y}_{N}\}$ of $B^{X}(\overline{x},r)$.
Then, for any choice of representatives $y_{i}\in\overline{y}_{i}$,
we can modify $\delta$ to assume that $d_{t}(y_{i},x)<r$ for all
$i\in\{1,...,N\}$ whenever $t\in(T-\delta,T)$. $\square$

We may therefore find $\tau_{r}=\tau_{r}(\sigma)>0$ such that for
any $t\in(T-\tau_{r},T)$, we have
\[
d_{PGH}\left((B^{X}(\overline{x},r),d_{X},\overline{x}),(B^{\mathbb{R}^{n}}(0,r),d_{\mathbb{R}^{n}},0^{n})\right)<\sigma r.
\]
By Proposition \ref{thm:chenyuan}, Remark \ref{rem:epsregremark},
and choosing $\sigma=\sigma(A)>0$ sufficiently small, we have $\widetilde{r}_{Rm}(x,t)>\sigma r$
for all $t\in(T-\tau_{r_{0}},T)$. By choosing $t=t(r')$ sufficiently
close to $t$, we get a curvature bound $|Rm|\leq(r')^{-2}$on $B_{g}(x,t,r')\times[t,T)$,
so $x\in M\setminus\Sigma$.
\end{proof}
In particular, $\mathcal{R}_{X}$ has the structure of a smooth Riemannian
manifold, equipped with the metric $g_{X}$ using the homeomorphism
$\pi|(M\setminus\Sigma)\to\mathcal{R}_{X}$. Moreover, it is standard
that $g_{t}\to g_{T}$ in $C_{loc}^{\infty}(M\setminus\Sigma)$ as
$t\nearrow T$, hence $(M\setminus\Sigma,g_{t})\to(\mathcal{R}_{X},g_{X})$
in the $C^{\infty}$ Cheeger-Gromov sense, with canonical diffeomorphism
$\pi|M\setminus\Sigma$ and Gromov-Hausdorff approximations $\pi$.
Moreover, we know that $\mathcal{S}(X)$ has Hausdorff dimension $\leq n-2$
by \cite{cheegercolding1}. 

Now fix $x\in M$ corresponding to a singular point $\overline{x}\in\mathcal{S}(X)$.
Choose $t_{i}\nearrow T$ such that the conjugate heat kernels $K(x,t_{i};\cdot,\cdot)$
converge in $C_{loc}^{\infty}(M\times[0,T))$ to some conjugate heat
kernel $K\in C^{\infty}(M\times(0,T))$ at the singular time based
at $x$. Write $d\nu_{s}:=K(\cdot,s)dg_{s}$. Fix any sequence $\tau_{i}\searrow0$,
and define the rescaled flows $g_{t}^{i}:=\tau_{i}^{-1}g_{T+\tau_{i}t}$,
as well as the correspondingly rescaled conjugate heat kernels $K^{i}(y,t;z,s):=\tau_{i}^{\frac{n}{2}}K(y,T+\tau_{i}t;z,T+\tau_{i}s)$.
Also set $K^{i}(y,t):=\tau_{i}^{\frac{n}{2}}K(y,T+\tau_{i}t)$ and
$d\nu_{t}^{i}:=d\nu_{x,T;T+\tau_{i}t}=K^{i}(\cdot,t)dg_{t}^{i}$.
By Theorem 1.38 of \cite{bamlergen3}, we can pass to a subsequence
to obtain uniform $\mathbb{F}$-convergence within some correspondence
$\mathfrak{C}$ on compact time intervals:

\noindent 
\[
(M,(g_{t}^{i})_{t\in[-\tau_{i}^{-1}T,0)},(\nu_{t}^{i})_{t\in[-\tau_{i}^{-1}T,0)})\xrightarrow[i\to\infty]{\mathbb{F},\mathfrak{C}}(\mathcal{Y},(\mu_{t}^{\infty})_{t<0}),
\]
where $(\mathcal{Y},(\mu_{t}^{\infty})_{t\in(-\infty,0)})$ is a metric
soliton with Nash entropy $\mathcal{N}_{(\mu_{t}^{\infty})}(\tau)<0$
(in particuar, the soliton is not flat Euclidean space). Let $\mathcal{R}$
be the regular set of $\mathcal{Y}$, equipped with a Ricci flow spacetime
$(\mathcal{R},\mathfrak{t},\partial_{\mathfrak{t}},(g_{t}^{\infty})_{t\in(-\infty,0)})$.
By Theorem 1.19 of \cite{bamlergen3}, there is an $n$-dimensional
singular space $(Y,d_{Y},\mathcal{R}_{Y},g_{Y})$, a probability measure
$\mu$ on $Y$, a smooth function $f_{Y}\in C^{\infty}(\mathcal{R}_{Y})$,
and an identification $\mathcal{Y}=Y\times(-\infty,0)$ such that
the following hold for all $t\in(-\infty,0)$:

$(a)$ $(\mathcal{Y}_{t},d_{t},\mu_{t}^{\infty})=(Y\times\{t\},|t|^{\frac{1}{2}}d,\mu)$,

$(b)$ $\mathcal{R}=\mathcal{R}_{Y}\times(-\infty,0)$, and $\partial_{\mathfrak{t}}-\nabla f$
corresponds to the standard vector field on the second factor,

$(c)$ $(\mathcal{R}_{t},g_{t}^{\infty},f(\cdot,t))=(\mathcal{R}_{Y}\times\{t\},|t|g_{Y},f_{Y})$,
where $d\mu_{t}^{\infty}=K^{\infty}(\cdot,t)dg_{t}^{\infty}=:(4\pi|t|)^{-\frac{n}{2}}e^{-f}dg_{t}^{\infty}$
on $\mathcal{R}$

$(d)$ $Rc+\nabla^{2}f=\frac{1}{2\tau}g$ on $\mathcal{R}$, hence
$Rc(g_{Y})+\nabla^{2}f_{Y}=\frac{1}{2}g_{Y}$ on $\mathcal{R}_{Y}$

\noindent Using this identification, we can apply Theorem 9.31 of
\cite{bamlergen2} to obtain an exhaustion $(U_{i})_{i\in\mathbb{N}}$
of $\mathcal{R}_{Y}$ by precompact open sets, a sequence $\alpha_{i}\searrow0$,
along with embeddings $\psi_{i}:U_{i}\times(-\alpha_{i}^{-1},-\alpha_{i})\to M\times(-\alpha_{i}^{-1},-\alpha_{i})$
satisfying
\begin{align}
||\partial_{t}-(\psi_{i}^{-1})_{\ast}\partial_{t}^{i}||_{C^{i}(U_{i}\times(-\alpha_{i}^{-1},-\alpha_{i}))} & <\alpha_{i},\label{eq:esty}\\
||\psi_{i}^{\ast}g^{i}-g^{\infty}||_{C^{i}(U_{i}\times(-\alpha_{i}^{-1},-\alpha_{i}))} & <\alpha_{i},\nonumber \\
||\psi_{i}^{\ast}K^{i}-K^{\infty}||_{C^{i}(U_{i}\times(-\alpha_{i}^{-1},-\alpha_{i}))} & <\alpha_{i},\nonumber 
\end{align}
for all $i\in\mathbb{N}$, where the $C^{i}$ norms are with respect
to the metric $|t|g_{Y}$. Set $\psi_{i,t}|U_{i,}\times\{t\}:U_{i}\to M$.
Because the $\mathcal{Y}$ is a continuous metric flow on $(-\infty,0)$,
we can pass to a subsequence so that the $\mathbb{F}$-convergence
is timewise for all $t\in(-\infty,0)$. 
\begin{prop}
\label{prop:volconverge} For any $y_{\infty}\in\mathcal{R}_{Y}$
and $r>0$, we have 
\[
\mathcal{H}^{n}(B^{Y}(y_{\infty},r))=|B^{Y}(y_{\infty},r)\cap\mathcal{R}_{Y}|_{g_{Y}}\geq A^{-1}r^{n}.
\]
Moreover, the volume ratio 
\[
r\mapsto\frac{\mathcal{H}^{n}(B^{Y}(y_{\infty},r))}{r^{n}}
\]
is nonincreasing. 
\end{prop}

\begin{proof}
The first equality follows from $\mathcal{H}^{n}(Y\setminus\mathcal{R}_{Y})=0$
along with the fact that for any connected Riemannian manifold $(N^{n},h)$,
the $n$-dimensional Lebesgue measure induced by $h$ coincides with
the $n$-dimensional Hausdorff measure induced by $d_{h}$. Fix any
sequence $(z_{i},-1)$ of $H_{n}$-centers of $\nu^{i}$. By Lemma
\ref{lem:gromovhaus}, we have pointed Gromov-Hausdorff convergence
\[
(M,d_{g_{-1}^{i}},z_{i})\to(Y,d,z_{\infty})
\]
for some $z_{\infty}\in Y$, so we can find $y_{i}\in M$ converging
to $y_{\infty}$ with respect to this Gromov-Hausdorff convergence.
By Colding's volume convergence theorem, we get 
\[
\mathcal{H}^{n}(B^{Y}(y_{\infty},r))=\lim_{i\to\infty}|B_{g^{i}}(y_{i},-1,r)|_{g_{-1}^{i}}\geq A^{-1}r^{n}.
\]
For any $\kappa>0$, we know that $Rc(g^{i})\geq-\kappa g^{i}$ for
sufficiently large $i\in\mathbb{N}$, so that Bishop volume comparison
and volume convergence tell us that 
\[
r\mapsto\frac{\mathcal{H}^{n}(B^{Y}(y_{\infty},r))}{v_{-\kappa}(r)}
\]
is nonincreasing. Taking $\kappa\searrow0$ thus gives the remaining
claim. 
\end{proof}
\begin{prop}
\label{prop:tangentflowricciflat} If $(Y,d,\mathcal{R}_{Y},g_{Y})$
is the metric soliton corresponding to the tangent flow of 
\[
(M,(g_{t}^{i})_{t\in[-\tau_{i}^{-1}T,0)},(\nu_{t}^{i})_{t\in[-\tau_{i}^{-1}T,0)})
\]
as described above, then $(\mathcal{R}_{Y},g_{Y})$ is Ricci flat.
\end{prop}

\begin{proof}
Fix any sequence $\lambda_{j}\nearrow\infty$ and basepoint $y_{0}\in\mathcal{R}_{Y}$.
By Proposition \ref{prop:volconverge}, we know$\mathcal{H}^{n}(B(y_{0},r))\geq A^{-1}r^{n}$
for all $r>0$, and that the volume ratios monotonically decrease
to a positive number as $r\nearrow\infty$. Because $Y$ is a Gromov-Hausdorff
limit of some noncollapsed $(M_{i},d_{g_{-1}^{i}},x_{i})$, where
$(M_{i},(g_{t}^{i})_{t\in[-4,0]})$ are closed Ricci flows satisfying
$Rc(g_{t}^{i})\geq-\tau_{i}Ag_{t}^{i}$ where $\tau_{i}\searrow0$,
we may use a diagonal argument to get that the rescaled spaces $(Y,\lambda_{j}^{-1}d_{Y},y_{0})$
converge in the pointed Gromov-Hausdorff sense to another noncollapsed
Ricci limit space $(Z,d_{Z},z_{0})$, which is moreover a metric cone.
In particular, we can choose $z'\in Z\setminus\{z_{0}\}$ in the regular
set of $Z$, which corresponds under the Gromov-Hausdorff convergence
to a sequence $y_{j}\in Y$ with 
\[
\lim_{j\to\infty}\lambda_{j}^{-1}d_{Y}(y_{j},y_{0})=d_{Z}(z_{0},z')\in(0,\infty).
\]
Recall that Colding's volume convergence theorem also holds for sequences
of noncollapsed Ricci limit spaces (Theorem 10.15 of \cite{cheegernotes}),
so for any $\delta>0$, there exists $r=r(\delta)>0$ such that
\[
\lambda_{j}^{-n}\mathcal{H}^{n}(B^{Y}(y_{j},\lambda_{j}r))\geq(\omega_{n}-\delta/2)r^{n}
\]
for $j=j(\delta)\in\mathbb{N}$ sufficiently large. On the other hand,
we know that $(M,d_{g_{-1}^{i}},x_{i})\to(Y,d,y_{0})$ in the pointed
Gromov-Hausdorff sense for some sequence $x_{i}\in M$, so that for
any fixed $j\in\mathbb{N}$, there is a sequence $y_{i,j}\in M$ with
\[
(B_{g^{i}}(y_{i,j},-1,\lambda_{j}r),d_{g_{-1}^{i}},y_{i,j})\to(B^{Y}(y_{j},\lambda_{j}r),d_{g_{-1}^{i}},y_{i,j}))
\]
in the pointed Gromov-Hausdorff sense as $i\to\infty$. By Colding's
volume convergence theorem, for any $j=j(\delta)\in\mathbb{N}$ large,
we have the following for all $i=i(\delta,j)\in\mathbb{N}$ sufficiently
large:
\[
|B_{g^{i}}(y_{i,j},-1,\lambda_{j}r)|_{g_{-1}^{i}}\geq(\omega_{n}-\delta)(\lambda_{j}r)^{n}.
\]
Moreover, we have $Rc(g_{-1}^{i})\geq-\tau_{i}Ag_{-1}^{i}$, but $\tau_{i}\searrow0$,
so we can choose $\delta=\delta(A)>0$ according to Theorem \ref{thm:chenyuan}
and Remark \ref{rem:volremark}, so that 
\[
\widetilde{r}_{Rm}^{g^{i}}(y_{i,j},\lambda_{j}r)\geq c(A)(\lambda_{j}r)
\]
for $i=i(j)\in\mathbb{N}$ large. Taking $i\to\infty$, we get that
$y_{j}\in\mathcal{R}_{Y}$ and $\widetilde{r}_{Rm}^{g_{Y}}(y_{j})\geq c(A)\lambda_{j}r.$
Since $\lambda_{j}\nearrow\infty$, we have $\inf_{\mathcal{R}_{Y}}R=0$,
hence $(\mathcal{R}_{Y},g_{Y})$ is Ricci flat by Lemma \ref{lem:niargument}.
\end{proof}
\begin{proof}[Proof of Theorem \ref{thm:theoremhigherdim}]
 Let $\mathcal{Y}$ be a tangent flow based at $\nu_{x,T}$, which
is modeled on a singular shrinking soliton $(Y,d_{Y},\mathcal{R}_{Y},g_{Y})$.
We first pass to a subsequence so that the $\mathbb{F}$-convergence
is timewise at almost every time. We can apply Lemma \ref{lem:gromovhaus}
to the sequence $(M,(g_{t}^{i}))$ or a rescaling $\widetilde{g}_{t}^{i}:=\lambda g_{\lambda^{-1}t}^{i}$
to obtain $y_{i,t}'\in M$ and $y_{\infty,t}'\in\mathcal{Y}_{t}$
such that $(M,d_{g_{t}^{i}},y_{i,t}')\to(\mathcal{Y}_{t},d_{t},y_{\infty,t}')$
for almost every $t\in(-\infty,0)$. By Proposition \ref{prop:tangentflowricciflat},
$Rc(g_{Y})=0$ on $\mathcal{R}_{Y}$, so $Y$ is a metric cone by
Theorem 2.18 of \cite{bamlergen3}. 
\end{proof}
\begin{proof}[Proof of Theorem \ref{thm:theorem4}]
 Suppose by way of contradiction that $L:=\limsup_{t\nearrow T}(T-t)r_{Rm}^{-2}(x,t)<\infty$.
Since $\overline{x}\notin\mathcal{S}(X)$, we know $L>0$. Because
$r_{Rm}(x,t)\geq\frac{1}{2\sqrt{L}}\sqrt{T-t}$ for all $t\in(0,T)$
sufficiently close to $T$, we have
\[
\frac{1}{\sqrt{T-t}}\int_{t}^{T}\sqrt{T-s}R(x,s)ds\leq\frac{1}{\sqrt{T-t}}\int_{t}^{T}\frac{4L}{\sqrt{T-s}}ds=8L.
\]
Fix $t_{i}\nearrow T$ such that $K(x,t_{i};\cdot,\cdot)$ converge
in $C_{loc}^{\infty}(M\times[0,T))$ to a conjugate heat kernel at
the singular time based at $\overline{x}$. Then we can apply the
heat kernel lower bound of \cite{qizhangvol} to get

\[
K(y,t)\geq\frac{ce^{8L}}{(T-t)^{\frac{n}{2}}}\exp\left(-\dfrac{d_{g_{t}}^{2}(x,y)}{c(T-t)}\right),
\]
for all $y\in M$, (in fact, we even have the stronger bound where
we replace $d_{g_{t}}(x,y)$ with $d_{g_{T}}(x,y)$) where $c=c(g_{0})>0$.
Now fix $\tau_{i}\searrow0$, and consider the $\mathbb{F}$-convergence
within a correspondence
\[
(M,(g_{t}^{i})_{t\in[-\tau_{i}^{-1}T,0)},(\nu^{i})_{t\in[-\tau_{i}^{-1}T,0)})\xrightarrow[i\to\infty]{\mathbb{F},\mathfrak{C}}(\mathcal{Y},(\mu_{t}^{\infty})_{t<0})
\]
in the discussion preceding Proposition \ref{prop:volconverge}, where
$\mathcal{Y}$ is a metric soliton corresponding to the singular shrinking
GRS $(Y,d,\mathcal{R}_{Y},g_{Y},f)$. The rescaled metrics satisfy
\[
\frac{1}{C^{\ast}}\exp\left(-C^{\ast}d_{g_{-1}^{i}}^{2}(x,y)\right)\leq K^{i}(y,t)
\]
for all $(y,t)\in M\times[-2,0)$ and $i\in\mathbb{N},$ where $C^{\ast}=C^{\ast}(n,A,T,g_{0},L)<\infty$.
In particular, for any $D<\infty$, $r>0$, and $y\in B_{g^{i}}(x,-1,D)$,
we have 
\[
\int_{B_{g^{i}}(y,-1,r)}K^{i}(y,t)\geq\frac{1}{C^{\ast}}\exp\left(-C^{\ast}D^{2}\right)|B_{g^{i}}(y,-1,r)|_{g^{i}}\geq\frac{e^{-C^{\ast}D^{2}}}{C^{\ast}}A^{-1}r^{n},
\]
hence the hypotheses of Proposition \ref{prop:easyconvergence} are
satisfied with $x_{i}=x$.

Because $\mathcal{Y}$ is continuous on $[-2,0)$, the $\mathbb{F}$-convergence
is timewise at $t=-1$. We let $\mathfrak{C}=((Z_{t},d_{t}^{Z})_{t\in I'},(\phi_{t}^{i})_{t\in I'^{,i},i\in\mathbb{N}\cup\{\infty\}})$
be a correspondence such that
\[
\lim_{i\to\infty}d_{W_{1}}^{Z_{t}}\left((\phi_{-1}^{i})_{\ast}\nu_{-1}^{i},(\phi_{-1}^{\infty})_{\ast}\mu_{-1}^{\infty}\right)=0.
\]
Let $U_{i},\psi_{i}$ be as in (\ref{eq:esty}). By Proposition \ref{prop:easyconvergence},
we can pass to a subsequence to find $x_{\infty}\in Y$ such that
\[
\lim_{i\to\infty}d_{t}^{Z}((\phi_{-1}^{i})(x),(\phi_{-1}^{\infty})(x_{\infty}))=0
\]
 and
\[
(M,d_{g_{-1}^{i}},x)\to(Y,d,x_{\infty})
\]
in the pointed Gromov-Hausdorff sense. Because $\liminf_{i\to\infty}r_{Rm}^{g^{i}}(x,-1)\geq\frac{1}{2\sqrt{L}}>0$,
we moreover have $x_{\infty}\in\mathcal{R}$. 

By Proposition \ref{prop:tangentflowricciflat}, $(\mathcal{R}_{Y},g_{Y})$
must be Ricci flat. By Theorems 1.17, 1.19 of \cite{bamlergen3},
we know $\mathcal{Y}$ is a static flow, and $(Y,d)$ is a metric
cone $C(Z)$ over some compact metric space $(Z,d_{Z})$. Choose $r_{0}\in(0,1]$
such that $B(x_{\infty},2r_{0})\subseteq\mathcal{R}_{-1}$.

Let $\gamma:[0,1]\to M$ be a minimizing curve from $\psi_{i}(x_{\infty})$
to $x$ with respect to $g_{-1}^{i}$. Suppose by way of contradiction
that $\gamma([0,1])\not\subseteq\psi_{i}(U_{i})$, and define 
\[
u^{\ast}:=\inf\left\{ u\in[0,1];\gamma(u)\in\psi_{i}(\partial B(x_{\infty},r_{0}))\right\} .
\]
Because $B(x_{\infty},r_{0})\subset\subset U_{i}\times\{-1\}$ for
sufficiently large $i=i(r_{0},x_{\infty})\in\mathbb{N}$, we must
have 
\[
d_{g_{-1}^{i}}(\psi_{i}(x_{\infty}),x)\geq\text{length}_{g_{-1}^{i}}(\gamma|[0,u^{\ast}))\geq\frac{1}{2}\text{length}_{g_{-1}^{\infty}}(\psi_{i}^{-1}\circ\gamma|[0,u^{\ast}))\geq\frac{1}{2}r_{0}.
\]
However, 
\begin{align*}
d_{g_{-1}^{i}}(\psi_{i}(x_{\infty}),x)= & d_{-1}^{Z}\left((\phi_{-1}^{i}\circ\psi_{i})(x_{\infty}),\phi_{-1}^{i}(x)\right)\\
\leq & d_{-1}^{Z}(\phi_{-1}^{i}(x),\phi_{-1}^{\infty}(x_{\infty}))+d_{-1}^{Z}\left((\phi_{-1}^{i}\circ\psi_{i})(x_{\infty}),\phi_{-1}^{\infty}(x_{\infty})\right)\to0
\end{align*}
as $i\to\infty$ by the choice of $x_{\infty}$ and Theorem 9.31(d)
in \cite{bamlergen2}, a contradiction. We therefore have $x_{i}:=\psi_{i}^{-1}(x)\in U_{i}$
for sufficiently large $i\in\mathbb{N}$, and
\[
d_{-1}(x_{i},x_{\infty})\leq\text{length}_{g_{-1}^{\infty}}(\psi_{i}^{-1}\circ\gamma)\leq2\text{length}_{g_{-1}^{i}}(\gamma)\leq2d_{g_{-1}^{i}}(\psi_{i}(x_{\infty}),x)\to0
\]
as $i\to\infty$.

\noindent \textbf{Claim: }For any $\tau>0$, we have $\widetilde{r}_{Rm}^{g^{i}}(x,-\tau)\geq\frac{1}{2}r_{0}$
for $i=i(\tau)\in\mathbb{N}$ sufficiently large. 

Because $x_{\infty}\in\mathcal{R}_{-1}$ and $\mathcal{Y}$ is static,
we have $x_{\infty}(t)\in\mathcal{R}$ for all $t\in[-1,0)$. Moreover,
for any $\tau>0$, $\{x_{\infty}(t);t\in[-1,-\tau]\}$ is a compact
subset of $\mathcal{R}$, so admits a neighborhood $V\subset\subset\mathcal{R}$.
Because $(\psi_{i}^{-1})_{\ast}\partial_{t}\to\partial_{\mathfrak{t}}$
in $C_{loc}^{\infty}(\mathcal{R})$, a standard statement about continuous
dependence of ODE solutions on parameters (see Chapter 5 of \cite{hartmanode})
implies that for large $i\in\mathbb{N}$, the integral curve $\gamma_{i}$
of $(\psi_{i}^{-1})_{\ast}\partial_{t}$ starting at $x_{i}$ is well-defined
for all $t\in[-1,-\tau]$, and $\gamma_{i}(t)\to x_{\infty}(t)$ uniformly
in $t\in[-1,-\tau]$ as $i\to\infty$, so $\gamma_{i}([-1,-\tau])\subseteq V\subseteq U_{i}\times\{-\tau\}$
for sufficiently large $i\in\mathbb{N}$. In particular, $B(\gamma_{i}(\tau),\frac{3}{2}r_{0})\subset\subset\mathcal{R}$
when $i=i(\tau)\in\mathbb{N}$ is large, hence $B(\gamma_{i}(\tau),\frac{3}{2}r_{0})\subseteq U_{i}$
and 
\[
B_{g^{i}}(x,-\tau,r_{0})\subseteq\psi_{i}\left(B(\gamma_{i}(\tau),\frac{3}{2}r_{0})\right).
\]
However, because $\mathcal{Y}$ is static, we know $|Rm|_{g^{\infty}}(y(t),t)\leq\overline{C}$
for all $y\in B(x_{\infty},\frac{3}{2}r_{0})$ and $t\in[-1,0)$,
which implies that $|Rm|_{g^{i}}(\cdot,-\tau)\leq2\overline{C}$ on
$B_{g^{i}}(x,-\tau,r_{0})$ for $i=i(\tau)\in\mathbb{N}$ large. $\square$

Now apply Theorem \ref{thm:pseudolocality} on the ball $B_{g^{i}}(x,-\tau,\min\{r_{0},\overline{C}^{-2}\})$
for some $\tau=\tau(r_{0},\overline{C},A)\in(-1,0)$ sufficiently
small in order to contradict the fact that $x$ is a singular point.
\end{proof}

\section{$C^{0}$ Orbifold Structure in Dimension 4}

Throughout this section, we assume that $(M^{4},(g_{t})_{t\in[0,T)})$
is a closed, simply connected four-dimensional Ricci flow satisfying
$Rc(g_{t})\geq-Ag_{t}$ for all $t\in[0,T)$, as well as $A^{-1}r^{4}\leq|B(x,t,r)|_{g_{t}}\leq Ar^{4}$
for all $(x,t)\in M\times[0,T)$ and $r\in(0,1]$. 

\noindent As before, fix a sequence $\tau_{i}\searrow0$, a singular
point $\overline{x}\in X$, let $K\in C^{\infty}(M\times[0,T))$ be
a conjugate heat kernel at the singular time based at $x$, and let
$g_{t}^{i}:=\tau_{i}^{-1}g_{T+\tau_{i}t}$, $K^{i}(\cdot,t):=\tau_{i}^{\frac{n}{2}}K(\cdot,T+\tau_{i}t)$,
$d\nu_{t}^{i}:=K^{i}(\cdot,t)dg_{t}^{i}$. By Proposition \ref{prop:tangentflowricciflat}
and Theorems 1.17, 1.38, 1.47 of \cite{bamlergen3}, we can pass to
a subsequence to obtain uniform $\mathbb{F}$-convergence within some
correspondence $\mathfrak{C}$ on compact time intervals:

\noindent 
\[
(M,(g_{t}^{i})_{t\in[-\tau_{i}^{-1}T,0)},(\nu_{t}^{i})_{t\in[-\tau_{i}^{-1}T,0)})\xrightarrow[i\to\infty]{\mathbb{F},\mathfrak{C}}(\mathcal{Y},(\mu_{t}^{\infty})_{t\in(-\infty,0)}),
\]
where $(\mathcal{Y},(\mu_{t}^{\infty})_{t\in(-\infty,0)})$ is a static,
$H_{4}$-concentrated metric soliton. Let $(\mathcal{R},g^{\infty},\mathfrak{t},\partial_{\mathfrak{t}})$
be the spacetime structure on the regular set of $\mathcal{Y}$ as
before, so that by Theorem 1.47 of \cite{bamlergen3}, there is a
finite subgroup $\Gamma\leq O(4,\mathbb{R})$ and an identification
$\mathcal{Y}=C(\mathbb{S}^{3}/\Gamma)\times(-\infty,0)$ such that
the following hold for all $t\in(-\infty,0)$:

$(a)$ $(\mathcal{Y}_{t},d_{t})=(C(\mathbb{S}^{3}/\Gamma)\times\{t\},d)$,
where $d$ is the cone metric on $C(\mathbb{S}^{3}/\Gamma)$,

$(b)$ $\mathcal{R}=\left(C(\mathbb{S}^{3}/\Gamma)\setminus\{o_{\ast}\}\right)\times(-\infty,0)$,
where $o_{\ast}$ is the vertex of $C(\mathbb{S}^{3}/\Gamma)$, and
$\partial_{\mathfrak{t}}$ corresponds to the standard vector field
on the second factor,

$(c)$ $(\mathcal{R}_{t},g_{t}^{\infty})=\left(\left(C(\mathbb{S}^{3}/\Gamma)\setminus\{o_{\ast}\}\right)\times\{t\},g_{Y}\right)$,
where $g_{Y}=dr^{2}+r^{2}g_{\mathbb{S}^{3}/\Gamma}$ is the smooth
cone metric over $\mathbb{S}^{3}/\Gamma$,

$(d)$ $Rc=0$ and $d\mu_{t}^{\infty}=K^{\infty}(\cdot,t)dg_{t}^{\infty}=\frac{1}{|\Gamma|}(4\pi|t|)^{-2}\exp\left(-\frac{d^{2}(o_{\ast},\cdot)}{4|t|}\right)dg_{t}^{\infty}$
on $\mathcal{R}$.

\noindent For ease of notation, we also write $B(o_{\ast},r_{2}):=B^{C(\mathbb{S}^{3}/\Gamma)}(o_{\ast},r_{2})$,
$A(o_{\ast},r_{1},r_{2}):=B(o_{\ast},r_{2})\setminus\overline{B}(o_{\ast},r_{1})$
for all $r_{2}>r_{1}>0$. There is a sequence $\alpha_{i}\searrow0$
such that if $A_{i}:=A(o_{\ast},\alpha_{i},\alpha_{i}^{-1})$, then
there are embeddings $\psi_{i}:A_{i}\times(-2,0)\to M$ satisfying
\[
||\partial_{\mathfrak{t}}-(\psi_{i}^{-1})_{\ast}\partial_{t}||_{C^{i}(U_{i}\times(-\alpha_{i}^{-1},-\alpha_{i}))}<\alpha_{i},
\]
\[
||\psi_{i}^{\ast}g^{i}-g^{\infty}||_{C^{i}(U_{i}\times(-\alpha_{i}^{-1},-\alpha_{i}))}<\alpha_{i},
\]
\[
||\psi_{i}^{\ast}K^{i}-K^{\infty}||_{C^{i}(U_{i}\times(-\alpha_{i}^{-1},-\alpha_{i}))}<\alpha_{i}
\]
for all $i\in\mathbb{N}$, where the $C^{i}$ norms are with respect
to the metric $g_{Y}$. By abuse of notation, we write $\psi_{i}$
for the restriction $\psi_{i}|(A_{i}\times\{-1\}):A_{i}\to M$. In
the remainder of this section, all geometric quantities ($H_{4}$-centers,
for example) correspond to the rescaled flows $g^{i}$. 
\begin{prop}
\label{prop:part1lemma} $(M,d_{g_{-1}^{i}},x)$ converge in the pointed
Gromov-Hausdorff sense to $(C(\mathbb{S}^{3}/\Gamma),d,o_{\ast})$,
with smooth convergence on $C(\mathbb{S}^{3}/\Gamma)\setminus\{o_{\ast}\}$. 
\end{prop}

\begin{proof}
Suppose that $(x_{i},-1)$ are $H_{4}$-centers of $\nu^{i}$. Bamler's
Gaussian upper bounds for the conjugate heat kernel (Theorem 7.2 of
\cite{bamlergen1}) then give
\[
K^{i}(y,-1)\leq C\exp\left(-\frac{1}{C}d_{g_{-1}^{i}}^{2}(x_{i},y)\right)
\]
for all $y\in M$, where $C=C(A)<\infty$. On the other hand, we know
\[
K^{\infty}(v,-1)=\frac{1}{|\Gamma|(4\pi)^{2}}\exp\left(-\frac{1}{4}d^{2}(o_{\ast},v)\right),
\]
for all $v\in C(\mathbb{S}^{3}/\Gamma)\setminus\{o_{\ast}\}$, where
$|\Gamma|\leq C(A)$. Thus $\psi_{i}^{\ast}K^{i}\to K^{\infty}$ in
$C_{loc}^{\infty}(A_{i})$ implies that (after possibly reindexing)
\[
K^{i}(y,-1)\geq\frac{1}{|\Gamma|(8\pi)^{2}}\exp\left(-\frac{1}{4}d^{2}(o_{\ast},\psi_{i}^{-1}(y))\right)
\]
for any $y\in\psi_{i}(A_{i})$, when $i\in\mathbb{N}$ is sufficiently
large. Combining estimates, we see that
\[
d^{2}(o_{\ast},\psi_{i}^{-1}(y))\geq\frac{1}{C}d_{g_{-1}^{i}}^{2}(x_{i},y)-C
\]
for all $y\in\psi_{i}(A_{i})$ when $i\in\mathbb{N}$ is sufficiently
large, where $C=C(A)<\infty$. This gives an upper bound of the form
$d_{g_{-1}^{i}}(x_{i},y)\leq D$ for all $y\in\psi_{i}(\partial B(o_{\ast},1))$,
where $D=D(A)<\infty$. Because $M$ is simply connected, for any
$r\in(\alpha_{i},\alpha_{i}^{-1})$, $\psi_{i}(\partial B(o_{\ast},r))$
disconnects $M$ into two pieces $M_{i,r}',M_{i,r}''$ by the generalized
Jordan-Brouwer theorem, where $\psi_{i}(A(o_{\ast},\alpha_{i},r))\subseteq M_{i,r}'$
and $\psi_{i}(A(o_{\ast},r,\alpha_{i}^{-1}))\subseteq M_{i,r}''$. 

\noindent \textbf{Claim 1: }When $i\in\mathbb{N}$ is sufficiently
large, we have $x_{i}\in M_{i,4D}'$.

Suppose instead that $x_{i}\in M_{i,4D}''$, and let $\gamma:[0,1]\to M$
be a $g_{-1}^{i}$-minimizing geodesic from $x_{i}$ to some $y\in\psi_{i}(\partial B(o_{\ast},1))$.
Because $M$ is simply connected, we know $\psi_{i}(\partial B(o_{\ast},4D))$,
$\psi_{i}(\partial B(o_{\ast},1))$ each separate $M$ into two components,
so we can find $0\leq s_{1}<s_{2}\leq1$ such that $\gamma(s_{1})\in\psi_{i}(\partial B(o_{\ast},1))$,
$\gamma(s_{2})\in\psi_{i}(\partial B(o_{\ast},4D))$, and $\gamma|[s_{1},s_{2}]\subseteq\psi_{i}(A(o_{\ast},1,4D))$.
Then, for $i\in\mathbb{N}$ sufficiently large, 
\[
3D\leq\text{length}_{g_{Y}}(\psi_{i}^{-1}\circ\gamma|[s_{1},s_{2}])\leq2\text{length}_{g_{-1}^{i}}(\gamma|[s_{1},s_{2}])\leq2d_{g_{-1}^{i}}(x_{i},y)\leq2D,
\]
a contradiction. $\square$

For any $y\in\psi_{i}(A(o_{\ast},D,\frac{1}{2}\alpha_{i}^{-1}))$
and $z\in B_{g^{i}}(y,-1,\frac{D}{4})$, let $\gamma:[0,1]\to M$
be a minimizing $g_{-1}^{i}$-geodesic from $y$ to $z$. If $\gamma$
leaves $\psi_{i}(A(o_{\ast},\alpha_{i},\alpha_{i}^{-1}))$, we can
argue as in Claim 1 to obtain $\text{length}_{g_{-1}^{i}}(\gamma)\geq\frac{1}{2}D$
for $i\in\mathbb{N}$ sufficiently large, a contradiction. Thus, for
$i$ large, we have $|Rm|_{g^{i}}(\cdot,-1)\leq\epsilon_{i}$ on $B_{g^{i}}(y,-1,\frac{D}{4})$
for all $y\in\psi_{i}(A(o_{\ast},D,\frac{1}{2}\alpha_{i}^{-1}))$,
where $\epsilon_{i}\searrow0$. We can thus apply Theorem \ref{thm:pseudolocality}
to the ball $B(y,-1,\frac{D}{4})$ to get (after possibly increasing
$D=D(A)<\infty)$ $|Rm|_{g^{i}}\leq C(A)$ on $\psi_{i}(A(o_{\ast},D,\frac{1}{2}\alpha_{i}^{-1}))\times(-2,0)$
when $i=i(A)\in\mathbb{N}$ is large.

\noindent \textbf{Claim 2: }For $i\in\mathbb{N}$ sufficiently large,
we have $x\in M_{i,16D}'$.

Suppose instead that $x\in M_{i,16D}''$. Let $\gamma:[0,1]\to M$
be a $g_{t}^{i}$-minimizing curve from $x$ to any $y\in M_{i,8D}'$,
where $t\in[-1,0)$ is arbitrary. Arguing as in Claim 1, we find $0\leq s_{1}<s_{2}\leq1$
such that $\gamma(s_{1})\in\psi_{i}(\partial B(o_{\ast},4D))$ and
$\gamma(s_{2})\in\psi_{i}(\partial B(o_{\ast},8D))$, and $\gamma|[s_{1},s_{2}]\in\psi_{i}(A(o_{\ast},8D,16D))$.
From the curvature bound $|Rm|_{g^{i}}\leq C(A)$ on $\psi_{i}(A(o_{\ast},8D,16D))\times(-1,0)$,
we can estimate
\[
d_{g_{t}^{i}}(x,y)\geq\text{length}_{g_{t}^{i}}(\gamma|[s_{1},s_{2}])\geq\frac{1}{C(A)}\text{length}_{g_{-1}^{i}}(\gamma|[s_{1},s_{2}])\geq cD
\]
for all $t\in[-1,0)$, where $c=c(A)>0$. However, Proposition \ref{prop:heatkernel}
applied to the original flow $(M,(g_{t})_{t\in[0,T)})$ gives the
following after rescaling:
\[
K^{i}(y,-1)\leq C^{\ast}\exp\left(-\frac{1}{C^{\ast}}\lim_{t\nearrow0}d_{g_{t}^{i}}^{2}(y,x)\right)
\]
for all $y\in M$, where $C^{\ast}=C^{\ast}(A)$. In particular, for
$y\in M_{i,8D}'$, we have
\[
K^{i}(y,-1)\leq C^{\ast}\exp\left(-\frac{1}{C^{\ast}}c^{2}D^{2}\right).
\]
Now we integrate on $B(x_{i},-1,\sqrt{2H_{4}})\subseteq M_{i,8D}'$,
using the volume upper bound, and the concentration estimate near
$H_{4}$-centers (Proposition 3.13 of \cite{bamlergen1}) to get
\[
\frac{1}{2}\leq\int_{B(x_{i},-1,\sqrt{2H_{4}})}K^{i}(y,-1)dg_{-1}^{i}(y)\leq C(A)\exp\left(-\frac{c^{2}D^{2}}{C^{\ast}}\right),
\]
so that $D\leq C(A)$, a contradiction after adjusting $D=D(A)$.
$\square$

For any $\gamma\in(0,1]$, we note that $\widetilde{r}_{Rm}^{g^{i}}(\cdot,-1)\geq c(\gamma)$
on $\psi_{i}(A(o_{\ast},\gamma,\gamma^{-1}))$ $i=i(\gamma)\in\mathbb{N}$
is sufficiently large. 

\noindent \textbf{Claim 3: }For any $\gamma\in(0,1]$, we have $x\in M_{i,\gamma}'$
when $i=i(\gamma,A)\in\mathbb{N}$ is large.

Suppose not, so that after passing to a subsequence we have $x\in M_{i,16D}'\cap M_{i,\gamma}''$
for all $i\in\mathbb{N}$. Then we can pass to a subsequence to get
$\psi_{i}^{-1}(x)\to x_{\infty}\in A(o_{\ast},\frac{\gamma}{2},32D)$.
Because $\mathcal{Y}$ is static, we can argue as in the proof of
Theorem \ref{thm:theorem4} to get $\psi_{i,t}^{-1}(x)\to x_{\infty}(t)$
as $i\to\infty$, uniformly in $t\in[-1,-\beta]$ for $\beta\in(0,1)$
fixed. In particular, we can find $c(\gamma)>0$ such that $\widetilde{r}_{Rm}^{g^{i}}(x,-\beta)\geq c(\gamma)$,
so by taking $\beta=\beta(\gamma,A)$ sufficiently small, applying
Theorem \ref{thm:pseudolocality} gives a contradiction. $\square$

\noindent \textbf{Claim 4: }For any $\delta>0$, there exists $\gamma=\gamma(A,\delta)>0$
such that 

\noindent 
\[
\limsup_{i\to\infty}\text{diam}_{g_{-1}^{i}}(M_{i,\gamma}')<\delta.
\]

Set $\mathcal{B}_{i}:=M_{i,\gamma}'$, so that $\text{Area}_{g_{-1}^{i}}(\partial\mathcal{B}_{i})\leq2\text{Area}_{g_{-1}^{i}}(\partial B(o_{\ast},\gamma))\leq C(n)\gamma$
as $i\to\infty$. Thus we can apply Croke's isoperimetric inequality
(Theorem 13 of \cite{croke}) to the original (unrescaled) Riemannian
manifold $(M,g_{T-\tau_{i}})$ and then rescale to obtain
\[
\text{Area}_{g_{-1}^{i}}(\partial\mathcal{B}_{i})\geq c(A)\left(\min\{|\mathcal{B}_{i}|_{g_{-1}^{i}},|M\setminus\mathcal{B}_{i}|_{g_{-1}^{i}}\}\right)^{\frac{3}{4}}.
\]
On the other hand, for $i\in\mathbb{N}$ large, we know
\[
|M\setminus\mathcal{B}_{i}|_{g_{-1}^{i}}\geq|\psi_{i,-1}(A(o_{\ast},1,2))|_{g_{-1}^{i}}\geq c(\Gamma),
\]
hence $|\mathcal{B}_{i}|_{g_{-1}^{i}}\le C\gamma^{\frac{4}{3}}.$
For any $r>2C^{\frac{1}{4}}A^{\frac{1}{4}}\gamma^{\frac{1}{3}}$,
when $i=i(r)\in\mathbb{N}$ is sufficiently large, we have $B_{g^{i}}(y,-1,r)\cap\partial\mathcal{B}_{i}\neq\emptyset$
for all $y\in\mathcal{B}_{i}$, since otherwise $B_{g^{i}}(y,-1,r)\subseteq\mathcal{B}_{i}$,
in which case $A^{-1}r^{4}\leq|\mathcal{B}_{i}|_{g_{-1}^{i}}<C\gamma^{\frac{4}{3}},$a
contradiction. Thus
\[
\text{diam}_{g_{-1}^{i}}(\mathcal{B}_{i})\leq C(A)\gamma^{\frac{1}{3}}+\text{diam}_{g_{-1}^{i}}(\partial\mathcal{B}_{i})\leq2C(A)\gamma^{\frac{1}{3}}
\]
for sufficiently large $i\in\mathbb{N}$. $\square$

Claims 1-4 imply that $\sup_{i\in\mathbb{N}}d_{g_{-1}^{i}}(x,x_{i})<\infty$
for each $t\in[-1,0)$. Thus, given any subsequence of $(M,d_{g_{-1}^{i}},x)$,
we can pass to a further subsequence so that 
\[
\lim_{i\to\infty}d_{-1}^{Z}(\phi_{-1}^{i}(x),\phi_{-1}^{\infty}(x_{\infty}))=0,
\]
for some $x_{\infty}\in C(\mathbb{S}^{3}/\Gamma)$, hence
\[
(M,d_{g_{-1}^{i}},x)\to(C(\mathbb{S}^{3}/\Gamma),d_{C(\mathbb{S}^{3}/\Gamma)},x_{\infty})
\]
in the pointed Gromov-Hausdorff sense. Claims 3,4 imply that for any
$y\in\partial B(o_{\ast},2\gamma)$, 
\begin{align*}
d(x_{\infty},y)= & d_{-1}^{Z}(\phi_{-1}^{\infty}(x_{\infty}),\phi_{-1}^{\infty}(y))\\
\leq & d_{-1}^{Z}\left(\phi_{-1}^{i}(x),\phi_{-1}^{i}(\psi_{i}(y))\right)+d_{-1}^{Z}\left(\phi_{-1}^{i}(\psi_{i}(y)),\phi_{-1}^{\infty}(y)\right)+d_{-1}^{Z}\left(\phi_{-1}^{\infty}(x_{\infty}),\phi_{-1}^{i}(x)\right)\\
\leq & \text{diam}_{g_{-1}^{i}}(M_{i,2\gamma}')+\sup_{z\in U_{i}\cap A(o_{\ast},\gamma,4\gamma)}d_{-1}^{Z}\left(\phi_{-1}^{i}(\psi_{i}(z)),\phi_{-1}^{\infty}(z)\right)+d_{-1}^{Z}\left(\phi_{-1}^{\infty}(x_{\infty}),\phi_{-1}^{i}(x)\right).
\end{align*}
Taking $i\to\infty$ and appealing to Theorem 9.31(d) of \cite{bamlergen2},
we have
\[
d(x_{\infty},y)\leq\liminf_{i\to\infty}\text{diam}_{g_{-1}^{i}}(M_{i,2\gamma}')\leq C(A)\gamma^{\frac{1}{3}}.
\]
Thus $d(o_{\ast},x_{\infty})\leq2C(A)\gamma^{\frac{1}{3}}$, but $\gamma>0$
was arbitrary, so $x_{\infty}=o_{\ast}$.
\end{proof}
\begin{lem}
\label{lem:part2lemma} For any $x\in M$ corresponding to a singular
point $\overline{x}\in\mathcal{S},$ there exists a finite subgroup
$\Gamma\leq O(4,\mathbb{R})$ only depending on $x$ such that
\[
\lim_{t\nearrow T}d_{PGH}\left((M,(T-t)^{-\frac{1}{2}}d_{g_{t}},x),(C(\mathbb{S}^{3}/\Gamma),d_{C(\mathbb{S}^{3}/\Gamma)},o_{\ast})\right)=0,
\]
where $o_{\ast}$ is the vertex of $C(\mathbb{S}^{3}/\Gamma)$. 
\end{lem}

\begin{proof}
We first show that, for any $\epsilon>0$, there exists $\delta>0$
such that for any $t\in(T-\delta,T)$, there exists a finite subgroup
$\Gamma_{t}$ (possibly depending on $t$) such that
\[
d_{PGH}\left((M,(T-t)^{-\frac{1}{2}}d_{g_{t}},x),(C(\mathbb{S}^{3}/\Gamma_{t}),d_{C(\mathbb{S}^{3}/\Gamma_{t})},o_{\ast})\right)<\epsilon,
\]
and where $C(\mathbb{S}^{3}/\Gamma_{t})$ models a tangent flow of
$(M,(g_{t})_{t\in[0,T)})$ at $(x,T)$. Suppose this does not hold,
so that there are $t_{i}\nearrow T$ such that 
\[
d_{PGH}\left((M,(T-t_{i})^{-\frac{1}{2}}d_{g_{t_{i}}},x),(C(\mathbb{S}^{3}/\Gamma),d_{C(\mathbb{S}^{3}/\Gamma)},o_{\ast})\right)\geq\epsilon
\]
for all $i\in\mathbb{N}$ and all finite subgroups $\Gamma\leq O(4,\mathbb{R})$.
Let $K\in C^{\infty}(M\times(0,T))$ be a conjugate heat kernel at
$x$ based at the singular time $T$, and let $K^{i}\in C^{\infty}(M\times(0,T))$
be the conjugate heat kernels corresponding to the rescaled solutions
$g_{t}^{i}:=(T-t_{i})^{-1}g_{T+(T-t_{i})t}$. Write $\nu_{t}^{i}:=K^{i}(\cdot,t)dg_{t}$.
After passing to a subsequence, we have $\mathbb{F}$-convergence
\[
(M,(g_{t}^{i})_{t\in(-2,0)},(\nu_{t}^{i})_{t\in(-2,0)})\xrightarrow[i\to\infty]{\mathbb{F},\mathfrak{C}}(\mathcal{X},(\mu_{t}^{\infty})_{t\in(-2,0)})
\]
where $\mathcal{X}$ is the metric flow corresponding to the static
flow on $C(\mathbb{S}^{3}/\Gamma)$ for some finite subgroup $\Gamma\leq O(4,\mathbb{R})$.
From the previous section, we know this implies 
\[
(M,d_{g_{-1}^{i}},x)\to(C(\mathbb{S}^{3}/\Gamma),d_{C(\mathbb{S}^{3}/\Gamma)},o_{\ast})
\]
in the pointed Gromov-Hausdorff sense, a contradiction.

Now, recall that if $C(\mathbb{S}^{3}/\Gamma)$ models a tangent flow
at $(x,T)$, and if $\mathcal{N}(\tau)$ is the pointed Nash entropy
corresponding to $K$, then 
\[
\log\left(\frac{1}{|\Gamma|}\right)=\lim_{\tau\searrow0}\mathcal{N}(\tau),
\]
so if we take $N(x)$ to be any integer greater than $\exp\left(-\lim_{\tau\searrow0}\mathcal{N}(\tau)\right)$,
then we obtain
\[
\lim_{t\nearrow T}\inf_{|\Gamma|\leq N}d_{PGH}\left((M,(T-t)^{-\frac{1}{2}}d_{g_{t}},x),(C(\mathbb{S}^{3}/\Gamma),d_{C(\mathbb{S}^{3}/\Gamma)},o_{\ast})\right)=0,
\]
where the infimum is taken over all finite subgroups $\Gamma\leq O(4,\mathbb{R})$
with $|\Gamma|\leq N$. Finally, we note that there are only finitely
many finite subgroups of $O(4,\mathbb{R})$ with $|\Gamma|\leq N$
up to conjugation (in fact, embeddings of a finite group $G\hookrightarrow O(4,\mathbb{R})$
correspond to real representations $G\curvearrowright\mathbb{R}^{4}$,
of which there are finitely many up to isomorphism, and if two representations
of $G$ are isomorphic, then the corresponding subgroups of $O(4,\mathbb{R})$
are conjugate). Also, $C(\mathbb{S}^{3}/\Gamma)$, $C(\mathbb{S}^{3}/\Gamma')$
are isometric whenever $\Gamma,\Gamma'$ are conjugate in $O(4,\mathbb{R})$,
so there is a fixed distance $d>0$ such that any non-isometric $(C(\mathbb{S}^{3}/\Gamma),d_{C(\mathbb{S}^{3}/\Gamma)},o_{\ast}),(C(\mathbb{S}^{3}/\Gamma'),d_{C(\mathbb{S}^{3}/\Gamma')},o_{\ast}')$
with $|\Gamma|,|\Gamma'|\leq N$ satisfy
\[
d_{PGH}\left((B(o_{\ast},1),d_{C(\mathbb{S}^{3}/\Gamma)},o_{\ast}),(B(o_{\ast}',1),d_{C(\mathbb{S}^{3}/\Gamma')},o_{\ast}')\right)\geq d.
\]
However, $t\mapsto(B_{g}(x,t,1),(T-t)^{-\frac{1}{2}}d_{g_{t}},x)$
is continuous in the pointed Gromov-Hausdorff topology, so the conjugacy
class of $\Gamma_{t}$, hence the isometry class of $C(\mathbb{S}^{3}/\Gamma_{t})$,
must be constant for $t\in(T-\delta,T)$ for sufficiently small $\delta>0$. 
\end{proof}
\begin{rem}
A similar idea to the above argument for showing that $\Gamma$ does
not change along the subsequence was used in Proposition 2.2 of \cite{zilubamler}.
\end{rem}

\begin{proof}[Proof of Theorem \ref{thm:theorem2}]
 Part $(ii)$ is exactly the statement of Lemma \ref{lem:part2lemma}.
For part $(i)$, Proposition \ref{prop:tangentflowricciflat} tells
us that any tangent flow is a static flow corresponding to $C(\mathbb{S}^{3}/\Gamma_{x})$
for some finite subgroup $\Gamma_{x}\leq O(4,\mathbb{R})$. However,
Proposition \ref{prop:part1lemma} implies that the time $-1$ time-slices
of the rescaled flows also Gromov-Hausdorff converge to $C(\mathbb{S}^{3}/\Gamma_{x})$,
so $\mathbb{S}^{3}/\Gamma$ and $\mathbb{S}^{3}/\Gamma_{x}$ are isometric,
hence we can assume $\Gamma=\Gamma_{x}$. 
\end{proof}
\begin{proof}[Proof of Theorem \ref{thm:theorem1}]
 Let $x\in M$ correspond to a singular point $\overline{x}\in\mathcal{S}$.
By Lemma \ref{lem:part2lemma}, there exists $\delta=\delta(x)>0$
and a subgroup $\Gamma\leq O(4,\mathbb{R})$ such that
\[
\lim_{t\nearrow T}d_{PGH}\left((M,(T-t)^{-\frac{1}{2}}d_{g_{t}},x),(C(\mathbb{S}^{3}/\Gamma),d_{C(\mathbb{S}^{3}/\Gamma)},o_{\ast})\right),
\]
where $o_{\ast}$ is the vertex of $C(\mathbb{S}^{3}/\Gamma)$. 

\noindent \textbf{Claim 1: }For any $\beta\in(0,1)$, there exists
$\delta=\delta(x,\beta)>0$ such that $\widetilde{r}_{Rm}^{g}(y,t)\geq\frac{1}{6}d_{g_{t}}(y,x)$
for all $t\in(T-\delta,T)$ and $y\in B_{g}(x,t,\beta^{-1}\sqrt{T-t})\setminus\overline{B}_{g}(x,t,\beta\sqrt{T-t})$. 

Suppose by way of contradiction that $t_{i}\nearrow T$ and 
\[
y_{i}\in B_{g}(x,t_{i},\beta^{-1}\sqrt{T-t_{i}})\setminus\overline{B}_{g}(x,t_{i},\beta\sqrt{T-t_{i}})
\]
 satisfy
\[
\frac{\widetilde{r}_{Rm}^{g}(y_{i},t_{i})}{d_{g_{t_{i}}}(y_{i},x)}\leq\frac{1}{6}.
\]
and let $\phi_{i}:(B(o_{\ast},\alpha_{i}^{-1}),d,o_{\ast})\to(B_{g}(x,t_{i},\alpha_{i}^{-1}(T-t_{i})^{\frac{1}{2}}),(T-t_{i})^{-\frac{1}{2}}d_{g_{t_{i}}},x)$
be $\alpha_{i}$-Gromov-Hausdorff maps for some sequence $\alpha_{i}\searrow0$.
By hypothesis, we have $(T-t_{i})^{-\frac{1}{2}}d_{g_{t_{i}}}(x,y_{i})\in(\beta,\beta^{-1})$,
so we can pass to a subsequence so that 
\[
D:=\lim_{i\to\infty}(T-t_{i})^{-\frac{1}{2}}d_{g_{t_{i}}}(x,y_{i})\in[\beta,\beta^{-1}]
\]
exists. Also choose $y_{i}'\in C(\mathbb{S}^{3}/\Gamma)$ with $(T-t_{i})^{-\frac{1}{2}}d_{g_{t_{i}}}(\phi_{i}(y_{i}'),y_{i})\to0$,
so that $d(y_{i}',o_{\ast})\to D$ as $i\to\infty$. Then $B(y_{i}',\frac{D}{2})\subseteq C(\mathbb{S}^{3}/\Gamma)\setminus\{o_{\ast}\}$,
and

\[
B_{g}(y_{i},t_{i},\frac{D}{3}\sqrt{T-t_{i}})\subseteq\psi_{i}\left(B(y_{i}',D/2)\right)
\]
for large $i\in\mathbb{N}$, so because $C(\mathbb{S}^{3}/\Gamma)\setminus\{o_{\ast}\}$
is flat and (by Theorem \ref{thm:openandsmooth}) we have locally
uniform smooth convergence on $C(\mathbb{S}^{3}/\Gamma)\setminus\{o_{\ast}\}$,
\[
\lim_{i\to\infty}\sup_{B_{g}(y_{i},t_{i},D/3)}|Rm|_{g}(\cdot,t_{i})(T-t_{i})=0.
\]
This means $\widetilde{r}_{Rm}^{g}(y_{i},t_{i})\geq\frac{D}{4}(T-t_{i})^{\frac{1}{2}}$
for sufficiently large, hence 
\begin{align*}
\frac{1}{6}\geq\frac{\widetilde{r}_{Rm}^{g}(y_{i},t_{i})}{d_{g_{t_{i}}}(y_{i},x)}\geq\frac{\widetilde{r}_{Rm}^{g}(y_{i},t_{i})}{\sqrt{T-t_{i}}}\cdot\frac{\sqrt{T-t_{i}}}{d_{g_{t_{i}}}(y_{i},x)}\geq\frac{D}{4}\cdot\frac{4}{5D},
\end{align*}
for large $i\in\mathbb{N},$ a contradiction. $\square$

Now set $r:=r(A):=8\max\{\epsilon_{P}^{-1},e^{4}\}<\infty$, where
$\epsilon_{P}(A)>0$ is as in Theorem \ref{thm:pseudolocality}. Also
fix $E\geq2$, so that for $t\in(T-\delta,T)$ and $y\in B_{g}(x,t,Er\sqrt{T-t})\setminus\overline{B}_{g}(x,t,\frac{1}{2}r\sqrt{T-t})$,
we have

\[
|Rm|_{g}(y,s)\leq\frac{1}{T-t}.
\]
for all $s\in(t,T)$. Here we have fixed $\delta=\delta(x,r^{-1}E^{-1})>0$
as in Claim 1. In particular, $y\in M\setminus\Sigma$, and $(g_{t})_{t\in[0,T)}$
extends smoothly to the Riemannian metric $g_{T}$ on 
\[
B_{g}(x,t,Er\sqrt{T-t})\setminus\overline{B}_{g}(x,t,\frac{1}{2}r\sqrt{T-t})\subseteq M\setminus\Sigma
\]
whenever $t\in(T-\delta,T)$. Moreover, we have 
\[
|Rm|_{g_{T}}\leq\frac{1}{T-t}
\]
on this subset. Now suppose $y\in\partial B_{g}(x,t,r\sqrt{T-t})$,
where $t\in(T-\delta,T)$, and fix $z\in B_{g_{T}}(y,\sqrt{T-t})$.
Let $\gamma:[0,1]\to M\setminus\Sigma$ be a curve from $y$ to $z$
with $\text{length}_{g_{T}}(\gamma)\leq2\sqrt{T-t}$. If $\gamma([0,1])\subseteq B_{g}(y,t,\frac{r}{2}\sqrt{T-t})$,
then $|Rm|_{g}(\gamma(u),s)\leq\frac{1}{T-t}$ for all $u\in[0,1]$
and $s\in[t,T]$, so standard distortion estimates give 
\[
d_{g_{t}}(y,z)\leq\text{length}_{g_{t}}(\gamma)\leq e^{4}\text{length}_{g_{T}}(\gamma)\leq2e^{4}\sqrt{T-t}.
\]
If $\gamma([0,1])\not\subseteq B_{g}(y,t,\frac{r}{2}\sqrt{T-t})$,
then set 
\[
u_{\ast}:=\inf\{u\in[0,1];d_{g_{t}}(y,\gamma(u))=\frac{r}{2}\sqrt{T-t}\},
\]
and again apply distortion estimates, this time concluding
\[
\frac{r}{2}\sqrt{T-t}=d_{g_{t}}(y,\gamma(u))\leq e^{4}\text{length}_{g_{T}}(\gamma|[0,u_{\ast}])\leq2e^{4}\sqrt{T-t},
\]
contradicting our choice of $r$. We thus conclude that $\widetilde{r}_{Rm}^{(M\setminus\Sigma,g_{T})}(y)\geq\sqrt{T-t}$. 

Now let $y\in B_{g}(x,T-\delta/2,r\sqrt{\delta/2})$ satisfy $\overline{y}\neq\overline{x}$,
and define
\[
D(t):=\frac{d_{g_{t}}(x,y)}{\sqrt{T-t}}.
\]
Then $\lim_{t\nearrow T}D(t)=\infty$ and $D(T-\delta/2)<r$, so there
exists $t=t(y)\in(T-\delta/2,T)$ such that $D(t)=r$. Because $2d_{g_{t}}(x,y)\geq d_{X}(\overline{x},\overline{y})$
for all $t\in(T-\delta,T)$, 
\[
r_{Rm}^{X}(\overline{y})=r_{Rm}^{(M\setminus\Sigma,g_{T})}(y)\geq\sqrt{T-t(y)}=r^{-1}d_{g_{t(y)}}(x,y)\geq c(A)d_{X}(\overline{x},\overline{y}).
\]
\textbf{Claim 2: }If $y\in M\setminus B_{g}(x,t,r\sqrt{T-t})$ for
some $t\in(T-\delta/2,T)$, then $\overline{y}\neq\overline{x}$.

Fix $\epsilon>0$, and let $\gamma:[0,1]\to M$ be any curve from
$x$ to $y$. Set 
\[
u_{-}:=\sup\left\{ u\in[0,1];d_{g_{t}}(\gamma(u),x)=\frac{r}{2}\sqrt{T-t}\right\} ,
\]
\[
u_{+}:=\inf\left\{ u\in[0,1];d_{g_{t}}(\gamma(u),x)=r\sqrt{T-t}\right\} ,
\]
so that $\text{length}_{g_{t}}(\gamma|[u_{-},u_{+}])\geq\frac{r}{2}\sqrt{T-t}$,
and $|Rm|(\gamma(u),s)\leq\frac{1}{T-t}$ for all $u\in[u_{-},u_{+}]$
and $s\in[t,T)$. In particular, we can estimate 

\[
\text{length}_{g_{s}}(\gamma)\geq\text{length}_{g_{s}}(\gamma|[s_{-},s_{+}])\geq\frac{r}{2}e^{-4}\sqrt{T-t},
\]
so taking the infimum over all curves $\gamma$ gives $d_{g_{s}}(x,y)\geq c(x,A)$,
and taking $s\nearrow T$ gives $d_{X}(\overline{x},\overline{y})\geq c(x,t,A)>0$.
$\square$

By Claim 2, $B_{g}(x,T-\delta/2,\sqrt{\delta/2}r)$ is a saturated
open set with respect to the equivalence relation $\sim$, so if $\pi:M\to X$
is the quotient map, then
\[
U:=\pi\left(B_{g}(x,T-\delta/2,\sqrt{\delta/2}r)\right)
\]
is a neighborhood of $\overline{x}$ in $X$ such that $U\cap\mathcal{S}(X)=\{\overline{x}\}$.
Moreover, we have $\widetilde{r}_{Rm}^{X}(\overline{y})\geq c(A)d(\overline{x},\overline{y})$
for all $\overline{y}\in U\setminus\{\overline{x}\}$, so by applying
Theorem \ref{thm:pseudolocality} and then Shi's local derivative
estimates on 

\[
B_{g_{T}}(y,c(A)\epsilon_{P}d(\overline{x},\overline{y}))\times[T-(c(A)\epsilon_{P}d(\overline{x},\overline{y}))^{2},T],
\]
we obtain $C_{k}=C_{k}(A)<\infty$ for each $k\in\mathbb{N}$ such
that 
\[
|\nabla^{k}Rm|_{g}(\overline{y})\leq C_{k}d^{-2-k}(\overline{x},\overline{y})
\]
for all $\overline{y}\in U\setminus\{\overline{x}\}$. Thus, for any
tangent cone $(C(Y),c_{\ast})$ of $(X,d)$ at $\overline{x}$, convergence
is smooth away from the vertex $c_{\ast}$, and in particular $C(Y)$
must have smooth link, so $C(Y)$ cannot split any factor of $\mathbb{R}$
(otherwise the vertex would produce a line of nonsmooth points, a
contradiction). In particular, $\mathcal{S}(X)=\mathcal{S}^{0}(X)$.
Moreover, we showed that $\mathcal{S}(X)$ is discrete, so by the
compactness of $X$, $\mathcal{S}(X)$ must be finite. 

\noindent \textbf{Claim 3: }$C(Y)\setminus\{c_{\ast}\}$ is flat.

\noindent Fix $x_{0}=[r,y_{0}]\in(C(Y)\setminus\{c_{\ast}\})$ arbitrary
(see Section 2 for notation), and set $v:=\partial_{r}|_{x_{0}}$
(where $\partial_{r}$ is the radial vector field corresponding to
$c_{\ast}$) which is a zero eigenvector for $Rc(g_{C(Y)})|_{x_{0}}$.
By a standard diagonal argument, $C(Y)$ is itself a Gromov-Hausdorff
limit of appropriate dilations of $(M,d_{g_{t_{i}}},x_{i})$ for some
$(x_{i},t_{i})\in M\times[0,T)$, so by Theorems \ref{thm:pseudolocality},
\ref{thm:openandsmooth}, Shi's estimates, and the Hamilton-Cheeger-Gromov
compactness theorem, there is a neighborhood $U_{0}$ of $x_{0}$
in $C(Y)\setminus\{c_{\ast}\}$, and some $\delta_{0}>0$ such that
we can extract a Ricci flow $(U_{0},(\overline{g}_{t})_{t\in(-\delta_{0},0]})$
satisfying $\overline{g}_{0}=g_{C(Y)}$ on $U_{0}$ as well as $Rc(\overline{g}_{t})\geq0$
for all $t\in(-\delta,0]$. By possibly shrinking $U_{0}$, we can
extend $v$ to a unit-length vector field $V$ on $(U_{0},\overline{g}_{0})$
by parallel translation along radial geodesics emanating from $x_{0}$,
and then extend to a vector field on $U_{0}\times(-\delta_{0},0]$
by letting $V$ be constant in time. Then $\phi:=Rc_{\overline{g}_{0}}(V,V)\in C^{\infty}(U_{0}\times(-\delta_{0},0])$
satisfies $\phi\geq0$ on $U\times(-\delta,0]$, with $\phi(x_{0},t_{0})=0$.
We observe that $Rc_{\overline{g}_{0}}(\partial_{r})=0$ and $Rm_{\overline{g}_{0}}(\partial_{r},\cdot,\cdot,\cdot)=0$,
hence at $(x_{0},0)$, we have (where $(\partial_{r},e_{1},e_{2},e_{3})$
is an orthonormal basis of $T_{x_{0}}(C(Y)\setminus\{c_{\ast}\})$)
\begin{align*}
(\partial_{t}-\Delta)\phi(x_{0},0)= & (\partial_{t}Rc_{\overline{g}})(v,v)-(\Delta Rc_{\overline{g}})(v,v)\\
= & 2\sum_{j,k=1}^{3}Rm_{\overline{g}_{0}}(e_{j},\partial_{r},\partial_{r},e_{k})Rc_{\overline{g}_{0}}(e_{j},e_{k})-2\sum_{j=1}^{3}Rc_{\overline{g}_{0}}(e_{j},\partial_{r})Rc_{\overline{g}_{0}}(e_{j},\partial_{r})=0.
\end{align*}
Because $\partial_{t}\phi(x_{0},0)\leq0$ and $\Delta\phi(x_{0},0)\geq0$,
we have (abbreviating $Rc=Rc(g_{C(Y)})$)
\[
0=\Delta\phi(x_{0},0)=(\Delta Rc)_{x_{0}}(\partial_{r},\partial_{r}).
\]
Since $x_{0}\in C(Y)\setminus\{c_{\ast}\}$ was arbitrary, we obtain
$(\Delta Rc)(\partial_{r},\partial_{r})=0$ everywhere. Now once again
fix $x_{0}\in C(Y)\setminus\{c_{\ast}\}$, and choose an orthonormal
frame $(E_{i})_{i=0}^{3}$ for $g_{C(Y)}$ in a neighborhood $U_{0}$
of $x_{0}$ with $E_{0}=\partial_{r}$. For any $W\in\mathfrak{X}(U_{0})$,
we compute 
\[
(\nabla_{W}Rc)(\partial_{r},\partial_{r})=-2Rc(\nabla_{W}\partial_{r},\partial_{r})=0
\]
on $U_{0}$ since $Rc(\partial_{r})=0$. Because $\nabla_{E_{i}}\partial_{r}=\frac{1}{r}E_{i}$
for $i=1,2,3$, we have (at $x_{0}$)
\begin{align*}
0= & \left(\nabla_{\partial_{r}}(\nabla_{\partial_{r}}Rc)\right)(\partial_{r},\partial_{r})+\sum_{i=1}^{3}\left(\nabla_{E_{i}}(\nabla_{E_{i}}Rc)\right)(\partial_{r},\partial_{r})-\sum_{i=1}^{3}(\nabla_{\nabla_{E_{i}}E_{i}}Rc)(\partial_{r},\partial_{r})\\
= & -2\sum_{i=1}^{3}(\nabla_{E_{i}}Rc)(\nabla_{E_{i}}\partial_{r},\partial_{r})=2\sum_{i=1}^{3}Rc(\nabla_{E_{i}}\partial_{r},\nabla_{E_{i}}\partial_{r})\\
= & \frac{2}{r^{2}}\sum_{i=1}^{3}Rc(E_{i},E_{i}).
\end{align*}
Because $Rc\geq0$ on $C(Y)\setminus\{c_{\ast}\}$, we conclude that
$C(Y)\setminus\{c_{\ast}\}$ is Ricci-flat, hence $Rc(g_{Y})=2g_{Y}$.
Because $\dim(Y)=3$, this is only possible if $(Y,g_{Y})$ has constant
sectional curvature $1$, so that $C(Y)\setminus\{c_{\ast}\}$ is
flat. $\square$

Because the Gromov-Hausdorff convergence of dilations of $(X,d_{X},\overline{x})$
to $(C(Y),d_{C(Y)},c_{\ast})$ is smooth away from the vertex, and
because $C(Y)$ was an arbitrary tangent cone at $\overline{x}$,
we conclude that
\[
\lim_{\overline{y}\to\overline{x}}d^{2+k}(\overline{x},\overline{y})|\nabla^{k}Rm|(\overline{y})=0.
\]
Now one can show (see, for example, the proof of Theorem 7.2 in \cite{typeIscalar})
that there exists a finite subgroup $\Gamma'\leq O(4,\mathbb{R})$
such that any tangent cone of $(X,d)$ at $\overline{x}$ is $C(\mathbb{S}^{3}/\Gamma')$.
Gluing together Cheeger-Gromov diffeomorphisms on dyadic annuli as
in the proof of Proposition 3.2 of \cite{tiancalabi}, we get that
$X$ has the structure of a $C^{0}$ orbifold (which is $C^{\infty}$
away from the singular points) with finitely many conical orbifold
singularities. In fact, since we know that every tangent cone at $\overline{x}$
is $C(\mathbb{S}^{3}/\Gamma')$, with smooth convergence away from
the vertex, we can appeal to Step 1 of Theorem 5.7 in \cite{donaldsun1}
verbatim. It remains only to establish the following claim.

\noindent \textbf{Claim 4: }$C(\Gamma)=C(\Gamma')$. 

In fact, for any fixed $\beta\in(0,1)$, we have (writing $t=T-\beta\tau$)

\begin{align*}
d_{PGH} & \left((B_{g}(x,T-\beta\tau,\tau^{\frac{1}{2}}),\tau^{-\frac{1}{2}}d_{g_{T-\beta\tau}},x),(B(o_{\ast},1),d_{C(\mathbb{S}^{3}/\Gamma)},o_{\ast})\right)\\
 & =d_{PGH}\left((B_{g}(x,T-\beta\tau,\beta^{-\frac{1}{2}}(\beta\tau)^{\frac{1}{2}}),\beta^{\frac{1}{2}}(\beta\tau)^{-\frac{1}{2}}d_{g_{T-\beta\tau}},x),(B(o_{\ast},\beta^{-\frac{1}{2}}),\beta^{\frac{1}{2}}d_{C(\mathbb{S}^{3}/\Gamma)},o_{\ast})\right)\\
 & =\beta^{\frac{1}{2}}d_{PGH}\left((B_{g}(x,t,\beta^{-\frac{1}{2}}\sqrt{T-t}),(T-t)^{-\frac{1}{2}}d_{g_{t}},x),(B(o_{\ast},\beta^{-\frac{1}{2}}),d_{C(\mathbb{S}^{3}/\Gamma)},o_{\ast})\right)\to0
\end{align*}
as $\tau\searrow0$ (or equivalently as $t\nearrow T$). By a diagonal
argument, we can therefore find sequences $\beta_{j}\searrow0$ and
$\tau_{j}\searrow0$ such that 
\[
\lim_{j\to\infty}d_{PGH}\left((B_{g}(x,T-\beta_{j}\tau_{j},\tau_{j}^{\frac{1}{2}}),\tau_{j}^{-\frac{1}{2}}d_{g_{T-\beta_{j}\tau_{j}}},x),(B(o_{\ast},1),d_{C(\mathbb{S}^{3}/\Gamma)},o_{\ast})\right)=0,
\]
\[
\lim_{j\to\infty}d_{PGH}\left((B^{X}(\overline{x},\tau_{j}^{\frac{1}{2}}),\tau_{j}^{-\frac{1}{2}}d,\overline{x}),(B(o_{\ast}',1),d_{C(\mathbb{S}^{3}/\Gamma')},o_{\ast}')\right)=0,
\]
so it suffices to prove that if $g^{j}:=\tau_{j}^{-1}g_{T-\beta_{j}\tau_{j}}$
and $d_{j}:=\tau_{j}^{-\frac{1}{2}}d$, then 
\[
\lim_{j\to\infty}d_{PGH}\left((B_{g^{j}}(x,1),d_{g^{j}},x),(B^{(X,d_{j})}(\overline{x},1),d_{j},\overline{x})\right)=0.
\]
Recall that $d_{j}|(\mathcal{R}_{X}\times\mathcal{R}_{X})$ is equal
to the length metric of $(M\setminus\Sigma,\tau_{j}^{-1}g_{T})$,
and $\pi:M\to X=M/\sim$ is given by $\pi(x)=\overline{x}$. For all
$y,z\in B_{g^{j}}(x,1)$, we have
\begin{align*}
d_{j}(\pi(y),\pi(z))-d_{g^{j}}(y,z)\leq & \tau_{j}^{-\frac{1}{2}}(e^{A\beta_{j}\tau_{j}}-1)d_{g_{T-\beta_{j}\tau_{j}}}(y,z)\\
\leq & 2\tau_{j}^{-1}\cdot A\beta_{j}d_{g^{j}}(y,z)\leq\beta_{j}.
\end{align*}

Because $(B_{g^{j}}(x,2),d_{g^{j}},x)\to(B(o_{\ast},2),d_{C(\mathbb{S}^{3}/\Gamma)},o_{\ast})$
in the pointed Gromov-Hausdorff sense, with smooth convergence away
from the vertex (by Theorem \ref{thm:openandsmooth}), we have the
following: for any $\epsilon>0$, there exists $\sigma=\sigma(\epsilon)>0$
such that for $j=j(\epsilon)\in\mathbb{N}$ sufficiently large, we
have $\widetilde{r}_{Rm}^{g^{j}}(y)>\sigma$ for all $y\in B_{g^{j}}(x,2)\setminus B_{g^{j}}(x,\epsilon)$.
Now fix $y_{1},y_{2}\in B_{g^{j}}(x,2)\setminus B_{g^{j}}(x,\epsilon)$.
For $j=j(\sigma)\in\mathbb{N}$ sufficiently large, we can apply Theorem
\ref{thm:pseudolocality} on $B_{g^{j}}(y_{k},\sigma)$ for each $u\in[0,1]$
to get
\[
\sup_{B_{g}(y_{k},T-\beta_{j}\tau_{j},\epsilon_{P}\sigma\tau_{j}^{\frac{1}{2}})\times[T-\beta_{j}\tau_{j},T)}|Rm|\leq\tau_{j}^{-1}(\epsilon_{P}\sigma)^{-2}
\]
for $k=1,2$ and all sufficiently large $j=j(\sigma)\in\mathbb{N}$.
By an easy distortion estimate,
\[
\sup_{s\in[T-\beta_{j}\tau_{j},T)}\sup_{B_{g}(y_{k},s,e^{-4}\epsilon_{P}\sigma\tau_{j}^{\frac{1}{2}})}|Rm|\leq\tau_{j}^{-1}(\epsilon_{P}\sigma)^{-2}
\]
for $k=1,2$ and large $j=j(\sigma)\in\mathbb{N}$. We can thus apply
the standard lower bound for distance distortion (Theorem 18.7 in
\cite{chowbook3}) with $r_{0}=e^{-4}\epsilon_{P}\sigma\tau_{j}^{\frac{1}{2}}$
and $K=\tau_{j}^{-1}(\epsilon_{P}\sigma)^{-2}$ to get 
\begin{equation}
\partial_{s}d_{g_{s}}(y_{1},y_{2})\geq-6\left(Kr_{0}+\frac{1}{r_{0}}\right)\geq-c(\sigma,A)\tau_{j}^{-\frac{1}{2}}\label{eq:distderiv}
\end{equation}
for all $s\in[T-\beta_{j}\tau_{j},T)$. Integrating from $s=T-\beta_{j}\tau_{j}$
to $T$ gives
\[
d_{g_{T}}(y_{1},y_{2})-d_{T-\beta_{j}\tau_{j}}(y_{1},y_{2})\geq-c(\sigma,A)\beta_{j}\tau_{j}^{\frac{1}{2}},
\]
or equivalently
\[
d_{j}(\pi(y_{1}),\pi(y_{2}))-d_{g^{j}}(y_{1},y_{2})\geq-c(\sigma,A)\beta_{j}.
\]
Now fix $\epsilon>0$, and suppose $\overline{y}\in B^{(X,d_{j})}(\overline{x},1-3\epsilon)$.
Since $d_{j}(\overline{x},\overline{y})=\lim_{t\nearrow T}\tau_{j}^{-\frac{1}{2}}d_{g_{t}}(x,y)$,
for any $j\in\mathbb{N}$, we can find $t_{j}^{\ast}\in(T-\beta_{j}\tau_{j},T)$
satisfying $\tau_{j}^{-\frac{1}{2}}d_{g_{t_{j}^{\ast}}}(x,y)<1-3\epsilon$.
Let let $\gamma:[0,1]\to M$ be any $g_{t_{j}^{\ast}}$-minimizing
geodesic from $x$ to $y$. Suppose by way of contradiction that $\text{length}_{g^{j}}(\gamma)\geq1$,
and set
\[
u_{-}:=\sup\left\{ u\in[0,1];d_{g^{j}}(\gamma(u),x)=\epsilon\right\} ,
\]
\[
u_{+}:=\inf\left\{ u\in[0,1];d_{g^{j}}(\gamma(u),x)=1-\epsilon\right\} .
\]
Also let $\sigma=\sigma(\epsilon)>0$ be as above, so that integrating
(\ref{eq:distderiv}) again gives
\begin{align*}
\tau_{j}^{-\frac{1}{2}}d_{g_{t_{j}^{\ast}}}(x,y)\geq\tau_{j}^{-\frac{1}{2}} & d_{g_{t_{j}^{\ast}}}(\gamma(u_{1}),\gamma(u_{2}))\geq d_{g^{j}}(\gamma(u_{1}),\gamma(u_{2}))-c(\sigma,A)\beta_{j}\\
\geq & 1-2\epsilon-c(\sigma,A)\beta_{j},
\end{align*}
a contradiction for sufficiently large $j=j(\epsilon,A)\in\mathbb{N}$
(independent of $y$ and $\gamma$). In particular, $y\in B_{g^{j}}(x,1)$
for $j=j(\epsilon,A)\in\mathbb{N}$ sufficiently large, so that 
\[
\pi(B_{g^{j}}(x,1))\supseteq B^{(X,d_{j})}(\overline{x},1-3\epsilon).
\]
For any fixed $\epsilon>0$, $\pi$ is thus a $3\epsilon$-Gromov-Hausdorff
map
\[
B_{g^{j}}(x,1)\to B^{(X,d_{j})}(\overline{x},1)
\]
for $j=j(\epsilon)\in\mathbb{N}$ sufficiently large. Since $\epsilon>0$
was arbitrary, the claim follows.
\end{proof}
\begin{rem}
With small modifications, the proof of Claim 4 can be used to prove
Theorem 1.1 using Lemma \ref{lem:part2lemma} (without Claims 1-3).
However, Claims 1-3 or their proofs will be referenced in later sections. 
\end{rem}

\section{Codimension Two $\epsilon$-Regularity}
\begin{prop}
\label{prop:codim2} (Codimension two $\epsilon$-Regularity) For
any $A<\infty$ and $\underline{T}\geq0$, there exists $\epsilon_{0}=\epsilon_{0}(A,\underline{T})>0$
such that the following holds. Suppose $(M^{n},(g_{t})_{t\in[0,T)})$
is a closed Ricci flow satisfying $Rc(g_{t})\geq-Ag_{t}$ and $|B(x,t,r)|_{g_{t}}\geq A^{-1}r^{n}$,
for all $(x,t)\in M\times[0,T)$ and $r\in(0,1]$, where $T\geq\underline{T}$.
Then for any $(x,t)\in M\times[\frac{T}{2},T)$ and $r\in(0,\epsilon_{0}\sqrt{T-t}]$,
if 
\[
d_{PGH}\left(\left(B(x,t,\epsilon_{0}^{-1}r),d_{g_{t}},x\right),\left(B\left((0^{n-2},z),\epsilon_{0}^{-1}r\right),d_{\mathbb{R}^{n-2}\times Z},(0^{n-2},z)\right)\right)<\epsilon_{0}r
\]
for some pointed metric space $(Z,d_{Z},z)$, then $\widetilde{r}_{Rm}(x,t)\geq\epsilon_{0}r$. 
\end{prop}

\begin{rem}
The first part of the proof is a modification of Theorem 5.2 in \cite{cheegernaberdim4}
and Lemma 13.5 in \cite{bam1}. 
\end{rem}

\begin{proof}
Suppose the claim is false. Then we can find closed Ricci flows $(M_{i}^{n},(g_{t}^{i})_{t\in[0,T_{i})})$
satisfying $Rc(g_{t}^{i})\geq-Ag_{t}^{i}$, $|B_{g^{i}}(x,t,r)|_{g_{t}^{i}}\geq A^{-1}r^{n}$
for all $(x,t)\in M_{i}\times[0,T_{i})$ and $r\in(0,1]$, where $T_{i}\geq\underline{T}$,
along with points $(x_{i},t_{i})\in M\times[\frac{T_{i}}{2},T_{i})$,
a sequence $\epsilon_{i}\searrow0$, and scales $r_{i}\in(0,\epsilon_{i}\sqrt{T_{i}-t_{i}})$,
such that 
\[
d_{PGH}\left(\left(B_{g^{i}}(x_{i},t_{i},\epsilon_{i}^{-1}r_{i}),d_{g_{t}^{i}},x_{i}\right),\left(B_{g^{i}}\left((0^{n-2},z_{i}),\epsilon_{i}^{-1}r_{i}\right),d,(0^{n-2},z_{i})\right)\right)<\epsilon_{i}r_{i}
\]
for some pointed metric spaces $(Z_{i},d_{i},z_{i})$, yet $\widetilde{r}_{Rm}^{g^{i}}(x_{i},t_{i})<\epsilon_{i}r_{i}$.
Define $\widetilde{g}_{t}^{i}:=r_{i}^{-2}g_{t_{i}+r_{i}^{2}t}^{i}$
for $t\in[-r_{i}^{-2}t_{i},r_{i}^{-2}(T-t_{i}))$, so that we can
pass to a subsequence to assume that $(M,d_{\widetilde{g}_{0}^{i}},x_{i})$
converges in the pointed Gromov-Hausdorff sense to some metric product
\[
(\mathbb{R}^{n-2}\times Z_{\infty},d_{\mathbb{R}^{n-2}\times Z_{\infty}},(0^{n-2},z_{\ast}))
\]
Because $|B_{\widetilde{g}^{i}}(x_{i},0,1)|_{\widetilde{g}_{0}^{i}}\geq A^{-1}$
and $Rc(\widetilde{g}_{0}^{i})\geq-(n-1)r_{i}^{2}\to0$ as $i\to\infty$,
we can apply the ``almost metric product implies existence of splitting
maps'' theorem (Theorem 9.29 of \cite{cheegernotes}) to obtain $\delta_{i}$-splitting
maps (see Definition 1.20 of \cite{cheegernaberdim4}) 
\[
u_{i}:B_{\widetilde{g}^{i}}(x_{i},0,2)\to\mathbb{R}^{n-2},
\]
where $\lim_{i\to\infty}\delta_{i}=0$. Now fix a sequence $\delta_{i}'\searrow0$
such that we can apply Cheeger-Naber's Slicing theorem (Theorem 1.23
of \cite{cheegernaberdim4}) with parameter $\delta_{i}'$ given $\delta_{i}$-splitting
maps. That is, there are subsets $G_{\delta_{i}'}\subseteq B^{\mathbb{R}^{n-2}}(0^{n-2},1)$
such that the following hold:

$(1)$ $|G_{\delta_{i}'}|_{g_{\mathbb{R}^{n-2}}}>\omega_{n-2}-\delta_{i}'$,

$(2)$ If $s\in G_{\delta_{i}'}$, then $u_{i}^{-1}(s)\cap B_{\widetilde{g}^{i}}(x_{i},0,1)\neq\emptyset$,

$(3)$ For each $z\in u_{i}^{-1}(G_{\delta_{i}'})$, and $r\in(0,1]$,
there is a lower triangular matrix $T\in GL(n-2,\mathbb{R})$ with
positive diagonal entries such that $T\circ u_{i}:B_{\widetilde{g}^{i}}(z,0,r)\to\mathbb{R}^{n-2}$
is a $\delta_{i}'$-splitting map.

Observe that, by applying this theorem on $B_{\widetilde{g}^{i}}(x_{i},0,\gamma)$
for some fixed $\gamma\in(0,1)$, and using the fact that for any
$\epsilon>0$, $u_{i}|B_{\widetilde{g}^{i}}(x_{i},0,\gamma)$ is an
$\epsilon$-splitting map for sufficiently large $i\in\mathbb{N}$,
a diagonal argument gives regular values $s_{i}\in\mathbb{R}^{n-2}$
of $u_{i}$ such that $(3)$ holds, and
\[
u_{i}^{-1}(s_{i})\cap B_{\widetilde{g}^{i}}(x_{i},0,\gamma_{i})\neq\emptyset,
\]
where $\gamma_{i}\searrow0$. Any $w_{i}\in u_{i}^{-1}(s_{i})\cap B_{\widetilde{g}^{i}}(x_{i},0,\gamma_{i})$
satisfy $\widetilde{r}_{Rm}^{\widetilde{g}^{i}}(w_{i},0)\leq\widetilde{r}_{Rm}(x_{i},0)+\gamma_{i}$,
so that 
\[
\min_{y\in u_{i}^{-1}(s_{i})\cap B_{\widetilde{g}^{i}}(x_{i},0,1)}\frac{\widetilde{r}_{Rm}^{\widetilde{g}^{i}}(y,0)}{d_{\widetilde{g}_{0}^{i}}\left(y,\partial B_{\widetilde{g}^{i}}(x_{i},0,1)\right)}\to0
\]
Choose $y_{i}\in u_{i}^{-1}(s_{i})\cap B_{\widetilde{g}^{i}}(x_{i},0,1)$
achieving the minima, set $\rho_{i}:=\widetilde{r}_{Rm}^{\widetilde{g}^{i}}(y_{i},0)$,
and define $\widehat{g}_{t}^{i}:=\rho_{i}^{-2}\widetilde{g}_{\rho_{i}^{2}t}^{i}$,
$t\in[(r_{i}\rho_{i})^{-2}t_{i},(r_{i}\rho_{i})^{-2}(T-t_{i}))$.
We know $\widetilde{r}_{Rm}^{\widehat{g}^{i}}(y_{i},0)=1$, $\lim_{i\to\infty}d_{\widehat{g}_{0}^{i}}(y_{i},\partial B_{\widetilde{g}^{i}}(x_{i},0,1))=\infty$,
and for any $z\in u_{i}^{-1}(s_{i})$ with $d_{\widehat{g}_{0}^{i}}(y,z)\leq\frac{1}{2}d_{\widehat{g}_{0}^{i}}(y_{i},\partial B_{\widetilde{g}^{i}}(x_{i},0,1))$,
we have
\[
\widetilde{r}_{Rm}^{\widehat{g}^{i}}(z,0)=\frac{\widetilde{r}_{Rm}^{\widetilde{g}^{i}}(z,0)}{\widetilde{r}_{Rm}^{\widetilde{g}^{i}}(y_{i},0)}\geq\frac{d_{\widetilde{g}_{0}^{i}}\left(z,\partial B_{\widetilde{g}^{i}}(x_{i},0,1)\right)}{d_{\widetilde{g}_{0}^{i}}\left(y,\partial B_{\widetilde{g}^{i}}(x_{i},0,1)\right)}\geq\frac{1}{2}.
\]
That is, for any $D<\infty$, we have $\widetilde{r}_{Rm}^{\widehat{g}^{i}}(\cdot,0)\geq\frac{1}{2}$
on $B_{\widehat{g}_{0}^{i}}(y_{i},0,D)\cap u_{i}^{-1}(s_{i})$ for
sufficiently large $i\in\mathbb{N}$. By $(3)$ and Theorem 9.29 of
\cite{cheegernotes}, we can find $\delta_{i}''$-splitting maps (with
respect to $\widehat{g}_{0}^{i}$) $v_{i}:=T_{i}\circ(u_{i}-s_{i}):B_{\widehat{g}_{0}^{i}}(z_{i},0,1)\to\mathbb{R}^{n-2}$,
where $T_{i}\in GL(n-2,\mathbb{R})$ and $\delta_{i}''\searrow0$. 

After passing to a subsequence, we can assume $(M,d_{\widehat{g}_{0}^{i}},y_{i})$
converge in the pointed Gromov-Hausdorff sense to some noncollapsed
Ricci limit space $(W,d_{W},w)$. Also, $|B_{\widehat{g}_{0}^{i}}(y,0,r)|_{\widehat{g}_{0}^{i}}\geq A^{-1}r^{n}$
for all $y\in M$ and $r>0$ when $i=i(r)\in\mathbb{N}$ is sufficiently
large, so $\mathcal{H}^{n}(B^{W}(w,r)))\geq A^{-1}r^{n}$ for all
$r\in(0,\infty)$ by Colding's volume convergence theorem. It follows
from the proof of the Transformation theorem (Theorem 1.32 of \cite{cheegernaberdim4})
that in fact $v_{i}:B_{\widehat{g}^{i}}(y_{i},0,r)\to\mathbb{R}^{n-2}$
are $C(r)\delta_{i}''$-splittings for any $r>0$ when $i=i(r)\in\mathbb{N}$
is sufficiently large. This is essentially because we know $T_{i,r}v_{i}:B_{\widehat{g}^{i}}(y_{i},0,r)\to\mathbb{R}^{n-2}$
are $\delta_{i}''$-splittings for some $T_{i,r}\in GL(n-2,\mathbb{R})$,
and we can then compare $T_{i}$ with $T_{i,r}$, using the fact that
they are both lower triangular with positive diagonal entries, to
conclude they are actually close (see the proof of Theorem 1.32 in
\cite{cheegernaberdim4} for details). We can conclude from the ``splitting
maps imply metric almost-splitting'' theorem (Theorem 3.6 of \cite{cheegercoldingwarped})
that $W$ splits off $\mathbb{R}^{n-2}$ isometrically, and $v_{i}$
converge locally uniformly (with respect to the Gromov-Hausdorff convergence)
to a 1-Lipschitz function $v:W\to\mathbb{R}^{n-2}$, such that $v:W=\mathbb{R}^{n-2}\times S\to\mathbb{R}^{n-2}$
is the projection map.

\noindent \textbf{Claim 1: }$S$ is smooth.

We follow Lemma 13.5 of \cite{bam1}. Because $\widetilde{r}_{Rm}^{\widehat{g}^{i}}(y_{i},0)=1$,
we know $w$ is in the (smooth and open) regular set of $W$. Suppose
$z\in W$ is also in the regular set, and choose a sequence $z_{i}\in M_{i}$
converging to $z$ with respect to the Gromov-Hausdorff convergence.
By smooth convergence on the regular set, we have $\sigma:=\liminf_{i\to\infty}\widetilde{r}_{Rm}^{\widehat{g}^{i}}(z_{i},0)>0$.
Let $f_{i}:B^{W}(z,\frac{1}{2}\sigma)\to B_{\widehat{g}^{i}}(y_{i},0,\sigma)$
be diffeomorphisms with $f_{i}^{\ast}\widehat{g}_{0}^{i}\to g^{W}$,
and such that $f_{i}\to\text{id}_{W}$ locally uniformly with respect
to the Gromov-Hausdorff convergence, where $g^{W}$ is the metric
on the regular set of $W$. Because $v_{i}$ are harmonic and converge
locally uniformly to $v$, local elliptic regularity and Arzela-Ascoli's
theorem lets us pass to a subsequence such that $f_{i}^{\ast}v_{i}\to v$
in $C_{loc}^{\infty}(B^{W}(z,\frac{1}{2}))$. The $C^{2}$ convergence
allows us to apply a quantitative version of the implicit function
theorem to find $z_{i}'\in u_{i}^{-1}(0)$ with $d_{\widehat{g}_{0}^{i}}(z_{i}',z_{i})\to0$.
This implies $z_{i}'\to z$, and $\widetilde{r}_{Rm}^{\widehat{g}^{i}}(z_{i}',0)\geq\frac{1}{2}$
then implies $r_{Rm}^{W}(z)\geq\frac{1}{2}$ (using backwards pseudolocality
and Shi's local derivative estimates). Because $S$ is connected,
this implies $r_{Rm}^{W}\ge\frac{1}{2}$ on all of $S$. $\square$

Since $W=\mathbb{R}^{n-2}\times S$ is a metric product and $S$ is
$C^{\infty}$, we can conclude that $W$ is $C^{\infty}$, and (by
Theorem \ref{thm:openandsmooth}) $(M,\widehat{g}_{0}^{i},y_{i})$
converges to $(W,g^{W},w)$ in the $C^{\infty}$ Cheeger-Gromov sense,
where $g^{W}$ is a $C^{\infty}$ Riemannian metric on $W$ whose
length metric coincides with $d_{W}$. We can also conclude that $Rc(g_{W})\geq0$
everywhere, and $g^{W}=g_{\mathbb{R}^{n-2}}+g^{S}$, where $g^{S}$
is a $C^{\infty}$ Riemannian metric on $S$ with nonnegative curvature.
In fact, we can conclude by Theorem \ref{thm:pseudolocality} and
Shi's derivative estimates that $(M,(\widehat{g}_{t}^{i})_{t\in[-\epsilon,\epsilon]},y_{i})$
converges in the $C^{\infty}$ Cheeger-Gromov sense to a complete,
pointed Ricci flow $(\mathbb{R}^{n-2}\times S,g_{\mathbb{R}^{n-2}}+g_{t}^{S},(0^{n-2},w^{\ast}))_{t\in[-\epsilon,\epsilon]}$
with $g_{0}^{S}=g^{S}$, $w=(0^{n-2},w^{\ast})$, for some $\epsilon>0$,
which has uniformly bounded curvature, and satisfies $Rc(g_{t}^{S})\geq0$
everywhere. Set $g_{t}^{\infty}:=g_{\mathbb{R}^{n-2}}+g_{t}^{S}$.
If $(S,g^{S})$ has vanishing curvature somewhere, then $R(g_{t}^{\infty})=0$
somewhere, and because $R(g_{t}^{\infty})\geq0$ everywhere, the strong
maximum principle for $R(g^{\infty})$ and uniqueness for complete
Ricci flow solutions with bounded curvature (Theorem 1.1 of \cite{completenoncompactunique})
implies that $R(g_{t}^{\infty})=0$ everywhere. This is only possible
if $g_{t}^{\infty}$ is flat, contradicting $\widetilde{r}_{Rm}^{g^{\infty}}(w,0)=1$.
We may therefore assume $(S,g^{S})$ has strictly positive curvature
everywhere. In particular, $S$ is diffeomorphic to $\mathbb{R}^{2}$.
The Cohn-Vossen inequality thus gives $\int_{S}R(g_{0}^{S})dg_{0}^{S}\leq4\pi$,
so that for any $x\in M$ and $r>0$, we have
\begin{align*}
\frac{1}{|B_{g^{S}}(x,0,r)|_{g_{0}^{S}}}\int_{B_{g^{S}}(x,0,r)}R(g_{0}^{S})dg_{0}^{S}\leq & \frac{\omega_{n-2}r^{n-2}}{|B_{g^{S}}(x,0,r)\times B^{\mathbb{R}^{n-2}}(0^{n-2},r)|_{g_{0}^{\infty}}}4\pi\leq\frac{4\pi A\omega_{n-2}}{r^{2}}.
\end{align*}
By a criterion of Shi (Theorem 9.2 in \cite{hamchow}), the flow $(g_{t}^{S})_{t\in(-\epsilon,\epsilon)}$
can be extended to a complete, immortal Ricci flow $(g_{t}^{S})_{t\in(-\epsilon,\infty)}$
with bounded curvature on compact time intervals. 

\noindent \textbf{Claim 2: }After passing to a further subsequence,
we can assume $(M_{i},(\widehat{g}_{t}^{i})_{t\in(-\epsilon,(r_{i}\rho_{i})^{-2}(T-t_{i}))},y_{i})$
converge in the $C^{\infty}$ pointed Cheeger-Gromov-Hamilton sense
to $(W,(g_{t}^{\infty})_{t\in(-\epsilon,\infty)},w)$.

By $\lim_{i\to\infty}r_{i}(T-t_{i})^{-\frac{1}{2}}=0$, for any $\tau<\infty$,
$\widehat{g}_{t}^{i}$ is defined for $t\in(-\epsilon,\tau]$ whenever
$i=i(\tau)\in\mathbb{N}$ is sufficiently large. Let $t^{\ast}\in[0,\infty]$
be supremal such that there exists $C<\infty$ such that for any $D<\infty$
and $\delta>0$, we have
\[
\sup_{B_{\widehat{g}^{i}}(y_{i},t,D)\times[0,t^{\ast}-\delta]}|Rm|_{\widehat{g}^{i}}\leq C
\]
for all $i=i(D,\delta,t^{\ast})\in\mathbb{N}$ sufficiently large.
Suppose by way of contradiction that $t^{\ast}<\infty$. Then we can
pass to a subsequence to get pointed Cheeger-Gromov-Hamilton convergence
\[
(M_{i},(\widehat{g}_{t}^{i})_{t\in(-\epsilon,t^{\ast})},y_{i})\to(\overline{W},(\overline{g}_{t}^{\infty})_{t\in(-\epsilon,t^{\ast})},\overline{w})
\]
to some complete Ricci flow with bounded curvature on compact time
intervals. By the uniqueness \cite{completenoncompactunique} of complete
Ricci flows with bounded curvature, and because $(M_{i},\widehat{g}_{0}^{i},y_{i})\to(W,g^{\infty},w)$
in the pointed Cheeger-Gromov sense, we have $(\overline{W},(\overline{g}_{t}^{\infty})_{t\in(-\epsilon,t^{\ast})},\overline{w})=(W,(g_{t}^{\infty})_{t\in(-\epsilon,t^{\ast})},w)$. 

Because $(W,(g_{t}^{\infty})_{t\in(-\epsilon,\infty)})$ is immortal,
with bounded curvature on compact time intervals, there exists $C=C(t^{\ast})<\infty$
such that for any $D'<\infty$ and $\delta>0$,
\[
\sup_{t\in[0,t^{\ast}-\delta]}\sup_{B_{\widehat{g}^{i}}(y_{i},t,D')}|Rm|\leq C
\]
for $i=i(D',\delta)\in\mathbb{N}$ sufficiently large. For any $D<\infty$,
we can thus take $\delta=\delta(t^{\ast})>0$ small and $D'=D'(D,t^{\ast})<\infty$
large and apply Theorem \ref{thm:pseudolocality} to obtain
\[
\sup_{t\in[0,t^{\ast}+\delta)}\sup_{B_{\widehat{g}^{i}}(y_{i},t,D)}|Rm|\leq\overline{C}(t^{\ast}),
\]
for $i=i(D,t^{\ast})\in\mathbb{N}$ sufficiently large. This contradicts
the definition of $t^{\ast}$, so we must have $t^{\ast}=\infty$,
and the claim follows from the Cheeger-Gromov-Hamilton compactness
theorem and the aforementioned uniqueness result \cite{completenoncompactunique}.
$\square$

If $(S,(g_{t}^{S})_{t\in(-\epsilon,\infty)})$ is a Type-IIb immortal
solution:
\[
\lim_{t\to\infty}\sup_{x\in S}t|Rm|_{g^{S}}(x,t)=\infty,
\]
then a Type-IIb rescaling (as in Proposition 8.20 of \cite{hamchow})
procedure produces a singularity model which is a the Cigar soliton
(Theorem 9.4 of \cite{hamchow}). However, because the limit of a
limit is a limit (see Lemma 8.26 of \cite{hamchow}), we obtain the
Cigar soliton (times flat $\mathbb{R}^{n-2}$) as a smooth Cheeger-Gromov-Hamilton
limit of some Ricci flows $(M_{i},(\overline{g}_{t}^{i})_{t\in(a_{i},b_{i})},z_{i})$
satisfying $Rc(\overline{g}_{t}^{i})\geq-A\overline{g}_{t}^{i}$ and
$|B_{\overline{g}^{i}}(z_{i},0,r)|_{\overline{g}_{0}^{i}}\geq A^{-1}r^{n}$
for all $r\in(0,\infty)$, when $i=i(r)\in\mathbb{N}$ is sufficiently
large. We thus obtain that the Cigar soliton (times flat $\mathbb{R}^{n-2}$)
has maximal volume growth, which is a contradiction.

We may therefore assume the immortal solution is Type-III, so that
\[
\Lambda:=\limsup_{t\to\infty}\sup_{x\in S}tR_{g^{S}}(x,t)<\infty.
\]
Hamilton's Harnack inequality implies $\partial_{t}((t+\epsilon)R_{g^{S}}(w,t))\geq0$,
so $tR_{g^{S}}(w,t)\geq\frac{1}{2}\epsilon R_{g^{S}}(w,0)$ for all
$t\in[1,\infty)$, hence $\Lambda>0$. Choose $(w_{i},\tau_{i})\in S\times(0,\infty)$
such that $\tau_{i}\nearrow\infty$ and $\lim_{i\to\infty}\tau_{i}R_{g^{S}}(w_{i},\tau_{i})=\Lambda$.
Define $\overline{Q}_{i}:=R_{g^{S}}(w_{i},\tau_{i})\searrow0$ and
$\overline{g}_{t}^{i}:=\overline{Q}_{i}g_{\overline{Q}_{i}^{-1}t}^{S}$
for $t\in[0,\infty)$, so that 
\[
R_{\overline{g}^{i}}\left(w_{i},\tau_{i}\overline{Q}_{i}\right)=\overline{Q}_{i}^{-1}R_{g^{S}}(w_{i},\tau_{i})=1,
\]
and for $\sigma\in(0,1)$, we have the following for all $(x,t)\in S\times[\sigma,\sigma^{-1}]$
when $i=i(\sigma)\in\mathbb{N}$ is sufficiently large: 
\[
tR_{\overline{g}^{i}}(x,t)=t\overline{Q}_{i}^{-1}R_{g^{S}}(x,\overline{Q}_{i}t)\leq\Lambda+\sigma.
\]
After passing to a subsequence, $(S,(\tau_{i}^{-1}g_{\tau_{i}t}^{S})_{t\in(0,\infty)},w)$
thus converges in the $C^{\infty}$ pointed Cheeger-Gromov sense as
$i\to\infty$ to a complete noncompact immortal Ricci flow $(E,(g_{t}^{E})_{t\in(0,\infty)},e_{\ast})$
with positive curvature satisfying $tR_{g^{E}}(x,t)\leq\Lambda$ for
all $t>0$, with equality at $t=\Lambda$, since $\tau_{i}\overline{Q}_{i}\to\Lambda$.
By Theorem 1.4 of \cite{huaidongtypeiii}, $(E,(g_{t}^{E})_{t\in(0,\infty)})$
is a nonflat expanding gradient Ricci soliton $(E,(g_{t}^{E})_{t\in(0,\infty)},e_{\ast})$
by Theorem 1.4 of \cite{huaidongtypeiii}. Reasoning as in the Type-IIb
case, $(E,g_{1}^{E})$ has maximal volume growth. By the classification
of 2-dimensional gradient Ricci solitons in \cite{2dsolitons}, $g^{E}$
must be an asymptotically conical expanding Ricci soliton asymptotic
at infinity to $C(\mathbb{S}_{\beta}^{1})$ for some $\beta\in(0,2\pi)$,
constructed in Appendix A of \cite{2dexpanderconstruction} (see also
section 4.5 of \cite{hamchow}). Here, $e_{\ast}\in E$ corresponds
to the point fixed by the isometric $\mathbb{S}^{1}$ action, which
is also the global maximum of the soliton potential function.

Again using the diagonal argument from Lemma 8.26 in \cite{hamchow},
we can find $(\check{y}_{i},\check{t}_{i})\in M_{i}\times[\frac{T_{i}}{2},T_{i})$
and some $\check{r}_{i}\searrow0$ such that $-\check{r}_{i}^{-2}\check{t}_{i}\to-\infty$,
$(T-\check{t}_{i})\check{r}_{i}^{-2}\to\infty$ and the rescaled flows
$\check{g}_{t}^{i}:=\check{r}_{i}^{-2}g_{\check{r}_{i}^{2}t+\check{t}_{i}}$
are such that $(M_{i},(\check{g}_{t}^{i})_{t\in(0,(T-\check{t}_{i})\check{r}_{i}^{-2})},\check{y}_{i})$
converge in the pointed Cheeger-Gromov sense to the expanding soliton
$(\mathbb{R}^{n-2}\times E,(g_{\mathbb{R}^{n-2}}+g_{t}^{E})_{t\in(0,\infty)},(0^{n-2},e_{\ast}))$.
Moreover, if we let $(\check{\nu}_{t}^{i})_{t\in(-\check{r}_{i}^{-2}\check{t}_{i},1)}$
be the conjugate heat kernels of the rescaled flows based at $(\check{y}_{i},1)$,
then we can pass to a subsequence to obtain a metric flow pair $(\check{\mathcal{X}},(\check{\mu}_{t}^{\infty})_{t<0})$
such that
\[
(M_{i},(\check{g}_{t}^{i})_{t\in[-\check{r}_{i}^{-2}\check{t}_{i},1]},(\check{\nu}_{t}^{i})_{t\in[-\check{r}_{i}^{-2}\check{t}_{i},1[})\xrightarrow[i\to\infty]{\mathbb{F},\mathfrak{C}}(\check{\mathcal{X}},(\check{\mu}_{t}^{\infty})_{t\in(-\infty,1[})
\]
on compact time intervals, where $\mathfrak{C}$ is a fixed correspondence,
and
\[
\mathcal{\check{\mathcal{X}}}=\left((\check{\mathcal{X}}_{t})_{t\in(-\infty,1]},\check{\mathfrak{t}},(\check{d}_{t})_{t\in(-\infty,1]},(\check{\nu}_{x;s})_{x\in\check{\mathcal{X}},s\in(-\infty,\check{\mathfrak{t}}(x))}\right)
\]
is an $H_{n}$-concentrated future-continuous metric flow of full
support. In fact, this follows from Section 2 of \cite{bamlergen3},
where the required lower bound on Nash entropy follows as in Lemma
\ref{lem:basiclemma}, using Theorem 8.1 of \cite{bamlergen1} to
estimate
\[
A^{-1}<|B_{\check{g}^{i}}(\check{y}_{i},0,1)|_{\check{g}_{0}^{i}}\leq C(A,n)\exp\left(\mathcal{N}_{\check{y}_{i},0}^{\check{g}^{i}}(1)\right).
\]
Let $(\check{\mathcal{R}},\check{\mathfrak{t}},\partial_{\check{\mathfrak{t}}},\check{g}^{\infty})$
denote the corresponding Ricci flow spacetime structure on the regular
set of $\check{\mathcal{X}}$. Let $(\check{U}_{i})$ be an increasing
compact exhaustion of $\check{\mathcal{R}}$, and $\check{\psi}_{i}:\check{U}_{i}\to M_{i}$
be diffeomorphisms such that 
\[
||\check{\psi}_{i}^{\ast}\check{K}^{i}-\check{K}^{\infty}||_{C^{i}(\check{U}_{i})}+||\check{\psi}_{i}^{\ast}\check{g}^{i}-\check{g}^{\infty}||_{C^{i}(\check{U}_{i})}+||(\check{\psi}_{i}^{-1})_{\ast}\partial_{t}-\partial_{\check{\mathfrak{t}}}||_{C^{i}(\check{U}_{i})}\leq\alpha_{i}
\]
for some sequence $\alpha_{i}\searrow0$, where $d\check{\mu}_{t}^{\infty}=\check{K}^{\infty}(\cdot,t)d\check{g}_{t}^{\infty}$
on $\check{\mathcal{R}}$ and $d\check{\nu}_{t}^{i}=\check{K}^{i}(\cdot,t)d\check{g}_{t}^{i}$. 

\noindent \textbf{Claim 3:} $\check{\mathcal{X}}_{(0,1)}$ is isometric
as a metric flow to the smooth Ricci flow $(\mathbb{R}^{n-2}\times E,(g_{\mathbb{R}^{n-2}}+g_{t}^{E})_{t\in(0,1)})$. 

From the smooth Cheeger-Gromov-Hamilton convergence
\[
(M_{i},(\check{g}_{t}^{i})_{t\in(0,2)},\check{y}_{i})\to(\mathbb{R}^{n-2}\times E,(g_{\mathbb{R}^{n-2}}+g_{t}^{E})_{t\in(0,2)},(0^{n-2},e_{\ast})),
\]
we know that
\[
\limsup_{i\to\infty}\sup_{B_{\check{g}^{i}}(\check{y}_{i},0,D)\times(\kappa,1]}|Rm|_{\check{g}^{i}}\leq C(\kappa)<\infty
\]
for any fixed $D<\infty$ and $\kappa\in(0,1)$. Thus, for any fixed
$\kappa\in(0,1)$, we can apply Proposition 9.5 of \cite{bam1} to
get
\[
\limsup_{i\to\infty}d_{W_{1}}^{\check{g}_{-\kappa}^{i}}(\check{\nu}_{1-\kappa}^{i},\delta_{\check{y}_{i}})\leq C
\]
for some $C=C(\kappa)<\infty$. Fix $\overline{y}_{\infty}\in\mathcal{\check{R}}_{-\kappa}$,
and set $\overline{y}_{i}:=\check{\psi}_{i}(\overline{y}_{\infty})$,
so that $\overline{y}_{i}\xrightarrow[i\to\infty]{\mathfrak{C}}\overline{y}_{\infty}$
by Theorem 9.31(c) of \cite{bamlergen2}. Because $\check{\nu}_{1-\kappa}^{\infty}(B(\overline{y}_{\infty},1))>0$
and $\lim_{i\to\infty}||\check{\psi}_{i}^{\ast}\check{K}^{i}-\check{K}^{\infty}||_{C^{0}(\check{U}_{i})}=0$,
we have
\[
\liminf_{i\to\infty}\check{\nu}^{i}(B_{\check{g}^{i}}(\overline{y}_{i},1-\kappa,2))\geq\check{\nu}_{1-\kappa}^{\infty}(B(\overline{y}_{\infty},1))>0.
\]
Thus, for sufficiently large $i\in\mathbb{N}$, 
\[
\frac{1}{2}\check{\nu}_{1-\kappa}^{\infty}(B(\overline{y}_{\infty},1))\left(d_{\check{g}_{1-\kappa}^{i}}(\check{y}_{i},\overline{y}_{i})-1\right)<\int_{M_{i}}d_{\check{g}_{1-\kappa}^{i}}(\check{y}_{i},y)d\check{\nu}_{1-\kappa}^{i}(y)\leq C,
\]
and in particular, 
\[
\limsup_{i\to\infty}d_{\check{g}_{1-\kappa}^{i}}(\check{y}_{i},\overline{y}_{i})<\infty.
\]
We can now apply Theorem 9.58 of \cite{bamlergen3} to conclude that
the parabolic cylinders $P(\overline{y}_{\infty};D,-1+2\kappa,0)\subseteq\check{\mathcal{R}}$
are unscathed for any $\kappa\in(0,\frac{1}{4})$ and $D<\infty$.
Because $\kappa,D$ were arbitrary, it follows that $\check{\mathcal{R}}_{(-1,0)}=\check{\mathcal{X}}_{(-1,0)}$,
and we can identify $\check{\mathfrak{\mathcal{R}}}_{(-1,0)}$ with
the canonical spacetime of a smooth Ricci flow. Moreover, we have
Cheeger-Gromov-Hamilton convergence $(M,(\check{g}_{t}^{i})_{t\in(0,1)},\overline{y}_{i})\to(\check{\mathcal{R}}_{t},(\check{g}_{t}^{\infty})_{t\in(0,1)},\overline{y}_{\infty})$,
hence the Ricci flow spacetime is isometric to $(\mathbb{R}^{n-2}\times E,(g_{\mathbb{R}^{n-2}}+g_{t}^{E})_{t\in(0,1)})$.
$\square$

Suppose $(\mathcal{U}_{i})_{i\in\mathbb{N}}$ is an exhaustion of
$\mathbb{R}^{n-2}\times E$ and $\zeta_{i}:\mathcal{U}_{i}\to M_{i}$
are embeddings such that $\text{\ensuremath{\zeta_{i}(\overline{y}_{\infty})=\overline{y}_{i}}}$
and $\zeta_{i}^{\ast}\check{g}^{i}\to\check{g}^{\infty}$ as $i\to\infty$
in $C_{loc}^{\infty}((\mathbb{R}^{n-2}\times E)\times(0,1))$. Direct
computation (see Section 5 of \cite{hamchow}) shows that 
\[
\lim_{d_{g_{1}^{E}}(e_{\ast},y)\to\infty}|Rm|_{g^{E}}(y,1)d_{g_{1}^{E}}^{2}(e_{\ast},y)=0.
\]
Now fix $r_{0}\in(0,\infty)$ such that 
\[
|Rm|_{g^{E}}(y,1)d_{g_{1}^{E}}^{2}(e_{\ast},y)\leq\frac{1}{4}
\]
for all $y\in E$ with $d_{g_{1}^{E}}(y,e_{\ast})\geq r_{0}$. Suppose
$y\in E$ satisfies $r_{y}:=d_{g_{1}^{E}}(y,e_{\ast})\geq r_{0}$,
so that 
\[
\sup_{B_{g^{E}}(y,1,\frac{1}{2}r_{y})}|Rm|_{\check{g}^{\infty}}(\cdot,1)\leq\frac{1}{r_{y}^{2}},
\]
hence 
\[
\sup_{B_{\check{g}^{i}}(\zeta_{i}(0^{n-2},y),1,\frac{1}{4}r_{y})}|Rm|_{\check{g}^{i}}(\cdot,1)\leq\frac{2}{r_{y}^{2}}
\]
for sufficiently large $i\in\mathbb{N}$. By Theorem \ref{thm:pseudolocality}
and Shi's estimates on the backwards cylinder $B_{\check{g}^{i}}(\zeta_{i}(0^{n-2},y),1,\frac{\epsilon_{P}}{4}r_{y})\times[1-\frac{1}{16}(\epsilon_{P}r_{y})^{2},1]$
when $i\in\mathbb{N}$ is large, we get
\[
|\nabla^{k}Rm|_{\check{g}^{i}}(\zeta_{i}(0^{n-2},y),1)\leq C(k,A)r_{y}^{-2-k},
\]
hence taking $i\to\infty$ gives $|\nabla^{k}Rm|_{\check{g}^{\infty}}(0^{n-2},y)\leq C(k,A)r_{y}^{-2-k}$,
hence $|\nabla^{k}Rm|_{g_{E}}(e_{\ast},y)\leq C(k,A)r_{y}^{-2-k}$.
We have thus shown 
\[
\sup_{y\in E}|\nabla^{k}Rm|_{g_{E}}(y)d_{g_{1}^{E}}^{2+k}(e_{\ast},y)\leq C(k,A)<\infty.
\]
This is the condition needed to apply Theorem 4.3.1 of \cite{conethesis}
to guarantee that the expanding soliton $(E,(g_{t}^{E})_{t\in(-1,\infty)})$
converges smoothly to $C(\mathbb{S}_{\beta}^{1})$ away from the vertex:
\[
\lim_{t\searrow0}\sup_{K}|(\nabla^{g_{C(\mathbb{S}_{\beta}^{1})}})^{k}(g_{C(\mathbb{S}_{\beta}^{1})}-g_{t}^{E})|_{g_{C(\mathbb{S}_{\beta}^{1})}}=0
\]
for any compact subset $K\subseteq C(\mathbb{S}_{\beta''}^{1})\setminus\{o_{\ast}\}\cong E\setminus\{e_{\ast}\}$
and integer $k\in\mathbb{N}$. In particular, we have 
\[
\lim_{t\searrow0}\sup_{K}|Rm|_{g^{E}}(\cdot,t)=0.
\]

\noindent Now let $y'\in\left((\mathbb{R}^{n-2}\times E)\setminus(\mathbb{R}^{n-2}\times\{e_{\ast}\})\right)\times\{1\}\subseteq\check{\mathcal{X}}_{1}$
be arbitrary. Because $\lim_{t\searrow0}\sup|Rm|_{g^{E}}(\cdot,t)\to0$,
Theorem \ref{thm:pseudolocality} gives uniform curvature bounds on
$B_{\widehat{g}^{i}}(\zeta_{i}(y'),t',2r')\times[t'-2(r')^{2},t'+2(r')^{2}]$
for all $t'\in(0,1]$, where $r'=r'(y')>0$ is independent of $t'$.
We can again apply Theorem 9.58, now concluding that $P(y'(t);r',-(r')^{2},0)\subseteq\check{\mathcal{R}}_{(t-(r')^{2},\delta)}$
is unscathed for all $t\in(0,1]$. Choosing $t'\in(0,\frac{1}{4}(r')^{2})$,
we know that $P(y';r',-(r')^{2},0)\cap\check{\mathcal{X}}_{(0,t)}$
corresponds to the flow $(g_{\mathbb{R}^{n-2}}+g_{t}^{E})_{t\in(0,t')}$
on some open subset $U\subseteq(\mathbb{R}^{n-2}\times E)\setminus(\mathbb{R}^{n-2}\times\{e_{\ast}\})$,
so that 
\[
(P(y';r',-(r')^{2},0)\cap\check{\mathcal{X}}_{0},\check{g}_{0}^{\infty})
\]
is isometric to an open subset of the flat Riemannian cone $\mathcal{C}:=\mathbb{R}^{n-2}\times(C(\mathbb{S}_{\beta}^{1})\setminus\{o_{\ast}\})$
and in particular, $Rm(\check{g}_{0}^{\infty})=0$ on $P(y';r',-(r')^{2},0)\cap\check{\mathcal{X}}_{0}$.
However, we know that $R_{\check{g}^{\infty}}\geq0$ everywhere on
$\check{\mathcal{R}}$, so the strong maximum principle for $R_{\check{g}^{\infty}}$
(Corollary 12.43 of \cite{chowbook2}) gives $Rc(\check{g}^{\infty})=0$
on $\check{\mathcal{R}}_{<0}$. Because $y'$ was arbitrary, we see
that the flow of $\partial_{\check{\mathfrak{t}}}$ produces an open
embedding
\[
\iota:\mathcal{C}\cong(\mathbb{R}^{n-2}\times E)\setminus(\mathbb{R}^{n-2}\times\{e_{\ast}\})\times\{1\}\hookrightarrow\check{\mathcal{R}}_{0},
\]
and that the image of $\iota$ is Riemannian isometric to $\mathcal{C}$
equipped with the flat cone metric $g_{\mathcal{C}}$. 

\noindent Suppose by way of contradiction that $z\in\check{\mathcal{R}}_{0}\setminus\iota(\mathcal{C})$,
and let $\epsilon>0$ be such that the integral curve $\gamma:(-2\epsilon,2\epsilon)\to\check{\mathcal{R}}$
of $\partial_{\check{\mathfrak{t}}}$ with $\gamma(0)=z$ is well-defined.
Then $\gamma(\epsilon)$ corresponds to a point in $\mathbb{R}^{n-2}\times E$
not in the domain of $\iota$; that is, $\gamma(\epsilon)\in(\mathbb{R}^{n-2}\times\{e_{\ast}\})\times\{\epsilon\}$.
However, $\lim_{t\searrow-1}|Rm|_{\check{g}^{\infty}}(\gamma(t))=\infty$,
contradicting $r_{Rm}^{\check{\mathcal{X}}}(z)>0$. Thus $\iota$
is surjective. Because $\check{\mathcal{X}}_{-1}$ is the metric completion
of $\check{\mathcal{R}}_{-1}$ equipped with its length metric, it
follows that $\iota$ extends to a metric isometry$\check{\mathcal{X}}_{0}\to\mathbb{R}^{n-2}\times C(\mathbb{S}_{\beta}^{1})$. 

\noindent Because $Rc(\check{g}^{\infty})=0$ on $\check{\mathcal{R}}_{\leq0}$,
we know that $\check{\mathcal{X}}_{<0}$ is a static metric flow by
Theorem 1.17 of \cite{bamlergen3}, corresponding to some Ricci flat
singular space $(\check{X}',\check{d}',\check{\mathcal{R}}',\check{g}')$.

\noindent \textbf{Claim 4: $(\check{X}',\check{d}')$ }is isometric
to $\mathbb{R}^{n-2}\times C(\mathbb{S}_{\beta}^{1})$. 

For any $y'\in\check{\mathcal{R}}_{0}$, there exists $\epsilon>0$
such that the integral curve $\gamma:(-\epsilon,\epsilon)\to\check{\mathcal{R}}$
of $\partial_{\check{\mathfrak{t}}}$ with $\gamma(0)=y'$ is well-defined.
Because $\check{\mathcal{X}}_{<0}$ is static and $\check{\mathcal{X}}_{(0,1)}=\check{\mathcal{R}}_{(0,1)}$,
we conclude that $\gamma$ extends to an integral curve of $\partial_{\check{\mathfrak{t}}}$
defined on $(-\infty,1)$. Flowing along $\partial_{\check{\mathfrak{t}}}$
therefore gives a smooth embedding $\iota':\check{\mathcal{R}}_{0}\hookrightarrow\check{\mathcal{R}}_{1}=\check{\mathcal{R}}'$,
or equivalently $\mathcal{C}\hookrightarrow\check{X}'$. Moreover,
because $Rc(\check{g}^{\infty})=0$ on $\check{\mathcal{R}}_{(-\infty,0]}$,
we have $\check{d}_{-1}(\iota'(z_{1}),\iota'(z_{2}))\leq\check{d}_{0}(z_{1},z_{2})$
for all $z_{1},z_{2}\in\check{\mathcal{R}}_{0}$. 

Suppose by way of contradiction that there exist $\epsilon>0$, $z',z''\in\check{\mathcal{R}}_{0}$,
and a curve $\gamma:[0,1]\to\check{\mathcal{R}}_{-1}$ from $\iota'(z')$
to $\iota'(z'')$ such that $\text{length}_{\check{g}_{-1}^{\infty}}(\gamma)<\check{d}_{0}(z',z'')-\epsilon$.
For sufficiently large $i\in\mathbb{N}$, we know that $\gamma([0,1])\subseteq\check{U}_{i}$,
so we can define $\gamma_{i}:=\check{\psi}_{i}\circ\gamma$. Then,
for large $i\in\mathbb{N}$, $\gamma_{i}$ is a curve from $z_{i}':=(\check{\psi}_{i}\circ\iota')(z')$
to $z_{i}'':=(\check{\psi}_{i}\circ\iota')(z'')$ in $M$ with
\[
\text{length}_{\check{g}_{-1}^{i}}(\gamma_{i})<\check{d}_{0}(z',z'')-\frac{\epsilon}{2}.
\]
Because $Rc(\check{g}_{t}^{i})\geq-\check{r}_{i}^{2}A$, we obtain 

\[
\text{length}_{\check{g}_{\kappa}^{i}}(\gamma_{i})<\check{d}_{0}(z',z'')-\frac{\epsilon}{4}
\]
for all $\kappa\in(0,1)$ when $i\in\mathbb{N}$ is sufficiently large.
Now let $\eta_{z'},\eta_{z''}:[-1,\frac{1}{2}]\to\check{\mathcal{R}}$
be the integral curves of $\partial_{\check{\mathfrak{t}}}$ with
$\eta_{z'}(0)=z'$ and $\eta_{z''}(0)=z''$, so that $\eta_{z'}(-1)=\iota'(z')$,
$\eta_{z''}(-1)=\iota'(z'')$, and $\eta_{z'}([-1,\frac{1}{2}]),\eta_{z''}([-1,\frac{1}{2}])\subseteq V\subset\subset\check{\mathcal{R}}$,
where $V$ is some open product domain $V'\times[-1,\frac{1}{2}]$.
Because 
\[
||(\check{\psi}_{i}^{-1})_{\ast}\partial_{t}-\partial_{\check{\mathfrak{t}}}||_{C^{0}(U_{i})}<\alpha_{i},
\]
a standard result about continuous dependence of ODE solutions on
parameters (see Chapter 5 of \cite{hartmanode}) tells us that the
integral curves $\eta_{z'}^{i},\eta_{z''}^{i}:[-1,\frac{1}{2}]\to\check{\mathcal{R}}$
of $(\check{\psi}_{i}^{-1})_{\ast}\partial_{t}$ satisfying $\eta_{z'}^{i}(-1)=\iota'(z')$,
$\eta_{z''}^{i}(-1)=\iota'(z'')$ are well-defined (do not escape
$V'\times[-1,\frac{1}{2}]$), and satisfy $\lim_{i\to\infty}\eta_{z'}^{i}(s)=\eta_{z'}(s)$,
$\lim_{i\to\infty}\eta_{z''}^{i}(s)=\eta_{z''}(s)$ for any $s\in[-1,\frac{1}{2}]$.
In particular, we have $\gamma_{i}(0)\times[-1,\frac{1}{2}]\in\check{\psi}_{i}(\check{U}_{i})$
for sufficiently large $i\in\mathbb{N}$. Using the locally uniform
curvature estimates at times $t\in[\frac{\kappa}{2},\kappa]$, Theorem
9.58 of \cite{bamlergen2} guarantees that $\gamma_{i}([0,1])\times\{\kappa\}\subseteq\check{\psi}_{i}(\check{U}_{i})$
for sufficienly large $i\in\mathbb{N}$. Because $\check{g}^{i}\to\check{g}^{\infty}$
in $C_{loc}^{0}(\check{\mathcal{R}})$, we obtain
\[
\check{d}_{\kappa}(\check{\psi}_{i}^{-1}(z_{i}'),\check{\psi}_{i}^{-1}(z_{i}''))\leq\text{length}_{\check{g}_{\kappa}^{\infty}}(\psi_{i,\kappa}^{-1}\circ\gamma_{i})<\check{d}_{0}(z',z'')-\frac{\epsilon}{8}.
\]
Moreover, we have 
\[
\check{d}_{\kappa}(\check{\psi}_{i}^{-1}(z_{i}'),z'(\kappa))=\check{d}_{\kappa}(\eta_{z'}^{i}(\kappa),\eta_{z'}(\kappa))\to0
\]
as $i\to\infty$, and similarly for $z''$. Combining estimates gives
\[
\check{d}_{\kappa}(z'(\kappa),z''(\kappa))\leq\check{d}_{0}(z',z'')-\frac{\epsilon}{16}
\]
for all $\kappa\in(0,\frac{1}{2}]$. For any $\kappa\in(0,\frac{1}{2}]$,
we can choose a curve $\zeta:[0,1]\to\check{\mathcal{R}}_{-1}$ from
$z'(\kappa)$ to $z''(\kappa)$ with 
\[
\text{length}_{\check{g}_{\kappa}^{\infty}}(\zeta)<\check{d}_{\kappa}(z'(\kappa),z''(\kappa))+\frac{\epsilon}{32},
\]
and $\zeta([0,1])\cap(\mathbb{R}^{n-2}\times B(e_{\ast},c\epsilon))=\emptyset$
(this is possible because $\mathbb{R}^{n-2}\times\{e_{\ast}\}$ has
codimension 2 in $\mathbb{R}^{n-2}\times E$). However, we know that
\[
\sup_{\mathbb{R}^{n-2}\times(E\setminus B(e_{\ast},c\epsilon))}|Rc|(\cdot,t)\leq1
\]
for $t\in[0,\tau(\epsilon)]$. Defining $\zeta_{t}:[0,1]\to\mathcal{R}_{t}$
by $\zeta_{t}(s):=(\zeta(s))(t)$, we have
\[
\frac{d}{dt}\left(e^{t}\text{length}_{\check{g}_{t}^{\infty}}(\zeta_{t})\right)=e^{t}\left(1-\int_{0}^{1}\frac{Rc(\dot{\zeta}_{t}(s),\dot{\zeta}_{t}(s))}{|\dot{\zeta}_{t}(s)|}ds\right)\text{length}_{\check{g}_{t}^{\infty}}(\zeta_{t})\geq0
\]
for any $t\in[0,\tau(\epsilon)]$, so that 
\[
\limsup_{\kappa\searrow0}\check{d}_{\kappa}(z'(\kappa),z''(\kappa))\geq\limsup_{\kappa\searrow0}e^{-\kappa}\text{length}_{\check{g}_{\kappa}^{\infty}}(\zeta_{0})-\frac{\epsilon}{32}\geq\check{d}_{0}(z',z'')-\frac{\epsilon}{32}.
\]
Combining estimates, we therefore arrive at
\begin{align*}
\check{d}_{0}(z',z'')-\frac{\epsilon}{16}\geq & \limsup_{\kappa\searrow0}\check{d}_{\kappa}(z'(\kappa),z''(\kappa))\geq\check{d}_{0}(z',z'')-\frac{\epsilon}{32},
\end{align*}
a contradiction.

Therefore, $\iota':(\check{\mathcal{R}}_{0},\check{d}_{0})\hookrightarrow(\check{\mathcal{R}}_{-1},\check{d}_{-1})$
is an isometric embedding, which extends to an isometric embedding
$(\check{\mathcal{X}}_{0},\check{d}_{0})\hookrightarrow(\check{\mathcal{X}}_{-1},\check{d}_{-1})$.
Suppose by way of contradiction that $\iota'(\check{\mathcal{R}}_{0})\neq\check{\mathcal{R}}_{-1}$.
Because $\check{\mathcal{R}}_{-1}$ is connected, we can find a curve
$\gamma:[0,1]\to\check{\mathcal{R}}_{-1}$ from a point $v_{1}\in\iota'(\check{\mathcal{R}}_{0})$
to some $v_{2}\in\check{\mathcal{R}}_{-1}\setminus\iota'(\check{\mathcal{R}}_{0})$,
and such that $\widetilde{r}_{Rm}^{\check{\mathcal{X}}_{t}}\left((\gamma(u))(t)\right)\geq\sigma$
for all $u\in[0,1]$ and $t\in[-1,0)$, where $\sigma>0$ is an appropriately
chosen constant. By truncating $\gamma$, we can assume $\gamma([0,1))\subseteq\iota'(\check{\mathcal{R}}_{0})$
and $\gamma(1)\in\iota'(\check{\mathcal{X}}_{0}\setminus\check{\mathcal{R}}_{0})$.
For any fixed $t\in[-1,0)$, we have $\left(\gamma([0,1])\right)(t)\subseteq\check{U_{i}}$
for $i=i(t)\in\mathbb{N}$ sufficiently large, hence $\widetilde{r}_{Rm}^{\check{g}^{i}}(\psi_{i}(\gamma(u)(t)),t)\geq c(A,\sigma)$
for all $u\in[0,1]$ for large $i=i(t)\in\mathbb{N}$. By choosing
$t\in[-1,0)$ sufficiently close to $0$, we can apply Theorem \ref{thm:pseudolocality}
to conclude $r_{Rm}^{\check{g}^{i}}(\psi_{i}(\gamma(u)(0)),0)\geq c(A,\sigma)$
for all $u\in[0,1)$ when $i=i(u)\in\mathbb{N}$ is sufficiently large,
since by hypothesis $\gamma(u)(0)\in\check{\mathcal{R}}_{0}$ for
all $u\in[0,1)$. For $u\in[0,1)$ sufficiently close to $1$, Theorem
9.58 of \cite{bamlergen2} gives
\[
r_{Rm}^{\check{g}^{\infty}}((\gamma(u))(0),0)\geq2\check{d}_{0}((\gamma(u))(0),\mathbb{R}^{n-2}\times\{o_{\ast}\})
\]
since $(\iota')^{-1}(\gamma(1))\notin\check{\mathcal{R}}_{0}$, a
contradiction. Thus $\iota'(\check{\mathcal{R}}_{0})=\check{\mathcal{R}}_{-1}$,
but because $\check{\mathcal{X}}_{-1}$ is the metric completion of
$\check{\mathcal{R}}_{-1}$, it must be isometric to $\mathbb{R}^{n-2}\times C(\mathbb{S}_{\beta}^{1})$.
$\square$

\noindent However, Theorem 1.17 of \cite{bamlergen3} then tells us
that the singular set $\mathbb{R}^{n-2}\times\{o_{\ast}\}$ of $\mathbb{R}^{n-2}\times C(\mathbb{S}_{\beta''}^{1})$
has Minkowski dimension at most $n-4$, which is false.
\end{proof}

\section{Codimension Three $\epsilon$-Regularity}

We can use codimension two $\epsilon$-regularity to estimate the
size of the large curvature region at curvature scales $\widetilde{r}_{Rm}\apprle\sqrt{T-t}$. 
\begin{prop}
\label{prop:bigbigpoints} For any $A<\infty$ and $\underline{T}>0$,
there exist $r_{0}=r_{0}(A,\underline{T})>0$ and $E=E(A,\underline{T})<\infty$
such that the following holds. Let $(M^{n},(g_{t})_{t\in[0,T)})$
be a closed Ricci flow satisfying $Rc(g_{t})\geq-Ag_{t}$ and $|B(x,t,r)|_{g_{t}}\geq A^{-1}r^{n}$
for all $(x,t)\in M\times[0,T)$ and $r\in(0,1]$, where $T\geq\underline{T}$.
Then, for all $(x,t)\in M\times[\frac{T}{2},T)$, $r\in(0,r_{0}\sqrt{T-t}]$,
and $s\in(0,1]$ we have 
\[
|\{\widetilde{r}_{Rm}^{g}(\cdot,t)<sr\}\cap B_{g}(x,t,r)|_{g_{t}}\leq Es^{3}r^{n}.
\]
\end{prop}

\begin{proof}
Let $\epsilon_{0}>0$ be as in Theorem \ref{prop:codim2}. Define
the rescaled metric $\widetilde{g}:=(T-t)^{-1}g_{t}$, and suppose
$x\in M$ satisfies 
\[
d_{PGH}\left(\left(B_{\widetilde{g}}(x,r'),d_{\widetilde{g}},x\right),\left(B\left((0^{n-2},z_{\ast}),r'\right),d_{\mathbb{R}^{n-2}\times C(Z)},(0^{n-2},z_{\ast})\right)\right)<\epsilon_{0}^{2}r'
\]
for some $r'\in[s,1]$ and some metric cone $C(Z)$. Then 
\[
d_{PGH}\left(\left(B_{g}(x,t,\sqrt{T-t}r'),(T-t)^{-\frac{1}{2}}d_{g_{t}},x\right),\left(B\left((0^{n-2},z_{\ast}),r'\right),d_{\mathbb{R}^{n-2}\times C(Z)},(0^{n-2},z_{\ast})\right)\right)<\epsilon_{0}^{2}r',
\]
so we can apply Theorem \ref{prop:codim2} with $r=\epsilon_{0}\sqrt{T-t}r'<\epsilon_{0}\sqrt{T-t}$
to get $\widetilde{r}_{Rm}^{g}(x,t)\geq\epsilon_{0}^{2}\sqrt{T-t}r'$,
or equivalently $\widetilde{r}_{Rm}^{\widetilde{g}}(x)\geq\epsilon_{0}^{2}r'\geq\epsilon_{0}^{2}s$.
Let $\widetilde{\mathcal{S}}_{\eta,r}^{k}$ denote the quantitative
singular strata of the rescaled metric $\widetilde{g}$. Then 
\[
\{\widetilde{r}_{Rm}^{\widetilde{g}}<s\}\subseteq\widetilde{\mathcal{S}}_{\epsilon_{0}^{2},\epsilon_{0}^{-2}s}^{n-3}
\]
for all $s\in(0,\epsilon_{0}^{2}]$. We can thus apply Theorem \ref{thm:singsetsize}
with $\eta:=\epsilon_{0}^{2}$ to obtain
\begin{align*}
|\{\widetilde{r}_{Rm}^{\widetilde{g}}<sr\}\cap B_{\widetilde{g}}(x,r)|_{\widetilde{g}}\leq & |\widetilde{\mathcal{S}}_{\epsilon_{0}^{2},\epsilon_{0}^{-2}sr}^{n-3}\cap B_{\widetilde{g}}(x,r)|_{\widetilde{g}}\leq Cs^{3}r^{n}
\end{align*}
for all $s\in(0,1]$ and $r\in(0,\epsilon_{0}^{2}]$, where $C=C(A,\underline{T})<\infty$.
In terms of the unrescaled metric, this is
\[
|\{\widetilde{r}_{Rm}^{g}(\cdot,t)<sr\}\cap B_{g}(x,t,r)|_{g_{t}}\leq Cs^{3}r^{n}
\]
for all $s\in(0,1]$ and $r\in(0,\epsilon_{0}^{2}\sqrt{T-t}]$. 
\end{proof}
\begin{prop}
\label{prop:codim3} (Codimension 3 $\epsilon$-Regularity) For any
$A<\infty$ and $\underline{T}\geq0$, there exists $\epsilon_{0}=\epsilon_{0}(A,\underline{T})>0$
such that the following holds. Suppose $(M^{n},(g_{t})_{t\in[0,T)})$
is a closed Ricci flow satisfying $Rc(g_{t})\geq-Ag_{t}$ and $|B(x,t,r)|_{g_{t}}\geq A^{-1}r^{n}$,
for all $(x,t)\in M\times[0,T)$ and $r\in(0,1]$, where $T\geq\underline{T}$.
For any $(x,t)\in M\times[\frac{T}{2},T)$ and $r\in(0,\epsilon_{0}\sqrt{T-t}]$,
if 
\[
d_{PGH}\left(\left(B(x,t,r),d_{g_{t}},x\right),\left(B\left((0^{n-3},z_{\ast}),r\right),d_{\mathbb{R}^{n-3}\times C(Z)},(0^{n-3},z_{\ast})\right)\right)<\epsilon_{0}r
\]
for some metric cone $C(Z)$, then $\widetilde{r}_{Rm}(x,t)\geq\epsilon_{0}r$. 
\end{prop}

\begin{proof}
Suppose the claim is false. Then we can find closed Ricci flows $(M_{i}^{n},(g_{t}^{i})_{t\in[0,T_{i})})$
satisfying $Rc(g_{t}^{i})\geq-Ag_{t}^{i}$, $|B_{g^{i}}(x,t,r)|_{g_{t}^{i}}\geq A^{-1}r^{n}$
for all $(x,t)\in M_{i}\times[0,T_{i})$ and $r\in(0,1]$, where $T_{i}\geq\underline{T}$,
along with points $(x_{i},t_{i})\in M\times[\frac{T_{i}}{2},T_{i})$,
a sequence $\epsilon_{i}\searrow0$, and scales $r_{i}\in(0,\epsilon_{i}\sqrt{T_{i}-t_{i}})$,
such that 
\[
d_{PGH}\left(\left(B_{g^{i}}(x_{i},t_{i},\epsilon_{i}^{-1}r_{i}),d_{g_{t}^{i}},x_{i}\right),\left(B\left((0^{n-3},z_{\ast}^{i}),\epsilon_{i}^{-1}r_{i}\right),d_{\mathbb{R}^{n-3}\times C(Z_{i})},(0^{n-3},z_{\ast}^{i})\right)\right)<\epsilon_{i}r_{i}
\]
for some metric spaces $(Z_{i},d_{i})$, yet $\widetilde{r}_{Rm}^{g^{i}}(x_{i},t_{i})<\epsilon_{i}r_{i}$.
Define $\widetilde{g}_{t}^{i}:=r_{i}^{-2}g_{t_{i}+r_{i}^{2}t}^{i}$
for $t\in[-r_{i}^{-2}t_{i},r_{i}^{-2}(T-t_{i}))$, so that we can
pass to a subsequence to assume that $(M,d_{\widetilde{g}_{0}^{i}},x_{i})$
converges in the pointed Gromov-Hausdorff sense to some metric cone
\[
(\mathbb{R}^{n-3}\times C(Z_{\infty}),d_{\mathbb{R}^{n-3}\times C(Z_{\infty})},(0^{n-3},z_{\ast}))
\]
based at its vertex, where $\text{diam}(Z_{\infty})\leq\pi$. For
brevity, we write $(W,d_{W},w_{\ast}):=(\mathbb{R}^{n-3}\times C(Z_{\infty}),d_{\mathbb{R}^{n-3}\times C(Z_{\infty})},(0^{n-3},z_{\ast}))$
and $(\mathcal{R}_{W},g_{W})$ for the smooth Riemannian structure
on $\mathcal{R}(W)$. 

By Proposition \ref{prop:bigbigpoints}, there exist $E=E(A,\underline{T})<\infty$
and $r_{0}=r_{0}(A,\underline{T})>0$ such that
\[
|\{\widetilde{r}_{Rm}^{g^{i}}(\cdot,t)<sr\}\cap B_{g}(x,t,r)|_{g_{t}}\leq Es^{3}r^{n}
\]
for all $(x,t)\in M_{i}\times[\frac{T_{i}}{2},T_{i})$, $r\in(0,r_{0}\sqrt{T-t})$
$s\in(0,1]$. In terms of the rescaled metrics, this is
\begin{align}
|\{\widetilde{r}_{Rm}^{\widetilde{g}^{i}}(\cdot,0)<sr\}\cap B_{\widetilde{g}^{i}}(x,0,r)|_{\widetilde{g}_{0}^{i}}= & r_{i}^{-n}|\{\widetilde{r}_{Rm}^{g^{i}}(\cdot,t_{i})<srr_{i}\}\cap B_{g^{i}}(x,t_{i},rr_{i})|_{g_{t_{i}}^{i}}\label{eq:usefulscaleinvariant}\\
\leq & r_{i}^{-n}\cdot Es^{3}(rr_{i})^{n}=Es^{3}r^{n}\nonumber 
\end{align}
for all $s\in(0,1]$ and $r\in(0,r_{0}r_{i}^{-1}\sqrt{T-t_{i}}]$.
Since $r_{i}^{-1}\sqrt{T-t_{i}}\to\infty$, (\ref{eq:usefulscaleinvariant})
holds for all $s,r\in(0,1]$ when $i\in\mathbb{N}$ is sufficiently
large.

\noindent \textbf{Claim 1: }For any $y\in W$ and $s\in(0,\frac{1}{2}]$,
we have 
\[
|\{\widetilde{r}_{Rm}^{W}<s\}\cap B^{W}(y,1)\cap\mathcal{R}_{W}|_{g_{W}}\leq2^{10n}Es^{3}.
\]

Suppose by way of contradiction there are $y\in W$, $r,s\in(0,\frac{1}{2}]$
such that the claim fails. Because $\mathcal{R}_{W}\subseteq W$ is
dense, we can assume $y\in\mathcal{R}_{W}$. Then we can find $\sigma>0$
such that the set 
\[
S:=\{\sigma<\widetilde{r}_{Rm}^{W}<s\}\cap B^{W}(y,1)
\]
satisfies $|S|_{g_{W}}>2^{10n}Es^{3}$. By Theorem \ref{thm:openandsmooth},
there is an exhaustion $(U_{i})$ of $\mathcal{R}_{W}$ along with
diffeomorphisms $\psi_{i}:U_{i}\to M$ such that $\psi_{i}^{\ast}\widetilde{g}_{0}^{i}\to g_{W}$
in $C_{loc}^{\infty}(\mathcal{R}_{W})$ as well as 
\begin{equation}
\sup_{K\cap U_{i}}|d_{\widetilde{g}_{0}^{i}}(\psi_{i}(y_{1}),\psi_{i}(y_{2}))-d_{W}(y_{1},y_{2})|\leq\eta_{i}(K)\label{eq:distest}
\end{equation}
for each compact subset $K\subseteq\mathcal{R}_{W}$, where $\lim_{i\to\infty}\eta_{i}(K)=0$.
Because $S\subset\subset\mathcal{R}_{W}$, we have $S\subseteq U_{i}$
for sufficiently large $i\in\mathbb{N}$. Moreover, (\ref{eq:distest})
implies that $\psi_{i}(S)\subseteq B_{\widetilde{g}_{0}^{i}}(\psi_{i}(y),0,2)$
for sufficiently large $i\in\mathbb{N}$. 

\noindent \textbf{Subclaim: }$\widetilde{r}_{Rm}^{\widetilde{g}^{i}}(\psi_{i}(x),0)<4s$
for all $x\in S$ when $i\in\mathbb{N}$ is sufficiently large.

Otherwise, there are $x_{i}\in S$ such that 
\[
\sup_{B_{\widetilde{g}^{i}}(\psi_{i}(x_{i}),0,4s)}|Rm|_{\widetilde{g}_{0}^{i}}\leq\frac{1}{16}s^{-2}.
\]
Pass to a subsequence so that 
\[
x_{i}\to x\in\{\sigma\leq\widetilde{r}_{Rm}^{W}\leq s\}\cap\overline{B}^{W}(y,1)
\]
with respect to the Gromov-Hausdorff convergence. Then 
\[
(B_{\widetilde{g}^{i}}(\psi_{i}(x_{i}),0,4s),d_{\widetilde{g}_{0}^{i}},x_{i})\to(B^{W}(x,4s),d,x)
\]
in the pointed Gromov-Hausdorff sense, but the Cheeger-Gromov compactness
theorem tells us that 
\[
(B_{\widetilde{g}^{i}}(\psi_{i}(x_{i}),0,4s),\widetilde{g}_{0}^{i},x_{i})\to(\widehat{B},\widehat{g},\widehat{x})
\]
in the $C^{\infty}$ Cheeger-Gromov sense for some smooth (incomplete)
Riemannian manifold $\widehat{B}$ with curvature tensor satisfying
$|Rm|_{\widehat{g}}\leq\frac{1}{16}s^{-2}$. Moreover, 
\[
(B_{\widetilde{g}^{i}}(\psi_{i}(x_{i}),0,2s),d_{\widetilde{g}_{0}^{i}},x_{i})\to(B_{\widehat{g}}(\widehat{x},2s),d_{\widehat{g}},\widehat{x}),
\]
so $B^{W}(x,2s)$ is isometric to a ball in a smooth Riemannian manifold
equipped with its length metric, and with $|Rm|\leq\frac{1}{16}s^{-2}$.
In particular, $B^{W}(x,2s)\subseteq\mathcal{R}$ and $|Rm|_{g}\leq\frac{1}{4}s^{-2}$
on $B^{W}(x,2s)$, so that $\widehat{r}_{Rm}^{W}(x)\geq2s$, a contradiction.
$\square$

From the subclaim, we get
\[
\psi_{i}(S)\subseteq\{\widetilde{r}_{Rm}^{\widetilde{g}^{i}}(\cdot,0)<4s\}\cap B_{\widetilde{g}^{i}}(x,0,2),
\]
so that (\ref{eq:usefulscaleinvariant}) and $\psi_{i}^{\ast}\widetilde{g}_{0}^{i}\to g_{W}$
in $C_{loc}^{\infty}(\mathcal{R}_{W})$ imply
\begin{align*}
2^{10n}Es^{3}<|S|_{g_{W}}= & \lim_{i\to\infty}|\psi_{i}(S)|_{\widetilde{g}_{0}^{i}}\leq\liminf_{i\to\infty}|\{\widetilde{r}_{Rm}^{\widetilde{g}^{i}}(\cdot,0)<4s\}\cap B_{\widetilde{g}^{i}}(x,0,2)|_{\widetilde{g}_{0}^{i}}\\
\leq & E(4s)^{3}2^{n}<2^{8n}Es^{3},
\end{align*}
a contradiction. $\square$

\noindent \textbf{Claim 2: $Z_{\infty}$ }is the length space corresponding
to a smooth Riemannian manifold.

If $Z_{\infty}$ is not smooth, then $\mathcal{S}(W)$ contains a
subset isometric to $\mathbb{R}^{n-3}\times[0,\infty)$, so that the
Hausdorff dimension of $\mathcal{S}(W)$ is at least $n-2$. To get
a contradiction, it therefore suffices to prove that the Hausdorff
dimension of 
\[
S:=\mathcal{S}(W)\cap B^{W}(y,1)
\]
is at most $n-3$ for any $y\in W$. The proof of this given Claim
1 is standard, but we include it for completeness.

Given $s\in(0,1]$, let $\{y_{1},...,y_{N}\}$ be a maximal subset
of $S$ such that $d_{W}(y_{i},y_{j})\geq\frac{s}{2}$ for any distinct
$i,j\in\{1,...,N\}$. Then $S\subseteq\cup_{j=1}^{N}B^{W}(y_{j},s)$,
and 
\[
\cup_{j=1}^{N}B^{W}(y_{j},s)\subseteq\{\widetilde{r}_{Rm}^{W}<s\}\cap B^{W}(y,2)
\]
since $\widetilde{r}_{Rm}^{W}$ is $1$-Lipschitz. Thus
\begin{align*}
NA^{-1}s^{n}\leq & \sum_{j=1}^{N}\mathcal{H}^{n}(B^{W}(y_{j},s))\leq\mathcal{H}^{n}(\{\widetilde{r}_{Rm}^{W}<s\}\cap B^{W}(y,2))\\
= & |\{\widetilde{r}_{Rm}^{W}<s\}\cap B^{W}(y,2)\cap\mathcal{R}_{W}|_{g_{W}}\leq2^{14n}Es^{3},
\end{align*}
which implies that $N\leq2^{14}AEs^{3-n}$, hence 
\[
\mathcal{H}_{s}^{n-3}(S)\leq CNs^{n-3}\leq2^{14n}CAE.
\]
Taking $s\searrow0$ gives $\mathcal{H}^{n-3}(S)<\infty$. $\square$

In particular, $\mathbb{R}^{n-3}\times(C(Z_{\infty})\setminus\{z_{\ast}\})$
is a smooth 3-dimensional Riemannian cone with nonnegative Ricci curvature.
By the same argument as Claim 3 of Theorem 1, we know that $Z_{\infty}$
has constant curvature $1$, hence that $Z_{\infty}$ is the round
sphere or the round $\mathbb{R}P^{2}$, but because the pointed Gromov-Hausdorff
convergence 
\[
(M,d_{\widetilde{g}_{0}^{i}},x_{i})\to(\mathbb{R}^{n-3}\times C(Z_{\infty}),d_{\mathbb{R}^{n-3}\times C(Z_{\infty})},(0^{n-3},z_{\ast}))
\]
is smooth away from $\mathbb{R}^{n-3}\times\{z_{\ast}\}$, and because
$M$ is orientable, we must have $Z_{\infty}\cong\mathbb{S}^{2}$
(otherwise, there is an embedding $B(0^{n-3},1)\times\mathbb{R}P^{2}\times(1,2)\hookrightarrow M$,
which is impossible). That is, $(M,d_{\widetilde{g}_{0}^{i}},x_{i})$
converges to flat $\mathbb{R}^{n}$, so $r_{Rm}^{\widetilde{g}_{0}^{i}}(x_{i},0)\to\infty$,
contradicting $r_{Rm}^{\widetilde{g}_{0}^{i}}(x_{i},0)<\epsilon_{i}\to0$. 
\end{proof}
\begin{proof}[Proof of Theorem \ref{thm:highdimlp}]
 The proof of $(i)$ is a trivial modification of the proof of Proposition
\ref{prop:bigbigpoints}, where we use codimension three $\epsilon$-regularity
(Proposition \ref{prop:codim3}) instead of codimension two (that
is, we replace $\mathcal{S}_{\epsilon,r}^{n-3}$ with $\mathcal{S}_{\epsilon,r}^{n-4}$).

$(ii)$ Replacing $r$ with $r_{0}\sqrt{T-t}$ and $s$ with $\frac{s^{-1}}{r_{0}\sqrt{T-t}}$
in $(i)$, we can estimate
\begin{align*}
\int_{B(x,t,r_{0}\sqrt{T-t})}\widetilde{r}_{Rm}^{-p}(y,t)dg_{t}(y)\leq & p\int_{r_{0}^{-\frac{1}{2}}(T-t)^{-\frac{1}{2}}}^{\infty}s^{p-1}|\{\widetilde{r}_{Rm}^{-1}(\cdot,t)>s\}\cap B(x,t,r_{0}\sqrt{T-t})|_{g_{t}}ds\\
 & +\left(r_{0}^{-\frac{1}{2}}(T-t)^{-\frac{1}{2}}\right)^{p}|B(x,t,r_{0}\sqrt{T-t})|_{g_{t}}\\
\leq & pE\int_{r_{0}^{-\frac{1}{2}}(T-t)^{-\frac{1}{2}}}^{\infty}s^{p-1}\left(\frac{s^{-1}}{r_{0}\sqrt{T-t}}\right)^{4}\left(r_{0}\sqrt{T-t}\right)^{n}ds\\
 & +C(A,\underline{T})(T-t)^{\frac{n-p}{2}}\\
\leq & C(A,\underline{T},p)(T-t)^{\frac{n-p}{2}}.
\end{align*}
The claim follows from a standard covering argument.
\end{proof}

\section{Curvature Scale Decomposition in Dimension 4}

In this section, we again specialize to dimension four, where we decompose
each time slice of a Ricci flow satisfying (\ref{eq:ric}),(\ref{eq:vol})
according its curvature scale relative to the Type-I scale. The region
where $\widetilde{r}_{Rm}(\cdot,t)<<\sqrt{T-t}$ was estimated using
codimension three $\epsilon$-regularity. The region where $\widetilde{r}_{Rm}(\cdot,t)>>\sqrt{T-t}$
is estimated using the following proposition, and the intermediate
region where $\widetilde{r}_{Rm}(\cdot,t)\approx\sqrt{T-t}$ is dealt
with in the proof of Theorem \ref{thm:theorem3}.
\begin{prop}
\label{prop:smallbigpoints} Suppose $(M^{4},(g_{t})_{t\in[0,T)})$
is a closed, simply connected Ricci flow satisfying (\ref{eq:ric}),(\ref{eq:vol}).
Then there exist $C=C(A,T)<\infty$, and $E'=E'(X)<\infty$ such that
the following hold:

$(i)$ For any $(x,t)\in M\times[\frac{T}{2},T)$ and $s\in[\epsilon_{P}^{-1}\sqrt{T-t},1]$,
we have
\[
|\{\epsilon_{P}^{-1}\sqrt{T-t}\leq\widetilde{r}_{Rm}^{g}(\cdot,t)<s\}|_{g_{t}}\leq C(A,T)|\{\widetilde{r}_{Rm}^{X}<\epsilon_{P}^{-1}s\}\cap\mathcal{R}_{X}|_{g_{X}},
\]

$(ii)$ For any $s\in(0,1]$, we have 
\[
|\{\widetilde{r}_{Rm}^{X}<s\}\cap\mathcal{R}_{X}|_{g_{X}}\leq E's^{4}.
\]
\end{prop}

\begin{proof}
$(i)$ By Theorem \ref{thm:pseudolocality},
\[
\{\epsilon_{P}^{-1}\sqrt{T-t}<\widetilde{r}_{Rm}^{g}(\cdot,t)<s\}\subseteq\{x\in M\setminus\Sigma;\widetilde{r}_{Rm}^{g}(x,T)<\epsilon_{P}^{-1}s\}.
\]
Moreover, for any $x\in M$ with $\widetilde{r}_{Rm}^{g}(x,t)\geq\epsilon_{P}^{-1}(T-t)$,
we have $|Rm|(x,\tau)\leq\frac{1}{T-\tau}$ for all $\tau\in[t,T]$,
so we can integrate $\partial_{\tau}dg_{\tau}|_{x}\geq c(T-t)^{-1}dg_{\tau}|_{x}$
from $\tau=t$ to $\tau=T$ to obtain $dg_{T}|_{x}\geq cdg_{t}|_{x}$.
Thus
\begin{align*}
|\{\epsilon_{P}^{-1}\sqrt{T-t}<\widetilde{r}_{Rm}^{g}(\cdot,t)<s\}|_{g_{t}}\leq & C|\{x\in M\setminus\Sigma;\widetilde{r}_{Rm}^{g}(x,T)<\epsilon_{P}^{-1}s\}|_{g_{T}}\\
= & C|\{\overline{x}\in\mathcal{R}_{X};\widetilde{r}_{Rm}^{X}(\overline{x})<\epsilon_{P}^{-1}s\}|_{g_{X}}.
\end{align*}
$(ii)$ Let $\overline{x}_{1},...,\overline{x}_{N}\in X$ be the singular
points. Fix $r_{0}>0$ such that $d(\overline{x}_{i},\overline{x}_{j})>2r_{0}$
for distinct $i,j\in\{1,...,N\}$. Because $X\setminus\{\overline{x}_{1},...,\overline{x}_{N}\}$
is smooth, we can find $\sigma_{0}>0$ such that $\{\widetilde{r}_{Rm}^{X}<\sigma_{0}\}\subseteq\bigcup_{j=1}^{N}B^{X}(\overline{x}_{j},r_{0})$.
Moreover, by possibly shrinking $r_{0}$ (and $\sigma_{0}$ accordingly),
we can assume $\widetilde{r}_{Rm}^{X}\geq c_{0}d(\overline{x}_{j},\cdot)$
on $B^{X}(\overline{x}_{j},r_{0})$, for $j=1,...,N$, and some constant
$c_{0}>0$. When $s\geq\sigma_{0}$, we can estimate
\[
|\{\widetilde{r}_{Rm}^{X}<s\}\cap\mathcal{R}_{X}|_{g_{X}}\leq|\mathcal{R}_{X}|_{g_{X}}\leq C\sigma^{-4}s^{4},
\]
so it suffices to consider the case where $s<\sigma_{0}$. Then 
\begin{align*}
|\{\widetilde{r}_{Rm}^{X}<s\}\cap\mathcal{R}_{X}|_{g_{X}}\leq & \sum_{j=1}^{N}|B^{X}(\overline{x}_{j},c_{0}^{-1}s)\cap\mathcal{R}_{X}|_{g_{X}}\\
\leq & NC(A,T)c_{0}^{-4}s^{4}.
\end{align*}
\end{proof}
\begin{proof}[Proof of Theorem \ref{thm:theorem3}]
 Since $(M^{4},(g_{t})_{t\in[0,T-\tau]})$ is smooth, it suffices
to prove this for $t\in[T-\tau,T)$. Let $\overline{x}_{1},...,\overline{x}_{N}\in X$
be the singular points of $X$, and choose representative points $x_{1},...,x_{N}\in M$. 

\noindent \textbf{Claim 1: }For any $\alpha\in[1,\infty)$, there
exists $\delta=\delta(\alpha)>0$ such that
\[
\{\widetilde{r}_{Rm}^{g}(\cdot,t)<\alpha\sqrt{T-t}\}\subseteq\cup_{j=1}^{N}B_{g}(x_{j},t,\delta^{-1}\sqrt{T-t})
\]
for all $t\in[T-\delta,T)$. 

Assume not, so that we can find sequences $\delta_{i}\searrow0$,
$t_{i}\in[T-\delta_{i},T)$, and points 
\[
y_{i}\in M\setminus\cup_{j=1}^{N}B_{g}(x_{j},t,\delta_{i}^{-1}\sqrt{T-t})
\]
satisfying 
\[
\frac{\widetilde{r}_{Rm}^{g}(y_{i},t_{i})}{\sqrt{T-t_{i}}}<\alpha.
\]
Pass to a subsequence so that $y_{i}\to y$ in $M$. Because $\widetilde{r}_{Rm}^{g}(y_{i},t_{i})\to0$,
we must have $y\in\Sigma$, hence $\overline{y}\in\{\overline{x}_{1},...,\overline{x}_{N}\}$.
By Claim 2 in the proof of Theorem \ref{thm:theorem1}, we get $\limsup_{t\nearrow T}d_{g_{t}}(y,x)(T-t)^{-\frac{1}{2}}<\infty$
for some $x\in\{x_{1},...,x_{N}\}$. By the choice of $y_{i}$, this
implies
\[
\liminf_{i\to\infty}\frac{d_{g_{t_{i}}}(y,y_{i})}{\sqrt{T-t_{i}}}\geq\liminf_{i\to\infty}\left(\frac{d_{g_{t_{i}}}(x,y_{i})-d_{g_{t_{i}}}(x,y)}{\sqrt{T-t_{i}}}\right)=\infty.
\]
Let $E\geq2$ be arbitrary. Let $\gamma:[0,1]\to M$ be any curve
from $y$ to $y_{i}$. By Claim 1 of Theorem \ref{thm:theorem1},
there exist $\delta=\delta(A,E)>0$ and $r=r(A)<\infty$ such that
for $t\in(T-\delta,T)$ and $y'\in B_{g}(y,t,2Er\sqrt{T-t})\setminus\overline{B}_{g}(y,t,\frac{1}{2}r\sqrt{T-t})$,
we have

\[
|Rm|(y',s)\leq\frac{1}{T-t}
\]
for all $s\in(T-\delta,T)$. Similarly to the proof of Theorem \ref{thm:theorem1},
we define

\[
u_{-}^{i}:=\sup\left\{ u\in[0,1];d_{g_{t_{i}}}(\gamma(u),y)=r\sqrt{T-t_{i}}\right\} ,
\]
\[
u_{+}^{i}:=\inf\left\{ u\in[0,1];d_{g_{t_{i}}}(\gamma(u),y)=(E+1)r\sqrt{T-t_{i}}\right\} ,
\]
so that $\text{length}_{g_{t_{i}}}(\gamma|[u_{-}^{i},u_{+}^{i}])\geq Er\sqrt{T-t_{i}}$
for $i\in\mathbb{N}$ sufficiently large (independently of $\gamma$).
Because $|Rm|_{g}(\gamma(u),s)\leq\frac{1}{T-t_{i}}$ for all $u\in[u_{-}^{i},u_{+}^{i}]$
and $s\in[T-t_{i},T)$, we get
\[
\text{length}_{g_{s}}(\gamma|[u_{-}^{i},u_{+}^{i}])\geq\frac{r}{e^{4}}E\sqrt{T-t_{i}},
\]
so letting $s\nearrow T$, then taking the infimum over all curves
$\gamma$ gives 
\[
d_{X}(\overline{y},\overline{y}_{i})\geq\frac{r}{e^{4}}E\sqrt{T-t_{i}},
\]
hence
\[
\liminf_{i\to\infty}(T-t_{i})^{-\frac{1}{2}}d_{X}(\overline{y},\overline{y}_{i})\geq\frac{r}{e^{4}}E.
\]
However, $E<\infty$ was arbitrary, so we have
\[
\lim_{i\to\infty}(T-t_{i})^{-\frac{1}{2}}d_{X}(\overline{y},\overline{y}_{i})=\infty.
\]
In particular, we can use $\liminf_{\overline{z}\to\overline{y}}d_{X}^{-1}(\overline{z},\overline{y})\widetilde{r}_{Rm}^{X}(\overline{z})>0$
to get
\[
(T-t_{i})^{-\frac{1}{2}}\widetilde{r}_{Rm}^{X}(\overline{y}_{i})=\frac{d_{X}(\overline{y},\overline{y}_{i})}{\sqrt{T-t_{i}}}\frac{\widetilde{r}_{Rm}^{X}(\overline{y}_{i})}{d_{X}(\overline{y},\overline{y}_{i})}\to\infty
\]
as $i\to\infty$. By Theorem \ref{thm:pseudolocality}, this implies
\[
\lim_{i\to\infty}\frac{\widetilde{r}_{Rm}^{g}(y_{i},t_{i})}{\sqrt{T-t_{i}}}=\infty,
\]
a contradiction. $\square$

We now apply Claim 1 with $\alpha=\epsilon_{P}^{-1}$ as in Theorem
\ref{thm:pseudolocality}. Then a standard covering argument on each
ball $B(x_{j},t,\delta^{-1}\sqrt{T-t})$ using volume doubling gives,
for any $t\in(T-\delta,T)$, some $N_{1}=N_{1}((g_{t})_{t\in[0,T)})\in\mathbb{N}$
(independent of $t$) along with points $z_{t}^{1},...,z_{t}^{N_{1}}$
such that 
\[
\{\widetilde{r}_{Rm}^{g}(\cdot,t)<\epsilon_{P}^{-1}\sqrt{T-t}\}\subseteq\bigcup_{j=1}^{N_{1}}B_{g}(z_{t}^{j},t,r_{0}\sqrt{T-t}),
\]
where $r_{0}$ is as in Theorem \ref{thm:highdimlp}. 

For any $s\in(0,r_{0}\sqrt{T-t}]$, we may thus apply Theorem \ref{thm:highdimlp}
to get 
\begin{align*}
|\{\widetilde{r}_{Rm}^{g}(\cdot,t)<s\}|_{g_{t}}\leq & \sum_{j=1}^{N_{1}}|\{\widetilde{r}_{Rm}^{g}(\cdot,t)<s\}\cap B_{g}(z_{t}^{j},t,r_{0}\sqrt{T-t})|_{g_{t}}\\
\leq & N_{1}E\left(\frac{s}{r_{0}\sqrt{T-t}}\right)^{4}\left(r_{0}\sqrt{T-t}\right)^{4}\\
\leq & N_{1}Es^{4}.
\end{align*}
If $s\in[r_{0}\sqrt{T-t},\epsilon_{P}^{-1}\sqrt{T-t}]$, then Claim
1 gives
\begin{align*}
|\{\widetilde{r}_{Rm}^{g}(\cdot,t)<s\}|_{g_{t}}\leq & \sum_{j=1}^{N}|B_{g}(x_{j},t,\delta^{-1}\sqrt{T-t})|_{g_{t}}\\
\leq & C(A)\delta^{-4}N(T-t)^{2}\leq C(A)\delta^{-4}Nr_{0}^{-4}s^{4}.
\end{align*}
If $s\in[\epsilon_{P}^{-1}\sqrt{T-t},\infty)$, use Proposition \ref{prop:smallbigpoints}
and Claim 1 to get

\begin{align*}
|\{\widetilde{r}_{Rm}^{g}(\cdot,t)<s\}|_{g_{t}}\leq & |\{\epsilon_{P}^{-1}\sqrt{T-t}\leq\widetilde{r}_{Rm}^{g}(\cdot,t)<s\}|_{g_{t}}\\
 & +|\{\widetilde{r}_{Rm}^{g}(\cdot,t)<\epsilon_{P}^{-1}\sqrt{T-t}\}|_{g_{t}}\\
\leq & C(A)\delta^{-4}Nr_{0}^{-4}\epsilon_{P}^{-4}(T-t)^{2}+E's^{4}\\
\leq & C((g_{t})_{t\in[0,T)})s^{4}.
\end{align*}
\end{proof}

\section{Appendix: Gaussian Upper Bound for the Heat Kernel}

The following is a modification of the proof of Theorem 3.1 in \cite{cao},
which now applies to manifolds with possibly negative lower Ricci
curvature bounds. This estimate is similar to Theorem 1.4 of \cite{daviesricciflow},
but importantly does not depend on a bound for $\int_{0}^{T}\sup_{M}|Rc|(\cdot,t)dt$. 
\begin{prop}
\label{prop:heatkernel} Let $(M^{n},(g_{t})_{t\in[0,T)})$ be a solution
to the Ricci flow satisfying $Rc(g_{t})\geq-Ag_{t}$ and $|B(x,t,r)|_{g_{t}}\geq A^{-1}r^{n}$
for all $(x,t)\in M\times[0,T)$ and $r\in(0,1]$. Then there exists
$C=C(n,A,T,\text{diam}_{g_{0}}(M))<\infty$ such that
\[
K(x,t;y,s)\leq\dfrac{C}{(t-s)^{\frac{n}{2}}}\exp\left(-\dfrac{d_{g_{t}}^{2}(x,y)}{C(t-s)}\right)
\]
for all $x,y\in M$ and $0\leq s<t<T$.
\end{prop}

\begin{proof}
First suppose $u\in C^{\infty}(M\times[t_{2},t_{1}])$ is a positive
solution of $\partial_{t}u=\Delta_{g_{t}}u$. Fix $\xi\in C^{\infty}(M\times[t_{2},t_{1}])$
to be determined. Then 
\begin{align*}
\dfrac{d}{dt}\int_{M}u^{2}e^{\xi}dg_{t}= & \int_{M}(2u\Delta_{g_{t}}u+u^{2}\partial_{t}\xi-Ru^{2})e^{\xi}dg_{t}\\
= & \int_{M}\left(-2|\nabla u|_{g(t)}^{2}-2u\langle\nabla u,\nabla\xi\rangle+u^{2}\partial_{t}\xi-Ru^{2}\right)e^{\xi}dg_{t}\\
\leq & \int_{M}\left(\frac{1}{2}|\nabla\xi|_{g(t)}^{2}+\partial_{t}\xi\right)u^{2}e^{\xi}dg_{t}+nA\int_{M}u^{2}e^{\xi}dg_{t}.
\end{align*}
Fix $x\in M$, and define 
\[
I_{r}(t):=\int_{M\setminus B(x,t,r)}u^{2}(y,t)dg_{t}(y).
\]
Fix $t_{0}\in(t_{1},T)$, and define 
\[
\xi(y,t):=-\dfrac{(r-d_{g_{t}}(x,y))_{+}^{2}}{8(t_{0}-t)},
\]
for $(y,t)\in M\times[t_{2},t_{0})$. For any $t\in[t_{2},t_{1}]$
where $t\mapsto d_{g_{t}}(x,y)$ is differentiable,
\begin{align*}
\partial_{t}d_{g_{t}}(x,y)= & \frac{d}{dt}\left(\inf_{\gamma}\int_{0}^{l}|\dot{\gamma}(\tau)|_{g_{t}}d\tau\right)=\inf_{\gamma}\left(-\int_{0}^{l}\frac{Rc(\dot{\gamma}(\tau),\dot{\gamma}(\tau))}{|\dot{\gamma}(\tau)|}d\tau\right)\\
\leq & \inf_{\gamma}\left(A\int_{0}^{l}|\dot{\gamma}(\tau)|_{g_{t}}d\tau\right)=Ad_{g_{t}}(x,y),
\end{align*}
where the infimum is taken over all $g_{t}$-minimizing geodesics
$\gamma:[0,l]\to M$ from $x$ to $y$. We can thus estimate (when
$d_{g_{t}}(x,y)\leq r$, otherwise everything is zero)
\begin{align*}
\partial_{t}\xi+\frac{1}{2}|\nabla\xi|_{g_{t}}^{2}= & -\dfrac{(r-d_{g_{t}}(x,y))^{2}}{8(t_{0}-t)^{2}}+\dfrac{(r-d_{g_{t}}(x,y))}{4(t_{0}-t)}\partial_{t}d_{g_{t}}(x,y)+\dfrac{(r-d_{g_{t}}(x,y))^{2}}{32(t_{0}-t)^{2}}\\
\leq & -\dfrac{(r-d_{g_{t}}(x,y))^{2}}{16(t_{0}-t)^{2}}+Ar\dfrac{(r-d_{g_{t}}(x,y))}{4(t_{0}-t)}\\
\leq & A^{2}r^{2}.
\end{align*}
This implies
\[
\dfrac{d}{dt}\left(e^{-(nA+A^{2}r^{2})t}\int_{M}u^{2}e^{\xi}dg_{t}\right)\leq0.
\]
For $r>\rho>0$ and $0<t_{2}<t_{1}<t_{0}$, we can integrate to obtain

\begin{align*}
I_{r}(t_{1})\leq & \int_{M}u^{2}(y,t_{1})e^{\xi(y,t_{1})}dg_{t_{1}}(y)\\
\leq & e^{(nA+A^{2}r^{2})(t_{1}-t_{2})}\int_{M}u^{2}(y,t_{2})e^{\xi(y,t_{2})}dg_{t_{2}}(y)\\
\leq & e^{(nA+A^{2}r^{2})(t_{1}-t_{2})}\left(I_{\rho}(t_{2})+\int_{B(x,t_{2},\rho)}u^{2}(y,t_{2})e^{\xi(y,t_{2})}dg_{t_{2}}(y)\right)\\
\leq & e^{(nA+A^{2}r^{2})(t_{1}-t_{2})}\left(I_{\rho}(t_{2})+\exp\left(-\dfrac{(r-\rho)^{2}}{8(t_{0}-t_{2})}\right)\int_{B(x,t_{2},\rho)}u^{2}(y,t_{2})dg_{t_{2}}(y)\right).
\end{align*}
By Lemma \ref{lem:basiclemma}, Theorem 7.1 of \cite{bamlergen1}
gives $B^{\ast}=B^{\ast}(A,T,D)<\infty$ such that
\[
K(y,t;y,s)\leq\dfrac{B^{\ast}}{(t-s)^{\frac{n}{2}}}
\]
for all $y\in M$ and $0\leq s<t<T$, where $D:=\text{diam}_{g_{0}}(M)$.
From
\[
\frac{d}{dt}\int_{M}udg_{t}=-\int_{M}Rudg_{t}\leq nA\int_{M}udg_{t},
\]
we get $\int_{M}udg_{t}\leq e^{nAT}\int_{M}udg_{s}$, so by taking
$u(y,t):=K(y,t;x,s)$, we obtain
\[
\int_{B(x,t_{2},\rho)}K^{2}(y,t_{2};x,s)dg_{t_{2}}(y)\leq\dfrac{B^{\ast}}{(t_{2}-s)^{\frac{n}{2}}}\int_{B(x,t_{2},\rho)}u(y,t_{2})dg_{t_{2}}(y)\leq\dfrac{B}{(t_{2}-s)^{\frac{n}{2}}},
\]
where $B=B(A,T,D)<\infty$. Combining estimates, and then taking $t_{0}\searrow t_{1}$
gives
\[
I_{r}(t_{1})\leq e^{(nA+A^{2}r^{2})(t_{1}-t_{2})}\left(I_{\rho}(t_{2})+\dfrac{B}{(t_{2}-s)^{\frac{n}{2}}}\exp\left(-\dfrac{(r-\rho)^{2}}{8(t_{1}-t_{2})}\right)\right).
\]
Now fix $r>0$, $0\leq s<t\leq T$ and suppose $(r_{k}),(t_{k})$
are decreasing sequences with $r_{0}=r$, $t_{0}=t$, $t_{k}\searrow s$.
Then
\[
I_{r_{k}}(t_{k})\leq e^{(nA+A^{2}r_{k}^{2})(t_{k}-t_{k+1})}\left(I_{r_{k+1}}(t_{k+1})+\frac{B}{(t_{k+1}-s)^{\frac{n}{2}}}\exp\left(-\dfrac{(r_{k}-r_{k+1})^{2}}{8(t_{k}-t_{k+1})}\right)\right),
\]
so iterating and using $I_{r_{k}}(t_{k})\to0$ (since $K(\cdot,t;x,s)\to0$
locally uniformly on $M\setminus\{x\}$ as $t\searrow s$) gives
\begin{align*}
I_{r}(t)= & I_{r_{0}}(t_{0})\leq\sum_{k=0}^{\infty}\exp\left(\sum_{j=0}^{k}(nA+A^{2}r_{j}^{2})(t_{j}-t_{j+1})\right)\frac{B}{(t_{k+1}-s)^{\frac{n}{2}}}\exp\left(-\dfrac{(r_{k}-r_{k+1})^{2}}{8(t_{k}-t_{k+1})}\right)\\
\leq & c(A,T,D)e^{A^{2}Tr^{2}}\sum_{k=0}^{\infty}\frac{1}{(t_{k+1}-s)^{\frac{n}{2}}}\exp\left(-\frac{(r_{k}-r_{k+1})^{2}}{8(t_{k}-t_{k+1})}\right).
\end{align*}
We now choose the sequences of radii and times:
\[
r_{k}:=\left(\frac{1}{2}+\frac{1}{k+2}\right)r,\hfill t_{k}:=s+2^{-k}(t-s),
\]
so that
\[
r_{k}-r_{k+1}\geq\frac{r}{(k+3)^{2}},\qquad t_{k}-t_{k+1}=2^{-k-1}(t-s),
\]
hence we have
\[
\frac{(r_{k}-r_{k+1})^{2}}{8(t_{k}-t_{k+1})}\geq\frac{2^{k-2}}{(k+3)^{4}}\frac{r^{2}}{t-s}\geq\frac{\gamma2^{\frac{k}{2}+1}r^{2}}{t-s}
\]
for all $k\in\mathbb{N}$, where $\gamma>0$ is universal. Assuming
$r^{2}\geq t-s$, we therefore have
\begin{align*}
I_{r}(t)\leq & \dfrac{C(A,T,D)}{(t-s)^{\frac{n}{2}}}\exp\left(A^{2}Tr^{2}-\dfrac{\gamma r^{2}}{t-s}\right)\sum_{k=0}^{\infty}2^{nk}\exp\left(-\gamma2^{\frac{k}{2}}\right)\\
\leq & \dfrac{C(A,T,D)}{(t-s)^{\frac{n}{2}}}\exp\left(A^{2}Tr^{2}-\dfrac{\gamma r^{2}}{t-s}\right).
\end{align*}
Now set $\eta:=\frac{\gamma}{2A^{2}T}$. If $|t-s|<\eta$, and if
$r:=d_{g_{t}}(x,y)$ satisfies $r^{2}\geq4(t-s)$, then 
\begin{align*}
\int_{B(y,t,\sqrt{t-s})}K^{2}(z,t;x,s)dg_{t}(z)\leq & \int_{M\setminus B(x,t,\frac{r}{2})}K^{2}(z,t;x,s)dg_{t}(z)\\
\leq & \dfrac{C(A,T,D)}{(t-s)^{\frac{n}{2}}}\exp\left(-\frac{\gamma}{8}\dfrac{d_{g_{t}}^{2}(x,y)}{t-s}\right),
\end{align*}
so there exists $z_{0}\in B(y,t,\sqrt{t-s})$ such that
\[
K^{2}(z_{0},t;x,s)|B(y,t,\sqrt{t-s})|_{g_{t}}\leq\dfrac{C(A,T,D)}{(t-s)^{\frac{n}{2}}}\exp\left(-\frac{\gamma}{8}\dfrac{d_{g_{t}}^{2}(x,y)}{t-s}\right),
\]
By the volume lower bound and Q. Zhang's gradient estimate (Theorem
3.3 of \cite{qizhanggradient}) we therefore conclude
\[
K(y,t;x,s)\leq\dfrac{C(A,T,D)}{(t-s)^{\frac{n}{2}}}\exp\left(-\frac{\gamma}{16}\dfrac{d_{g_{t}}^{2}(x,y)}{t-s}\right)
\]
for all $x,y\in M$ and $0\leq s<t\leq T$ with $|t-s|<\eta$ and
$d_{g_{t}}^{2}(x,y)\geq4(t-s)$. If instead $d_{g_{t}}^{2}(x,y)\leq4(t-s)$,
then the on-diagonal upper bound gives
\[
K(y,t;x,s)\leq\frac{C(A,T,D)}{(t-s)^{\frac{n}{2}}}\leq\dfrac{C(A,T,D)}{(t-s)^{\frac{n}{2}}}\exp\left(-\frac{\gamma}{4}\dfrac{d_{g_{t}}^{2}(x,y)}{t-s}\right),
\]
so the estimate holds in this case as well. Finally, from $\text{diam}_{g_{t}}(M)\leq C(A,T,D)$,
if $|t-s|\geq\eta,$then 
\begin{align*}
K(y,t;x,s)\leq & \frac{C(A,T,D)}{(t-s)^{\frac{n}{2}}}\leq\dfrac{C(A,T,D)}{(t-s)^{\frac{n}{2}}}\exp\left(-\frac{\gamma}{4}\dfrac{d_{g_{t}}^{2}(x,y)}{t-s}\right)\exp\left(\frac{\gamma}{4}\frac{D^{2}}{\eta}\right)\\
\leq & \frac{C(A,T,D)}{(t-s)^{\frac{n}{2}}}\exp\left(-\frac{\gamma}{4}\dfrac{d_{g_{t}}^{2}(x,y)}{t-s}\right).
\end{align*}
\end{proof}
\printbibliography

\end{document}